\documentclass[11pt, leqno]{amsart}
\usepackage{amsfonts,amssymb}

\usepackage{amsmath}
\usepackage{amsthm}
\usepackage{amsrefs}
\usepackage{qsymbols}
\usepackage{latexsym}
\usepackage{chngcntr}

\usepackage{paralist}

\usepackage{mathtools}

\usepackage{esint}

\usepackage[hidelinks]{hyperref}

\usepackage{tikz-cd}
\usepackage{import}

\usepackage{csquotes}
\usepackage[dvipsnames]{xcolor}
\usepackage{todonotes}
\usepackage{comment}
\usepackage{enumitem}

\newtheorem{theorem}{Theorem}[section]

\newtheorem{lemma}[theorem]{Lemma}
\newtheorem{proposition}[theorem]{Proposition}

\theoremstyle{definition}
\newtheorem{definition}[theorem]{Definition}
\newtheorem{remark}[theorem]{Remark}

\newcommand{\IR}{\mathbb{R}}
\newcommand{\IC}{\mathbb{C}}
\newcommand{\IN}{\mathbb{N}}
\newcommand{\IZ}{\mathbb{Z}}

\newcommand{\cM}{\mathcal{M}}

\newcommand{\cF}{\mathcal{F}}
\newcommand{\cS}{\mathcal{S}}
\newcommand{\cP}{\mathcal{P}}

\renewcommand{\L}{\mathrm{L}}

\newcommand{\B}{\mathrm{B}}

\newcommand{\F}{\mathrm{F}}

\newcommand{\G}{\mathrm{G}}

\newcommand{\T}{\mathrm{T}}
\newcommand{\Z}{\mathrm{Z}}

\newcommand{\D}{{\mathrm{D}}}

\renewcommand{\d}{\mathrm{d}}

\newcommand{\loc}{\mathrm{loc}}

\newcommand{\esssup}{\mathrm{ess\, sup}}

\DeclareMathOperator{\supp}{supp}

\DeclareMathOperator{\Rg}{\mathcal{R}}

\numberwithin{equation}{section}

\title{A coherent theory of tent spaces and homogeneous Triebel--Lizorkin spaces}

\author{Luca Haardt}

\address{Karlsruhe Institute of Technology, Department of Mathematics, 76131 Karlsruhe, Germany}
\email{luca.haardt@kit.edu}

\date{\today}

\begin{document}
\begin{abstract}
    We introduce and systematically investigate a scale of tent spaces that characterizes homogeneous Triebel--Lizorkin spaces $\dot \F^{\beta}_{p,q}$. These spaces generalize the classical spaces of Coifman, Meyer, and Stein, and are shown to be equivalent to the weighted tent spaces with Whitney averages developed by Huang. We show that these tent spaces follow a functional analytic theory that mirrors that of Triebel--Lizorkin spaces, including duality, embeddings, discrete characterizations, John--Nirenberg-type properties, as well as real and complex interpolation. Furthermore, we provide a novel characterization of the endpoint spaces $\dot\F^\beta_{\infty,q}$, completing earlier work by Auscher, Bechtel, and the author.
\end{abstract}
\maketitle

\allowdisplaybreaks
\maketitle

\section{Introduction}
\label{sec: introduction}
$ $
In \cite{Auscher_Bechtel_Haardt}, the authors introduced new scales of function spaces that naturally arise in the characterization of homogeneous Besov and Triebel--Lizorkin spaces.
Characterizations of homogeneous Besov spaces $\dot \B^\beta_{p,q}$ and Triebel--Lizorkin spaces $\dot \F^\beta_{p,q}$ have attracted significant interest in recent decades due to their instrumentality in areas such as harmonic analysis, partial differential equations, and more.
Traditionally, these spaces are defined via discrete Littlewood--Paley blocks $(\Delta_k)_{k\in\IZ}$ that are, in fact, convolution operators with a kernel compactly supported away from zero in Fourier-space, see the monograph of Triebel \cite[Sec.\@ 5]{Triebel1}. In contrast, for applications to partial differential equations, it is of interest to have characterizations with respect to a continuous spectrum and using kernels without compactness of the Fourier-support. For example, kernels whose Fourier-transform vanishes sufficiently fast in zero and infinity, such as
\begin{align*}
    k_t(x) = \cF^{-1}\big(|t\xi|^N e^{-|t\xi|^2}\big)(x), \quad \text{for $t>0$ and some $N\in\IN$,}
\end{align*}
are of particular interest, because they lead to characterizations via the well-known Gauss--Weierstrass semigroup.\\

First continuous characterizations of Besov and Triebel--Lizorkin spaces for general kernels of this kind were given by Triebel in \cite[Sec.\@ 2.4.2 + 2.5.1]{Triebel2} in the range $1<p<\infty$, $1<q\leq \infty$ and $\beta\in\IR$ ($p=\infty$ is included for Besov spaces). They were later complemented by Ullrich \cite{Ullrich1} and Hui--Taibleson \cites{Bui_Taibleson} for the full range of parameters $0<p\leq\infty$ and $0<q\leq \infty$. Interestingly, these articles provide several characterizations for Triebel--Lizorkin spaces, which always have a corresponding counterpart for Besov spaces, except for one. More precisely, this additional characterization for Triebel--Lizorkin spaces uses (weighted) tent spaces $\T^{p,q}_\beta$ on $\IR^{d+1}_+$. They were first introduced by Coifman, Meyer and Stein \cite{Coifman_Meyer_Stein} in the unweighted case $\beta=0$, and later studied by many other authors in the weighted case $\beta\in \IR$, see for instance \cites{Amenta2, Hofmann_Mayboroda_McIntosh}. 
Roughly speaking, the tent space characterization reads
\begin{align*}
    f\in \dot \F^\beta_{p,q} \quad \Longleftrightarrow \quad     (\Phi_t\ast f)(x) \in \T^{p,q}_\beta,
\end{align*}
where $\Phi$ is a suitable kernel and $\Phi_t = t^{-d}\Phi(\cdot/t)$, see Section~\ref{sec: characterizations} for more details. In this construction, $(\Phi_t\ast f)(x)$ is an extension of $f$ to the upper half-space $\IR^{d+1}_+$, and $\T^{p,q}_\beta$ the extension space of $\dot \F^\beta_{p,q}$. However, a corresponding \enquote{tent space characterization} for the closely related Besov spaces $\dot \B^\beta_{p,q}$ has remained a notable gap in the theory. This gap was finally resolved in \cite{Auscher_Bechtel_Haardt}, where the authors introduced a new scale of function spaces, denoted by $\Z^{p,q,r}_\beta$ (see Definition~\ref{def: Z-space} below), which characterize homogeneous Besov spaces $\dot \B^\beta_{p,q}$ in the sense that
\begin{align*}
    f\in \dot \B^\beta_{p,q} \quad \Longleftrightarrow \quad   (\Phi_t\ast f)(x) \in \Z^{p,q,r}_\beta
\end{align*}
for any $0<r\leq \infty$. These spaces are extensions of weighted $\Z$-spaces, which were first introduced by Barton and Mayboroda to analyze boundary value problems with $\dot \B^\beta_{p,p}$-data, and later systematically studied by Amenta \cite{Amenta2} (see also \cite{AA}). In particular, Amenta showed that weighted $\Z$-spaces can be recovered by real interpolation of weighted tent spaces $\T^{p,q}_\beta$. To extend this behavior to the scale of $\Z^{p,q,r}_\beta$-spaces, the authors in \cite{Auscher_Bechtel_Haardt} introduced weighted tent spaces with an additional parameter $\T^{p,q,r}_\beta$ (see Definition~\ref{def: tent spaces}). Furthermore, they demonstrated that the scale of function spaces $\T^{p,q,r}_\beta$ is equivalent to the weighted tent spaces with Whitney average $\T^{p,r}_{q,\beta}$ developed by Huang \cite{Huang}, and characterizes homogeneous Triebel--Lizorkin spaces $\dot \F^\beta_{p,q}$. However, the characterization only covers the case $p<\infty$. We will complement this result by providing a corresponding endpoint characterization in the sense that
\begin{align*}
    f\in \dot \F^\beta_{\infty,q} \quad \Longleftrightarrow \quad   (\Phi_t\ast f)(x) \in \T^{\infty,q,r}_\beta
\end{align*}
for any $0<r\leq \infty$, see Theorem~\ref{thm: char of lifted Triebel data}. Here, the space $\T^{\infty,q,r}_\beta$ is defined as the set of all measurable functions $f$ on $\IR^{d+1}_+$ such that 
\begin{align*}
    &\|f\|_{\T^{\infty,q,r}_\beta}=\sup\limits_{\tau>0} \sup\limits_{y\in\IR^d} 
    \bigg(\int\limits_{0}^\tau  \fint\limits_{B(y,\tau)} \bigg(\fint\limits_{\frac{t}{2}}^t \fint\limits_{B(x,t)} |s^{-\beta}f(s,z)|^r \,\d z \d s\bigg)^\frac{q}{r}\,\frac{\d x \d t}{t}\bigg)^\frac{1}{q} < \infty.
\end{align*}
As an illustration, we obtain a Gauss--Weierstrass characterization of $\dot \F^\beta_{\infty,q}$ within the space $\T^{\infty,q,r}_\beta$ in the case $\beta<0$, see Proposition~\ref{prop: Gauss--Weierstrass char}.

\subsection{Tent spaces and Triebel--Lizorkin spaces}

In \cite{Auscher_Bechtel_Haardt}, a comprehensive theory for the spaces $\Z^{p,q,r}_\beta$ was developed, demonstrating a rigorous consistency with the theory of homogeneous Besov spaces $\dot \B^\beta_{p,q}$. This means that properties such as duality, embeddings and interpolation on the level of $\Z^{p,q,r}_\beta$-spaces possess one-to-one counterparts within the $\dot \B^\beta_{p,q}$ scale. By contrast, the function space theory for tent spaces still exhibits significant gaps when compared to the well-studied theory of homogeneous Triebel--Lizorkin spaces $\dot \F^\beta_{p,q}$. Our work aims to resolve these discrepancies. To highlight this parallel development, the subsequent sections are structured as follows: we first review the well-known theory of Triebel--Lizorkin spaces, then we present our new findings for $\T^{p,q,r}_\beta$-spaces in direct comparison, and finally identify the specific gaps in the existing literature that our results effectively close.

\subsubsection{Interpolation}
We introduce the following parameters. For $0< p_0,p_1,q_0,q_1, r_0,r_1\leq \infty$, $\beta_0,\beta_1\in\IR$ and $\theta\in (0,1)$ define
\begin{align*}
    \frac{1}{p_\theta} = \frac{1-\theta}{p_0} + \frac{\theta}{p_1},\quad \frac{1}{q_\theta} = \frac{1-\theta}{q_0} + \frac{\theta}{q_1},\quad \frac{1}{r_\theta} = \frac{1-\theta}{r_0} + \frac{\theta}{r_1}, \quad \beta_\theta = (1-\theta)\beta_0 + \theta \beta_1.
\end{align*}
The interpolation theory of homogeneous Triebel–Lizorkin spaces is well established; comprehensive treatments can be found in \cite{Triebel1}, \cite{Frazier_Jawerth}, and the references therein.
Combining \cite[Cor.\@ 6.7]{Frazier_Jawerth} and \cite[Sec.\@ 5.2.5]{Triebel1}, we can deduce the following identities:
\vspace{0.5cm}
\begin{itemize}
    \item{\makebox[4cm][l]{$[\dot \F_{p_0,q_0}^{\beta_0} , \dot\F_{p_1,q_1}^{\beta_1} ]_{\theta} = \dot\F_{p_\theta, q_\theta}^{\beta_\theta}$} if $\max\{q_0 , p_0\}<\infty$ or $\max\{q_1,p_1\}<\infty$,}\vspace{0.5cm}

    \item {\makebox[4cm][l]{$(\dot \F_{p_0,q}^{\beta} , \dot\F_{p_1,q}^{\beta} )_{\theta,p_\theta} = \dot\F_{p_\theta, q}^{\beta}$} if $\beta\in \IR$ and $ 0<q\leq \infty$,}\vspace{0.5cm}

    \item {\makebox[4cm][l]{$(\dot \F_{p,q_0}^{\beta_0} , \dot\F_{p,q_1}^{\beta_1} )_{\theta,q} = \dot\B_{p, q}^{\beta_\theta}\quad $} if $ \beta_0\neq \beta_1$ and $ 0<p,q\leq \infty$.}\vspace{0.5cm}
\end{itemize}
We establish analogous identities for the spaces $\T^{p,q,r}_\beta$ in Section~\ref{subsec: basic prop} and~\ref{sec: real interpolation}, which are summarized as follows:
\vspace{0.5cm}
\begin{enumerate}[label=(\roman*)]
    \item\label{item1} {\makebox[5.2cm][l]{$[\T^{p_0,q_0,r_0}_{\beta_0} , \T^{p_1,q_1,r_1}_{\beta_1} ]_{\theta} = \T^{p_\theta,q_\theta,r_\theta}_{\beta_\theta} $} if $\max\{q_0 , p_0, r_0\}<\infty$ or $\max\{q_1,p_1,r_1\}$,}\vspace{0.5cm}

    \item\label{item2} {\makebox[5.2cm][l]{$(\T^{p_0,q , r}_{\beta} , \T^{p_1,q,r}_{\beta} )_{\theta,p_\theta} = \T^{p_\theta, q,r}_{\beta}$} if $\beta\in \IR$ and $ 0<q,r\leq \infty$,}\vspace{0.5cm}

    \item\label{item3} {\makebox[5.2cm][l]{$(\T^{p,q_0,r}_{\beta_0} , \T^{p,q_1,r}_{\beta_1} )_{\theta,q} = \Z^{p, q,r}_{\beta_\theta}$} if $ \beta_0\neq \beta_1$ and $ 0<p,q,r\leq \infty$,}\vspace{0.5cm}
\end{enumerate}
see Proposition~\ref{prop: complex int}, Theorem~\ref{thm: real interpolation in p} and Proposition~\ref{prop: real interpolation tent spaces full range}, respectively.
We compare these with existing tent space results. First, complex interpolation \ref{item1} was already developed by Huang for weighted tent spaces with Whitney averages, denoted by $\T^{p,r}_{q,\beta}$, see \cite[Thm.\@ 3.6]{Huang}. Second, \ref{item2} is known for two parameter tent spaces $\T^{p,q}_\beta$, see for example \cite{Cao_et_al}. However, to the best of our knowledge, a similar result for $\T^{p,r}_{q,\beta}$-spaces is unknown. Finally, \ref{item3} was already proven in \cite[Prop.\@ 4.4]{Auscher_Bechtel_Haardt} but only in the range $0<p<\infty$.

\subsubsection{Discrete characterizations}

A powerful tool in the analysis of Triebel--Lizorkin spaces is the so-called $\varphi$-transform, see for example \cite{Frazier_Jawerth}. It connects the continuous space $\dot \F^\beta_{p,q}$ to the sequence space $\mathrm{\dot f}_{p,q}^\beta$, providing a discrete characterization of Triebel--Lizorkin spaces. Here, the space $\mathrm{\dot f}_{p,q}^\beta$ is defined as the set of all sequences $(s_Q)_{Q\in \square}\subset \IC$ such that
\begin{align*}
    &\|(s_Q)_Q\|_{\mathrm{\dot f}_{p,q}^\beta} = \Big\|\Big(\sum\limits_{Q\in\square} |Q|^{-\frac{\beta q}{d}} |s_Q|^q \mathbf{\tilde 1}_{Q}(\cdot) \bigg)^\frac{1}{q}\Big\|_{\L^p}<\infty\qquad (0<p<\infty),\\
    &\|(s_Q)_Q\|_{\mathrm{\dot f}_{\infty,q}^\beta} = \sup\limits_{P\in\square}\Big(\fint\limits_{P}\sum\limits_{Q\subset P} |Q|^{-\frac{\beta q}{d}} |s_Q|^q \mathbf{\tilde 1}_{Q}(\cdot) \,\d x \bigg)^\frac{1}{q}<\infty,
\end{align*}
where in the latter case the sum runs over all dyadic cubes $Q\in\square$ that are contained in $P$, and $\mathbf{\tilde 1}_Q(x) = |Q|^{-\frac{1}{2}}\mathbf{1}_Q$ is a $\L^2$-normalized cut-off function associated to $Q$. Similar discrete characterizations of $\T^{p,q,r}_\beta$ will be obtained in Proposition~\ref{prop: dyadic char} and~\ref{prop: dyadic char p=infty} and read as follows:
\begin{align*}
    \|f\|_{\T^{p,q,r}_\beta} \simeq \bigg\| \Big(\sum\limits_{Q\in\square} \mathbf{1}_{Q}(\cdot)|Q|^{-\frac{\beta q}{d}} \|f\|_{\L^r\big(\bar Q , \frac{\d y \d s}{s^{d+1}}\big)}^q \Big)^\frac{1}{q} \bigg\|_{\L^p}
\end{align*}
if $0<p<\infty$, and
\begin{align*}
    \|f\|_{\T^{\infty,q,r}_\beta} \simeq \sup\limits_{P\in\square} \bigg(\fint\limits_{P} \sum\limits_{Q\subset P} \mathbf{1}_{Q}(x)|Q|^{-\frac{\beta q}{d}} \|f\|_{\L^r\big(\bar Q , \frac{\d y \d s}{s^{d+1}}\big)}^q \,\d x \bigg)^\frac{1}{q}.
\end{align*}
Setting $s_Q = |Q|^{\frac{1}{2}}\|f\|_{\L^r\big(\bar Q , \frac{\d y \d s}{s^{d+1}}\big)}$ allows us to transfer many properties of the sequence spaces $\mathrm{\dot f}_{p,q}^\beta$ to our tent spaces $\T^{p,q,r}_\beta$. These include embeddings (Section~\ref{sec: embeddings}), duality (Section~\ref{sec: duality}) and real interpolation (Section~\ref{sec: real interpolation}), as well as a John--Nirenberg-type property for the endpoint spaces $\T^{\infty,q,r}_\beta$ (Proposition~\ref{prop: John-Nirenberg continuous}), which reads as follows:
\begin{align*}
    \|f\|_{\T^{\infty,q,r}_\beta} &\simeq  \sup\limits_{\tau>0} \sup\limits_{y\in\IR^d} 
        \bigg(  \fint\limits_{B(y,\tau)} \bigg(\int\limits_{0}^\tau\bigg(\fint\limits_{\frac{t}{2}}^t \fint\limits_{B(x,t)} |s^{-\beta}f(s,z)|^r \,\d z \d s\bigg)^\frac{q}{r}\,\frac{ \d t}{t}\bigg)^\frac{\alpha}{q}\d x\bigg)^\frac{1}{\alpha}
    \end{align*}
holds for any $0<\alpha<\infty$. To the best of our knowledge, the relationship between the sequence spaces $\mathrm{\dot f}_{p,q}^\beta$ and tent spaces has not been investigated before.

\subsubsection{Duality}
Triebel--Lizorkin spaces possess a rich duality theory, which can be found in the monograph of Triebel \cite[Sec.\@ 2.11]{Triebel1} and complemented by Frazier--Jawerth\cite[Sec.\@ 5]{Frazier_Jawerth}. Their duality theory can be summarized as follows:
\begin{align}
\label{eq: dual of F}
    (\dot \F^\beta_{p,q})' \simeq 
    \begin{cases}
       \dot \F^{-\beta}_{p',q'}  &\quad  \text{if } 1\leq p<\infty\text{ and } 0<q<\infty,\\
        \dot \B^{-\beta+d(\frac{1}{p}-1)}_{\infty,\infty}&\quad  \text{if }  0<p<1  \text{ and }  0<q<\infty.
    \end{cases}
\end{align}
Here, $q'$ denotes the Hölder conjugate exponent of $q$ defined by $1=\frac{1}{q}+\frac{1}{q'}$ in the case $q\in(1,\infty)$ and $q'=\infty$ if $q\in(0,1]$. Another central point of this article is to explore the corresponding duality theory for $\T^{p,q,r}_\beta$-spaces. Due to the powerful discrete characterizations of tent spaces, as described above, we obtain the following duality theory:
\begin{align*}
    (\T^{p,q,r}_\beta)' \simeq 
    \begin{cases}
        \T^{p',q',r'}_{-\beta}  &\quad  \text{if }  1\leq p<\infty \text{ and }  0<q<\infty,\\
        \Z^{\infty,\infty,r'}_{-\beta+d(\frac{1}{p}-1)}&\quad  \text{if } 0<p<1 \text{ and } 0<q<\infty,
    \end{cases}
\end{align*}
where $1\leq r<\infty$, see Theorem~\ref{thm: duality p>1 but q<1} and~\ref{thm: duality p<1}. We want to remark that duality results for $\T^{p,r}_{q,\beta}$-spaces are only known for parameters in the Banach range $1\leq p,q,r<\infty$. In addition, duality results for two-parameter tent spaces $\T^{p,q}_\beta$ spaces are only known for $1\leq q<\infty$.

\subsubsection{Embeddings}
\label{subsubsec: embeddings}
It is well-known that Triebel--Lizorkin spaces have a variety of function space embeddings. For instance, they possess a nesting property in the micro-local parameter $q$, that means
\begin{align*}
    \dot \F_{p,q_0}^{\beta} \subset \dot \F_{p,q_1}^{\beta} 
\end{align*}
for $0<p\leq \infty$, $0<q_0\leq q_1\leq \infty$ and $\beta\in\IR$. We will show in Lemma~\ref{lem: embedding in q}, that the nesting property
\begin{align*}
    \T^{p,q_0,r_0}_{\beta} \hookrightarrow \T^{p,q_1,r_1}_{\beta}
\end{align*}
holds for $0<r_1\leq r_0\leq \infty$ and the other parameters as above. To the best of our knowledge, these were unknown in this generality for tent spaces. Moreover, Triebel--Lizorkin spaces are known to satisfy so-called Hardy--Littlewood--Sobolev-type embeddings. These can be regarded as an extension of the usual Sobolev embedding to other function spaces. Roughly speaking, they describe a continuous embedding of a function space $X_0 =X(p_0,\beta_0)$ into another space $X_1=X(p_1,\beta_1)$, where $X_1$ has more integrability than $X_0$ in the sense that $p_0<p_1$ but loses regularity ($\beta_0>\beta_1$). The corresponding embeddings for Triebel--Lizorkin spaces are the following: for $0<p_0<p_1\leq\infty$, $0<q_0,q_1\leq \infty$ and $\beta_0,\beta_1\in \IR$ we have
\begin{align*}
    \dot \F_{p_0,q_0}^{\beta_0} \hookrightarrow \dot \F_{p_1,q_1}^{\beta_1} \quad \text{if} \quad \beta_0-\beta_1 = \frac{d}{p_0}-\frac{d}{p_1},
\end{align*}
see for example \cite[Thm.\@ 2.1]{Jawerth}. Remarkably, the micro-local parameters $q_0,q_1$ can be chosen arbitrary without any constrains. In contrast, on the level of tent spaces $\T^{p,q}_\beta$, Amenta \cite{Amenta2} proved the following Hardy--Littlewood--Sobolev-type embeddings:
\begin{align*}
    \T^{p_0,q_0}_{\beta_0} \hookrightarrow \T^{p_1,q_1}_{\beta_1}\quad \text{if} \quad \beta_0-\beta_1 = \frac{d}{p_0}-\frac{d}{p_1}, \quad q_1=q_0.
\end{align*}
They were later improved by Fraccaroli in \cite{Fraccaroli} to the case $q_1 \leq q_0$, but still do not resemble the Triebel--Lizorkin result. However, using our three-parameter tent spaces $\T^{p,q,r}_\beta$, we can prove an analogous Triebel--Lizorkin result on the level of tent spaces, which covers the known cases, and reads as follows: for $0<p_0<p_1\leq\infty$, $0<r_1\leq r_0\leq \infty$, $0<q_0,q_1\leq \infty$ and $\beta_0,\beta_1\in \IR$ we have
\begin{align*}
    \T^{p_0,q_0,r_0}_{\beta_0} \hookrightarrow \T^{p_1,q_1,r_1}_{\beta_1} \quad \text{if} \quad \beta_0-\beta_1 = \frac{d}{p_0}-\frac{d}{p_1},
\end{align*}
see Theorem~\ref{thm: Hardy-Sobolev embedding}. Finally, there are also mixed-type embeddings, which allow one to switch between the function scales $\dot \F_{p,q}^{\beta}$ and $\dot \B_{p,q}^{\beta}$. In particular, for $0<p_0<p_1\leq \infty$, $0<q \leq \infty$ and $\beta_0,\beta_1\in \IR$, we have
\begin{align*}
        \dot \F^{\beta_0}_{p_0,q} \hookrightarrow \dot \B^{\beta_1}_{p_1,p_0} \quad \text{if} \quad \beta_0-\beta_1 = \frac{d}{p_0}-\frac{d}{p_1},
\end{align*}
as well as
\begin{align*}
        \dot\B^{\beta_0}_{p_0,p_1} \hookrightarrow \dot \F^{\beta_1}_{p_1,q} \quad \text{if} \quad \beta_0-\beta_1 = \frac{d}{p_0}-\frac{d}{p_1}.
    \end{align*}
The first embedding was proven by Jawerth in \cite[Thm.\@ 2.1 (iii)]{Jawerth} and the second by Franke in \cite[Thm.\@ 1]{Franke}. We derive equivalent mixed-type embeddings for tent and $\Z$-spaces in Theorem~\ref{thm: mixed embedding}, that are summarized as follows: for $0<p_0<p_1\leq \infty$, $0<q\leq \infty$, $0<r_1\leq r_0\leq \infty$ and $\beta_0,\beta_1\in \IR$ we have
\begin{align*}
    \T^{p_0,q,r_0}_{\beta_0} \hookrightarrow \Z^{p_1,p_0,r_1}_{\beta_1}\quad \text{if} \quad \beta_0-\beta_1 = \frac{d}{p_0}-\frac{d}{p_1},
\end{align*}
as well as
\begin{align*}
    \Z^{p_0,p_1,r_0}_{\beta_0} \hookrightarrow \T^{p_1,q,r_1}_{\beta_1}\quad \text{if} \quad \beta_0-\beta_1 = \frac{d}{p_0}-\frac{d}{p_1}.
\end{align*}
We remark that mixed-type embeddings for two-parameter tent and $\Z$-spaces were proven in \cite[Thm.\@ 2.34]{AA}. However, they do not seem to be optimal with regard to the Besov--Triebel--Lizorkin theory.

\subsection{Roadmap and methods}

Our paper subdivides into two parts: First, Section~\ref{sec: characterizations} provides a characterization of $\dot \F^\beta_{\infty,q}$ in terms of the spaces $\T^{\infty,q,r}_\beta$. Second, we conduct a systematic study of functional-analytic properties of $\T^{p,q,r}_\beta$ in Section~\ref{sec: tent spaces}. This includes properties such as completeness, density, and complex interpolation (Section~\ref{subsec: basic prop}), discrete characterizations (Section~\ref{sec: dyadic char}), embeddings (Sections~\ref{sec: embeddings}), duality (Section~\ref{sec: duality}), as well as  real interpolation (Section~\ref{sec: real interpolation}).

\subsubsection*{Characterization of $\dot \F^\beta_{\infty,q}$}

This part relies heavily on the preliminary work by Ullrich and coauthors~\cite{Ullrich2} and a John--Nirenberg-type property for Triebel--Lizorkin and tent spaces at the endpoints. The latter property is used to lower the assumptions on the local means in the characterization, enabling Gauss-Weierstrass characterizations (Proposition~\ref{prop: Gauss--Weierstrass char}). In \cite{Ullrich2}, the authors have derived powerful characterizations for Besov--Triebel--Lizorkin--Hausdorff spaces using maximal functions. With them at hand, it is fairly straightforward to show the inclusion of $\dot \F^\beta_{\infty,q}$ into our four parameter tent space. The converse inclusion is harder as we have to dominate the pointwise maximal function by averages, which is unfortunately lengthy and technical. The characterization that we obtain in the end is Theorem~\ref{thm: char of lifted Triebel data}.

\subsubsection*{Functional-analytic properties of tent spaces}
In Section~\ref{subsec: basic prop}, we first use the equivalence of $\T^{p,q,r}_\beta$-spaces and weighted tent spaces with Whitney averages $\T^{p,r}_{q,\beta}$ to transfer known basic properties. These include completeness, dense subspaces, duality in the Banach range as well as complex interpolation theory. Additionally, we prove quantitative ``change of angle" formulas by combining basic integral estimates with covering arguments adapted to the geometry of Whitney boxes.

In Section~\ref{sec: dyadic char}, we develop discrete characterizations of tent spaces as a general tool, which follow from straightforward covering arguments (Proposition~\ref{prop: dyadic char} and~\ref{prop: dyadic char p=infty}). Along with the theory of sequence spaces, these lead to characterizations via ``discrete local means" (Proposition~\ref{prop: char with m}) and John--Nirenberg-type properties for $\T^{\infty,q,r}_\beta$ spaces (Proposition~\ref{prop: John-Nirenberg continuous}).

In Section~\ref{sec: embeddings}, we present Hardy--Littlewood--Sobolev-type embeddings (Theorem~\ref{thm: Hardy-Sobolev embedding}) and mixed-type embeddings of tent and $\Z$-spaces (Theorem~\ref{thm: mixed embedding}) that are consistent with the Besov--Triebel--Lizorkin theory described above. Their proofs rely on distributional descriptions of tent spaces combined with duality, convexity and real interpolation techniques.
This procedure is inspired by \cite[Thm.\@ 2.1]{Jawerth} and \cite[Thm.\@ 1]{Franke}, where the authors show analogous results for Besov--Triebel--Lizorkin spaces.

Section~\ref{sec: duality} develops a duality theory in the non-Banach range by embedding our tent spaces into simpler spaces for which duality results are already known. For the case $1\leq p<\infty$ and $0<q<\infty$, we embed them into vector-valued Lebesgue spaces (see Theorem~\ref{thm: duality p>1 but q<1}). For $0<p<1$ and $0<q<\infty$, we embed quasi-Banach tent spaces into quasi-Banach $\Z$-spaces or in tent spaces in the Banach range, for which duality theory is already known (see Theorem~\ref{thm: duality p>1 but q<1}).

Finally, in Section~\ref{sec: real interpolation}, we develop real interpolation results. There are two of them. The first (Theorem~\ref{thm: real interpolation in p}) interpolates in the parameter $p$ while fixing the parameters $q,r$ and $\beta$. It is based on the ``discrete local mean" characterization of $\T^{p,q,r}_\beta$ from Section~\ref{sec: dyadic char} and direct estimates of the ``best approximation" functional $E$. The second (Proposition~\ref{prop: real interpolation tent spaces full range}) fixes the parameters $p,r$ while allowing flexibility for the other parameters $q,\beta$. Its proof is based on a nesting property of tent and $\Z$-spaces and follows the procedure as in the proof of \cite[Prop.\@ 4.4]{Auscher_Bechtel_Haardt}.

\subsection{Acknowledgments}
The author would like to thank Pascal Auscher and Sebastian Bechtel for their enriching discussions, and Sebastian Bechtel in particular for his valuable feedback. Additionally, the author would like to thank Emiel Lorist for pointing out a gap in one of the arguments. The project was supported by \textit{Studienstiftung des deutschen Volkes}.

\subsection{Notation}
\label{subsec: notation}
We use the following notation throughout the paper.
We denote by $\IN,\IN_0,\IZ, \IR,\IC$ the set of all positive integers, all non-negative integers, all integers, all real numbers, and all complex numbers, respectively. Throughout, $d\geq 1$ denotes the dimension of the underlying Euclidean space. We write $\IR^{d+1}_+ = (0,\infty)\times \IR^{d}$ for the upper half-space. We denote the open ball centered at $x\in \IR^d$ with radius $t>0$ by $B(x,t)$. The Whitney box associated to $(t,x)\in \IR^{d+1}_+$ is denoted by $W(t,x)=(\frac{t}{2},t]\times B(x,t)$. For $k\in \IZ$, we define the set of all dyadic half-open cubes of generation $k$ as
\begin{align*}
    \square_k \coloneqq \{ 2^kx + [0,2^k)^d : x\in \IZ^d\},
\end{align*}
and denote by $\square = \cup_{k\in\IZ} \square_k$ the family of all dyadic cubes. Moreover, we denote by
\begin{align*}
    \bar Q \coloneqq \Big(\frac{\ell(Q)}{2} , \ell(Q) \Big] \times Q
\end{align*}
the associated Whitney box of a cube $Q$.
We use the symbol $\|\cdot\|_X$ for the (quasi-)norm of a (quasi-)normed space $X$, and $X'$ as its anti-dual space. For a measure space $\Omega$ and a Banach space $X$ we denote by $\L^0(\Omega;X)$ the space of all strongly measurable functions with the usual identification (that is, functions equal almost everywhere are identified). Moreover, for $0<p\leq \infty$, we write
\begin{align*}
    \L^p(\Omega; X) &= \bigg\{ f\in \L^0(\Omega;X): \|f\|_{\L^p} = \bigg(\int\limits_{\Omega} \|f(x)\|_X^p\,\d \mu\bigg)^\frac{1}{p}<\infty \bigg\}, 
\end{align*}
with the usual modification if $p=\infty$. For $p\in (1 ,\infty)$, we denote by $p'\in (1,\infty)$ the unique number such that $1 = \frac{1}{p} + \frac{1}{p'}$. Moreover, for $p\in(0,1]$, put $p'\coloneqq \infty$. Finally, if $p=\infty$, then set $p'\coloneqq 1$.
If $X$ is a quasi-Banach space, we also define for $0<p\leq \infty$ and $\beta\in \IR$ the weighted vector-valued sequence spaces $\ell^p_\beta(\IZ;X)$ as follows:
\begin{align*}
    \ell^p_\beta(\IZ;X)\coloneqq \Big\{ x: x=(x_k)_{k\in \IZ}\ \text{with}\ x_k\in X,\, \|x\|_{\ell^p_\beta(\IZ;X)} \coloneqq \Big(\sum\limits_{k\in\IZ} 2^{-k\beta p}\|x_k\|_X^p\Big)^\frac{1}{p}<\infty \Big\},
\end{align*}
again with the usual modification if $p=\infty$.
For $0<p,q\leq \infty$, we denote by $\L^p(\ell^q)$ the space of all sequences of complex-valued Borel measurable functions $(f_k)_{k\in\IZ}$ on $\IR^d$ such that
\begin{align*}
    \|(f_k)_{k\in\IZ}\|_{\L^p(\ell^q)} \coloneqq \bigg(\int\limits_{\IR^d} \Big(\sum\limits_{k\in\IZ} |f_k(x)|^q\Big)^\frac{p}{q}\,\d x \bigg)^\frac{1}{p}<\infty,
\end{align*}
with the usual modifications in the infinite cases.
For two real numbers $a,b$, we write $a\lesssim b$ if there exists a constant $C>0$, whose precise value may vary from line to line but depends at most on the structural constants, such that $a\leq Cb$. Conversely, we write $a\gtrsim b$ if $Ca\geq b$. Furthermore, we write $a\simeq b$ if both $a\lesssim b$ and $a\gtrsim b$ hold.

\section{Characterization of endpoint Triebel--Lizorkin spaces}

\label{sec: characterizations}

\noindent In this section, we provide a continuous characterization of endpoint Triebel--Lizorkin spaces with respect to general kernels with non-compact Fourier support. This complements the characterization \cite[Thm.\@ 2.4]{Auscher_Bechtel_Haardt}. We start with a brief introduction of the following objects from harmonic analysis as in \cite[Sec.\@ 2]{Auscher_Bechtel_Haardt}.

Let $\cS(\IR^d)$ denote the Schwartz space and $\cS'(\IR^d)$ its topological dual space, the space of tempered distributions. Furthermore, for $N\in\IN_0 \cup \{\infty\}$ we define
\begin{align*}
    \cS_N(\IR^d) \coloneqq \{ f \in \cS(\IR^d): [\D^\alpha \cF(f)](0) = 0, \quad \text{for all $\alpha\in \IN_0^d$ with $|\alpha|\leq N$}\},
\end{align*}
where $\cF$ denotes the Fourier transform and $\D^\alpha = \partial_1^{\alpha_1} \dots \partial_d^{\alpha_d}$. In this context, we also define $\cS_{-1}(\IR^d) \coloneqq \cS(\IR^d)$. Moreover, let $\cS_N'(\IR^d)$ denote the topological dual space of $\cS_N(\IR^d)$, where $\cS_N(\IR^d)$ is equipped with the subspace topology of $\cS(\IR^d)$. It can be identified with $\cS'(\IR^d)/\cP_N(\IR^d)$, where $\cP_N(\IR^d)$ is the space of all polynomials on $\IR^d$ of degree at most $N$ with complex coefficients, see \cite[Chap.\@ 5]{Triebel1}. If $\varphi$ and $\psi$ are integrable functions, then their convolution $\psi\ast \varphi$ is defined as
\begin{align*}
    (\psi\ast \varphi)(x) = \int\limits_{\IR^d} \psi(x-y)\varphi(y)\,\d y \quad (x\in\IR^d).
\end{align*}
If both $\psi$ and $\varphi$ belong to $\cS(\IR^d)$, then their convolution also belongs to $\cS(\IR^d)$. Furthermore, the convolution can be generalized to $(\psi, f ) \in \cS_N(\IR^d)\times \cS'_N(\IR^d)$ via $(\psi\ast f)(x) = f(\psi(x-\cdot))$, which is well-defined for $x\in\IR^d$ and has at most polynomial growth.\\
Finally, we introduce the definition of the Peetre maximal function for dilations. Fix $\Psi\in \cS_N(\IR^d)$ and $f\in \cS'_N(\IR^d)$. Then, we define for $a>0$ and $t>0$ the Peetre maximal function as
\begin{align}
    \label{eq:peetre_mf}
    (\Psi_t^*f)_a(x) = \sup\limits_{y\in\IR^d} \frac{|(\Psi_t\ast f) (x+y)|}{\Big(1+\frac{|y|}{t} \Big)^a} \quad (x\in \IR^d),
\end{align}
where we set $\Psi_t(\cdot) \coloneqq t^{-d}\Psi(\cdot/t)$.

The following definition of the endpoint Triebel--Lizorkin spaces $\dot\F^{\beta}_{\infty,q}(\IR^d)$ is due to Frazier and Jawerth~\cite{Frazier_Jawerth}. If $1<q<\infty$, this definition is equivalent to the one introduced by Triebel in \cite[Sec.\@ 5.1.4]{Triebel1}.
For further properties of these spaces, we refer to \cite{Bui_Taibleson}, \cite{Frazier_Jawerth} and \cite[Sec.\@ 5.1.4]{Triebel1}.\\

\begin{definition}
    Let $(\varphi_k)_{k\in\IZ}\subset \cS(\IR^d)$ be such that
    \begin{enumerate}
        \item $\varphi_k(x) = \varphi_0(2^{-k}x)$ for every $x\in\IR^d$ and $k\in\IZ$,

        \item $\supp \varphi_0 \subset \{x\in\IR^d: 2^{-1}\leq |x|\leq 2\}$,

        \item $\sum\limits_{k\in\IZ} \varphi_k(x) = 1$ for every $x\in\IR^d\setminus \{0\}$,
    \end{enumerate}
    and set $\Phi_k \coloneqq \mathcal{F}^{-1}(\varphi_k)\in \cS_\infty(\IR^d)$. For $0<q<\infty$ and $\beta\in \IR$ we define the endpoint homogeneous Triebel--Lizorkin space as
    \begin{align*}
        \dot\F^{\beta}_{\infty,q}(\IR^d) \coloneqq \Big\{ f\in \cS_\infty'(\IR^d): \|f\|_{\dot\F^{\beta}_{\infty,q}} <\infty \Big\},
    \end{align*}
    where
    \begin{align*}
        &\|f\|_{\dot\F^{\beta}_{\infty,q}} \coloneqq \sup\limits_{Q\in \square}\bigg(\fint\limits_Q \sum\limits_{k=-\log_2(\ell(Q))}^\infty 2^{k\beta q} |(\Phi_k\ast f)(x)|^q \,\d x \bigg)^\frac{1}{q}.
    \end{align*}
\end{definition}

\begin{remark}
    We excluded the case $q=\infty$ because the corresponding definition for $\dot\F^{\beta}_{\infty,\infty}(\IR^d)$ is equivalent to $\dot\B^{\beta}_{\infty,\infty}(\IR^d)$ (see for example \cite{Bui_Taibleson}), and the latter space was already characterized in \cite[Thm.\@ 2.2]{Auscher_Bechtel_Haardt}.
\end{remark}

\begin{definition}
\label{def: X norm}
    Let $0<q,\alpha<\infty$ and $(g_k)_{k\in\IZ}$ a family of measurable functions in $\IR^d$. Then, we define
    \begin{align*}
        \|(g_k)_{k\in\IZ}\|_{X^{q,\alpha}} \coloneqq \sup\limits_{Q\in \square}\bigg( \fint\limits_Q \bigg(\sum\limits_{k=-\log_2(\ell(Q))}^\infty |g_k(x)|^q\bigg)^\frac{\alpha}{q}\,\d x\bigg)^\frac{1}{\alpha}.
    \end{align*}
    Notice that $\|f\|_{\dot\F^{\beta}_{\infty,q}}  = \|(2^{k\beta} \Phi_k\ast f)_{k\in\IZ}\|_{X^{q,q}}$.
\end{definition}

As a technical tool for our characterization of $\dot\F^{\beta}_{\infty,q}(\IR^d)$, we recall the following convolution inequality from \cite[Lem.\@ 2.1]{Ullrich2}.

\begin{lemma}
\label{lem: convolution lemma}
    Let $0<q,\alpha<\infty$ and $\delta \in(\frac{d}{\alpha},\infty)$. Suppose that $(g_l)_{l\in\IZ}$ is a family of measurable functions in $\IR^d$. Then one has
    \begin{align*}
        \Big\|\Big(\sum\limits_{k\in\IZ} 2^{-|k-l|\delta} g_k\Big)_{l\in\IZ} \Big\|_{X^{q,\alpha}}\lesssim \|( g_l)_{l\in\IZ} \|_{X^{q,\alpha}}.
    \end{align*}
\end{lemma}

In \cite{Auscher_Bechtel_Haardt}, homogeneous Besov and Triebel--Lizorkin spaces are characterized by continuous families of convolution operators $(\Psi_t)_{t\in\IR}$ where the kernels decay fast enough at zero and infinity. In particular, a scale of tent spaces $\T^{p,q,r}_\beta$ was introduced (see Section~\ref{sec: tent spaces} below) to characterize homogeneous Triebel--Lizorkin spaces $\dot\F^{\beta}_{p,q}(\IR^d)$ in the case $0<p<\infty$. However, a corresponding endpoint characterization for $\dot\F^{\beta}_{\infty,q}(\IR^d)$ with respect to this tent space scale is still missing. The following theorem provides this missing characterization.

\begin{theorem}
\label{thm: char of lifted Triebel data}
    Let $\beta\in \IR$, $0<q<\infty$, $0<r\leq \infty$ and $R\in\IN_0\cup \{-1\}$ such that $R+1>\beta$. Furthermore, let $\Phi\in \cS(\IR^d)$ be such that
    \begin{align}
    \label{property 1}
        |\cF\Phi (\xi)|>0 \quad \text{on} \quad \Big\{\frac{\varepsilon}{2}<|\xi|<\varepsilon \Big\}
    \end{align}
    for some $\varepsilon>0$, and
    \begin{align}
    \label{property 2}
        [D^\alpha \cF\Phi] (0) = 0 \quad \text{for all $\alpha\in \IN_0^d$ with $|\alpha|\leq R$.}
    \end{align}
    Then, the space $\dot \F^\beta_{\infty,q}(\IR^d)$ can be characterized as follows (with the usual modifications if $r$ is infinite):
    \begin{align*}
        \dot \F^\beta_{\infty,q}(\IR^d) = \bigg\{f\in \cS'_R(\IR^d): \|\Phi_s\ast f  \|_{\T^{\infty,q,r}_\beta}<\infty \bigg\},
    \end{align*}
    where
    \begin{align*}
        \|\Phi_s\ast f\|_{\T^{\infty,q,r}_\beta} \coloneqq \sup\limits_{y\in\IR^d}\sup\limits_{\tau>0} \bigg(\int\limits_0^{\tau} \fint\limits_{B(y,\tau)} \bigg(\fint\limits_{\frac{t}{2}}^t \fint\limits_{B(x,t)} |s^{-\beta}(\Phi_s\ast f)(z)|^r\,\d z \d s\bigg)^\frac{q}{r}\,\frac{\d x \d t}{t}\bigg)^\frac{1}{q}.
    \end{align*}
\end{theorem}

\begin{remark}
\label{rem:characterization}
We clarify how the characterization identity has to be understood.
\begin{itemize}
    \item In this theorem, the homogeneous space on the left-hand side is understood as a realization within $\cS_R'(\IR^d)$, where $R$ is as in the statement. This realization aligns with the framework presented in \cite[Lem.\@ 3.4]{Ullrich2} in the following sense. By \cite[Lem.\@ 3.4]{Ullrich2}, we can realize $\dot \F^\beta_{\infty,q}(\IR^d)$ within the space $\cS_M'(\IR^d)$, where $M=\max\{\lfloor \beta\rfloor,-1\}$ and $\lfloor \beta\rfloor$ denotes the integer part of $\beta\in\IR$. Because $M\leq R$ by the bullet point below, we have the embedding $\cS_M'(\IR^d)\subset \cS_R'(\IR^d)$ and can thus realize $\dot \F^\beta_{\infty,q}(\IR^d)$ within $\cS_R'(\IR^d)$. As a consequence, the identity in Theorem~\ref{thm: char of lifted Triebel data} is understood as a set-theoretic equality in $\cS'_R(\IR^d)$ with corresponding equivalence of norms.

    \item We claim that $M\leq R$, which implies $\cS_M'(\IR^d)\subset \cS'_R(\IR^d)$. Indeed, if $M=-1$ then $M\leq R$ holds trivially. In the case $M=\lfloor \beta\rfloor$, assume for the sake of contradiction that $M>R$. Since $M,R\in\IZ$, this implies $R+1\leq M$ and thus
    \begin{align*}
        \beta< R+1\leq M  \leq  \beta,
    \end{align*}
    which is impossible.
\end{itemize}
\end{remark}

\noindent Before proving Theorem~\ref{thm: char of lifted Triebel data}, we want to mention that the assumptions of the theorem agree with those of \cite[Thm.\@ 2.4]{Auscher_Bechtel_Haardt}.
This allows us to use their techniques and derived identities as a black box for our proof. However, the reader is advised to keep a copy of the mentioned references at hand.

\begin{proof}[Proof of Theorem~\ref{thm: char of lifted Triebel data}]

Let $X$ denote the space on the right-hand side in the statement. Then we have to show the equivalence $\dot \F^\beta_{\infty,q}(\IR^d) = X$. To do so, we will divide the proof into two steps. The first step proves the continuous embedding $\dot \F^\beta_{\infty,q}(\IR^d) \supset  X$ while the second step provides the reverse inclusion. For the rest of the proof, we fix $f \in \cS'_R(\IR^d)$.\\

\noindent \textbf{Step 1:  $\dot \F^\beta_{\infty,q}(\IR^d) \supset  X$.}
Assume that $f\in X$ and fix $a>\frac{d}{\min\{1,q,r\}}$ and $0<\gamma< \min\{1,q,r\}$ such that $ a\gamma >d$. Then we can use the identity (2.11) of \cite{Auscher_Bechtel_Haardt}, which reads as follows: for $\delta \coloneqq R+1-\beta$ we have
\begin{align}
\label{eq: F2}
    2^{l\beta\gamma}|(\Phi_{2^{-l}}^*f)_a(x)|^\gamma &\lesssim \sum\limits_{k\in\IZ} 2^{-|k-l|\delta\gamma} 2^{k\beta\gamma} \sum\limits_{m=0}^\infty 2^{-mN\gamma} 2^{(k+m)d}\int\limits_{\IR^d} \frac{|(\widetilde{\Phi}_{k+m}f)(y)|^\gamma}{(1+2^k|x-y|)^{a\gamma}}\,\d y \nonumber\\
    &=\sum\limits_{k\in\IZ} 2^{-|k-l|\delta\gamma} g_k(x)
\end{align}
for all $x\in \IR^d$ , $l\in \IZ$ and $N\geq a$, where we define
\begin{align}
\label{eq: def phi tilde}
    (\widetilde{\Phi}_{k+m}f)(y) \coloneqq  \bigg(\int\limits_{1}^2\bigg(\fint\limits_{\frac{\lambda}{2}}^\lambda \fint\limits_{B(y,2^{-(k+m)}t)} |(\Phi_{2^{-(k+m)}t}\ast f)(z) |^r\,\d z \d t\bigg)^\frac{q}{r}\,\frac{\d \lambda }{\lambda}\bigg)^\frac{1}{q}
\end{align}
and
\begin{align*}
    g_k(x)\coloneqq 2^{k\beta\gamma} \sum\limits_{m=0}^\infty 2^{-mN\gamma} 2^{(k+m)d}\int\limits_{\IR^d} \frac{|(\widetilde{\Phi}_{k+m}f)(y)|^\gamma}{(1+2^k|x-y|)^{a\gamma}}\,\d y.
\end{align*}
Now, since $R+1>\beta$, take $\alpha>1$ such that $ \delta=R+1-\beta>\frac{d}{\alpha\gamma}$. Then, we can use the estimate \eqref{eq: F2}, in conjunction with Lemma~\ref{lem: convolution lemma}, to get
\begin{align*}
     \Big\| \big(2^{l\beta\gamma}|(\Phi_{2^{-l}}^*f)_a(x)|^\gamma \big)_{l\in\IZ}\Big\|_{X^{\frac{q}{\gamma},\alpha}} &\lesssim \Big\| \Big(\sum\limits_{k\in\IZ} 2^{-|k-l|\delta\gamma} g_k\Big)_{l\in\IZ}\Big\|_{X^{\frac{q}{\gamma},\alpha}}\\
     &\lesssim \big\| (g_l)_{l\in\IZ}\big\|_{X^{\frac{q}{\gamma},\alpha}}\\ 
     &=\Big\| \Big(2^{l\beta \gamma}\sum\limits_{m=0}^\infty 2^{-mN\gamma} 2^{(l+m)d}\int\limits_{\IR^d} \frac{|(\widetilde{\Phi}_{l+m}f)(y)|^\gamma}{(1+2^l|x-y|)^{a\gamma}}\,\d y \Big)_{l\in\IZ}\Big\|_{X^{\frac{q}{\gamma},\alpha}}\\
     &=\Big\| \Big(\sum\limits_{m=l}^\infty 2^{-(m-l)(N+\beta)\gamma} 2^{m(d+\beta\gamma)}\int\limits_{\IR^d} \frac{|(\widetilde{\Phi}_{m}f)(y)|^\gamma}{(1+2^l|x-y|)^{a\gamma}}\,\d y \Big)_{l\in\IZ}\Big\|_{X^{\frac{q}{\gamma},\alpha}},
\end{align*}
where we used an index shift in the last line. Observe that the term on the left-hand side of the above estimate is equivalent to $\|f\|_{\dot \F^\beta_{\infty,q}}^\gamma$ by definition of $\|\cdot\|_{X^{\frac{q}{\gamma},\alpha}}$ and a John--Nirenberg-type property of $\dot \F^\beta_{\infty,q}(\IR^d)$, see \cite[Thm.\@ 3.2 + Rem.\@ 3.3]{Ullrich2} and \cite[Prop.\@ 4.1]{Danchun_Yuan}. Hence,
\begin{align}
\label{eq: rht0}
    \|f\|_{\dot \F^\beta_{\infty,q}}^\gamma&\lesssim  \Big\| \Big(\sum\limits_{m=l}^\infty 2^{-(m-l)(N+\beta)\gamma} 2^{m(d+\beta\gamma)}\int\limits_{\IR^d} \frac{|(\widetilde{\Phi}_{m}f)(y)|^\gamma}{(1+2^l|x-y|)^{a\gamma}}\,\d y \Big)_{l\in\IZ}\Big\|_{X^{\frac{q}{\gamma},\alpha}}\eqqcolon I
\end{align}
Let $Q\in\square$ be a dyadic cube and set $j_Q \coloneqq -\log_2(\ell(Q))$. Fix $l\geq j_Q$ and $x\in Q$. Then, we decompose the integral in \eqref{eq: rht0} into
\begin{align}
\label{eq: F3}
    \int\limits_{\IR^d} \frac{|(\widetilde{\Phi}_{m}f)(y)|^\gamma}{(1+2^l|x-y|)^{a\gamma}}\,\d y &\leq \int\limits_{8Q} \frac{|(\widetilde{\Phi}_{m}f)(y)|^\gamma}{(1+2^l|x-y|)^{a\gamma}}\,\d y + \sum\limits_{\substack{z\in\IZ^d \\ |z|_\infty\geq 2}} \int\limits_{Q+\ell(Q)z} \frac{|(\widetilde{\Phi}_{m}f)(y)|^\gamma}{(1+2^l|x-y|)^{a\gamma}}\,\d y,
\end{align}
where $|\cdot|_\infty$ is maximum norm on $\IR^d$. Notice that the functions $x\mapsto  \frac{2^{dl}}{(1+2^l|x|)^{a\gamma}}$ are radial and have uniform $\L^1$-bound with respect to $l\in\IZ$ because $a\gamma>d$. Using the majorant property of the Hardy--Littlewood maximal function, see \cite[Thm.\@ 2.1.10]{Grafakos}, we estimate the local term on the right-hand side of \eqref{eq: F3} by
\begin{align*}
     \int\limits_{8Q} \frac{|(\widetilde{\Phi}_{m}f)(y)|^\gamma}{(1+2^l|x-y|)^{a\gamma}}\,\d y\lesssim 2^{-ld} \cM\big(|\widetilde{\Phi}_{m}f|^\gamma \mathbf{1}_{8Q}\big)(x) \eqqcolon I_1^m(x).
\end{align*}
To bound the series in \eqref{eq: F3}, we use the observation
\begin{align*}
    1+2^l|x-y|\gtrsim 2^l\ell(Q)|z|_\infty
\end{align*}
for $x\in Q$ and $y\in Q+\ell(Q)z$. This yields
\begin{align*}
    &\sum\limits_{\substack{z\in\IZ^d \\ |z|_\infty\geq 2}} \int\limits_{Q+\ell(Q)z} \frac{|(\widetilde{\Phi}_{m}f)(y)|^\gamma}{(1+2^l|x-y|)^{a\gamma}}\,\d y\\&\leq \sum\limits_{\substack{z\in\IZ^d \\ |z|_\infty\geq 2}} 2^{-la\gamma}2^{j_Qa\gamma}|z|_\infty^{-a\gamma} \int\limits_{Q+\ell(Q)z} |(\widetilde{\Phi}_{m}f)(y)|^\gamma\,\d y)\\
    &\lesssim\sum\limits_{\substack{z\in\IZ^d \\ |z|_\infty\geq 2}}  2^{-la\gamma}2^{j_Q(a\gamma-d)}|z|_\infty^{-a\gamma}   \cM\big(|\widetilde{\Phi}_{m}f|^\gamma\mathbf{1}_{Q+\ell(Q)z} \big)(x+\ell(Q)z)\\
    &\lesssim\sum\limits_{\substack{z\in\IZ^d \\ |z|_\infty\geq 2}}  2^{-ld}|z|_\infty^{-a\gamma}   \cM\big(|\widetilde{\Phi}_{m}f|^\gamma\mathbf{1}_{Q+\ell(Q)z} \big)(x+\ell(Q)z),
\end{align*}
where we used $a\gamma>d$ and $l\geq j_Q$. Define
\begin{align*}
    I_2^m(x) \coloneqq \sum\limits_{\substack{z\in\IZ^d \\ |z|_\infty\geq 2}}  2^{-ld}|z|_\infty^{-a\gamma}   \cM\big(|\widetilde{\Phi}_{m}f|^\gamma\mathbf{1}_{Q+\ell(Q)z} \big)(x+\ell(Q)z).
\end{align*}
Then, we can bound $I$ from \eqref{eq: rht0} by
\begin{align}
\label{eq: F4}
      I&\lesssim  \Big\| \Big(\sum\limits_{m=l}^\infty 2^{-(m-l)(N+\beta)\gamma} 2^{m(d+\beta\gamma)} I_1^m \Big)_{l\in\IZ}\Big\|_{X^{\frac{q}{\gamma},\alpha}} \nonumber\\
      &\quad + \Big\| \Big(\sum\limits_{m=l}^\infty 2^{-(m-l)(N+\beta)\gamma} 2^{m(d+\beta\gamma)} I_2^m \Big)_{l\in\IZ}\Big\|_{X^{\frac{q}{\gamma},\alpha}}.
\end{align}
We estimate both terms on the right-hand side of \eqref{eq: F4} in two separate steps.\\

\noindent \textbf{Step 1.1: Estimate the first term in \eqref{eq: F4}}.
Fix a dyadic cube $Q\in \square$ and focus on the term
\begin{align*}
    \bigg(\fint\limits_Q \bigg(\sum\limits_{l=-\log_2(\ell(Q))}^\infty \Big|\sum\limits_{m=l}^\infty 2^{-(m-l)(N+\beta)\gamma} 2^{m(d+\beta\gamma)}I_1^m(x) \Big|^\frac{q}{\gamma} \bigg)^\frac{\alpha\gamma}{q}\,\d x\bigg)^\frac{1}{\alpha}.
\end{align*}
By Definition~\ref{def: X norm}, this is the first term in \eqref{eq: F4} up to a supremum. Using Hölder's inequality, we get
\begin{align*}
    &\bigg(\fint\limits_Q \bigg(\sum\limits_{l=-\log_2(\ell(Q))}^\infty \Big|\sum\limits_{m=l}^\infty 2^{-(m-l)(N+\beta)\gamma} 2^{m(d+\beta\gamma)}I_1^m(x) \Big|^\frac{q}{\gamma} \bigg)^\frac{\alpha\gamma}{q}\,\d x\bigg)^\frac{1}{\alpha}\\
    &\lesssim \bigg(\fint\limits_Q \bigg(\sum\limits_{l=-\log_2(\ell(Q))}^\infty \sum\limits_{m=l}^\infty 2^{-(m-l)(N-1+\beta)q} 2^{m(\frac{d}{\gamma}+\beta) q}|I_1^m(x)|^\frac{q}{\gamma} \bigg)^\frac{\alpha\gamma}{q}\,\d x\bigg)^\frac{1}{\alpha}\\
    &=\bigg(\fint\limits_Q \bigg(\sum\limits_{l=-\log_2(\ell(Q))}^\infty\sum\limits_{m=l}^\infty 2^{-(m-l)(N-1+\beta-\frac{d}{\gamma})q} 2^{m\beta q}\Big[\cM\big(|\widetilde{\Phi}_{m}f|^\gamma\mathbf{1}_{8Q}\big)(x)\Big]^\frac{q}{\gamma}\bigg)^{\frac{\alpha\gamma}{q}}\,\d x\bigg)^\frac{1}{\alpha}.
\intertext{Since $\gamma<q$ and $\alpha>1$, we can apply Fefferman--Stein's vector valued inequality, see for example \cite[Thm.\@ 2.1]{Ullrich1}, to obtain}
    &\lesssim\bigg(\fint\limits_{8Q} \bigg(\sum\limits_{l=-\log_2(\ell(Q))}^\infty\sum\limits_{m=l}^\infty 2^{-(m-l)(N-1+\beta-\frac{d}{\gamma})q} 2^{m\beta q}|\widetilde{\Phi}_{m}f|^q(x)\bigg)^{\frac{\alpha\gamma}{q}}\,\d x\bigg)^\frac{1}{\alpha}.
\intertext{Finally, we conclude with Fubini's theorem}
    &=\bigg(\fint\limits_{8Q} \bigg(\sum\limits_{m=-\log_2(\ell(Q))}^\infty\sum\limits_{l=-\log_2(\ell(Q))}^m 2^{-(m-l)(N-1+\beta-\frac{d}{\gamma})q} 2^{m\beta q}|\widetilde{\Phi}_{m}f|^q(x)\bigg)^{\frac{\alpha\gamma}{q}}\,\d x\bigg)^\frac{1}{\alpha}\\
    &\lesssim \bigg(\fint\limits_{8Q} \bigg(\sum\limits_{m=-\log_2(\ell(Q))}^\infty 2^{m\beta q}|\widetilde{\Phi}_{m}f|^q(x)\bigg)^{\frac{\alpha\gamma}{q}}\,\d x\bigg)^\frac{1}{\alpha},
\end{align*}
where we take $N>1-\beta+\frac{d}{\gamma}$. Now, by definition of $\widetilde{\Phi}_{m}$ in \eqref{eq: def phi tilde}, we substitute $2^{-m}t = s$ to see that
\begin{align*}
    &\bigg(\fint\limits_{8Q} \bigg(\sum\limits_{m=-\log_2(\ell(Q))}^\infty 2^{m\beta q}|\widetilde{\Phi}_{m}f|^q(x)\bigg)^{\frac{\alpha\gamma}{q}}\,\d x\bigg)^\frac{1}{\alpha}\\
    &=\bigg(\fint\limits_{8Q}\bigg(\sum\limits_{m=-\log_2(\ell(Q))}^\infty   2^{m\beta q} \int\limits_{1}^2\bigg(\fint\limits_{\frac{\lambda}{2}}^\lambda \fint\limits_{B(x,2^{-m}t)} |(\Phi_{2^{-m}t}\ast f)(z) |^r\,\d z \d t\bigg)^\frac{q}{r}\,\frac{\d \lambda }{\lambda} \bigg)^{\frac{\alpha\gamma}{q}}\,\d x\bigg)^\frac{1}{\alpha}\\
    &=\bigg(\fint\limits_{8Q}\bigg(\sum\limits_{m=-\log_2(\ell(Q))}^\infty 2^{m\beta q} \int\limits_{1}^2\bigg(\fint\limits_{\frac{2^{-m}\lambda}{2}}^{2^{-m}\lambda} \fint\limits_{B(x,s)} |(\Phi_s\ast f)(z) |^r\,\d z \d s\bigg)^\frac{q}{r}\,\frac{\d \lambda }{\lambda}\bigg)^{\frac{\alpha\gamma}{q}}\,\d x\bigg)^\frac{1}{\alpha}.
\intertext{Finally, use the substitution $\lambda=2^mt$ and the covering property $B(x,s)\subset B(x,t)$, to get}
    &=\bigg(\fint\limits_{8Q}\bigg(\sum\limits_{m=-\log_2(\ell(Q))}^\infty  2^{m\beta q} \int\limits_{2^{-m}}^{2^{-m+1}}\bigg(\fint\limits_{\frac{t}{2}}^{t} \fint\limits_{B(x,s)} |(\Phi_s\ast f)(z) |^r\,\d z \d s\bigg)^\frac{q}{r}\,\frac{\d t }{t} \bigg)^{\frac{\alpha\gamma}{q}}\,\d x\bigg)^\frac{1}{\alpha}\\
    &\lesssim \bigg(\fint\limits_{8Q} \bigg(\int\limits_{0}^{2\ell(Q)} \bigg(\fint\limits_{\frac{t}{2}}^{t} \fint\limits_{B(x,t)} |s^{-\beta}(\Phi_s\ast f)(z) |^r\,\d z \d s\bigg)^\frac{q}{r}\,\frac{\d t }{t} \bigg)^{\frac{\alpha\gamma}{q}}\,\d x\bigg)^\frac{1}{\alpha}\\
    &\lesssim\sup\limits_{\tau>0}\sup\limits_{y\in\IR^d}\bigg(\fint\limits_{B(y,\tau)} \bigg(\int\limits_{0}^{\tau} \bigg(\fint\limits_{\frac{t}{2}}^{t} \fint\limits_{B(x,t)} |s^{-\beta}(\Phi_s\ast f)(z) |^r\,\d z \d s\bigg)^\frac{q}{r}\,\frac{\d t }{t} \bigg)^{\frac{\alpha\gamma}{q}}\,\d x\bigg)^\frac{1}{\alpha}.
\end{align*}
Using the John--Nirenberg-type property for $\T^{\infty,q,r}_\beta$-spaces from Proposition~\ref{prop: John-Nirenberg continuous} below, we can bound the last term by $\|\Phi_s\ast f\|_{\T^{\infty,q,r}_\beta}^\gamma$. In summary, we have shown 
\begin{align*}
    \bigg(\fint\limits_Q \bigg(\sum\limits_{l=-\log_2(\ell(Q))}^\infty \Big|\sum\limits_{m=l}^\infty 2^{-(m-l)(N+\beta)\gamma} 2^{m(d+\beta\gamma)}I_1^m(x) \Big|^\frac{q}{\gamma} \bigg)^\frac{\alpha\gamma}{q}\,\d x\bigg)^\frac{1}{\alpha}
    \lesssim \|\Phi_s\ast f\|_{\T^{\infty,q,r}_\beta}^\gamma
\end{align*}
for each dyadic cube $Q\in \square$. Taking supremum in $Q$ yields boundedness of the first term in \eqref{eq: F4}, namely
\begin{align}
\label{eq: F6}
    \Big\| \Big(\sum\limits_{m=l}^\infty 2^{-(m-l)(N+\beta)\gamma} 2^{m(d+\beta\gamma)} I_1^m \Big)_{l\in\IZ}\Big\|_{X^{\frac{q}{\gamma},\alpha}}\lesssim\|\Phi_s\ast f\|_{\T^{\infty,q,r}_\beta}^\gamma.
\end{align}

\noindent \textbf{Step 1.2: Estimate the second term in \eqref{eq: F4}}.
Using a similar procedure, we want to bound the series in \eqref{eq: F4} by $\|\Phi_s\ast f\|_{\T^{\infty,q,r}_\beta}^\gamma$ likewise. First, we apply Minkowski's inequality to get
\begin{align}
\label{eq: F5}
    &\Big\| \Big(\sum\limits_{m=l}^\infty 2^{-(m-l)(N+\beta)\gamma} 2^{m(d+\beta\gamma)} I_2^m \Big)_{l\in\IZ}\Big\|_{X^{\frac{q}{\gamma},\alpha}}\\
    &=\Big\| \Big(\sum\limits_{m=l}^\infty 2^{-(m-l)(N+\beta-d)\gamma} 2^{m\beta\gamma}\sum\limits_{\substack{z\in\IZ^d \\ |z|_\infty\geq 2}} |z|_\infty^{-a\gamma}   \cM\big(|\widetilde{\Phi}_{m}f|^\gamma\mathbf{1}_{Q+\ell(Q)z} \big)(x+\ell(Q)z) \Big)_{l\in\IZ}\Big\|_{X^{\frac{q}{\gamma},\alpha}}\nonumber\\
    &\leq\sum\limits_{\substack{z\in\IZ^d \\ |z|_\infty\geq 2}} |z|_\infty^{-a\gamma}\Big\| \Big(\sum\limits_{m=l}^\infty 2^{-(m-l)(N+\beta-d)\gamma} 2^{m\beta\gamma}   \cM\big(|\widetilde{\Phi}_{m}f|^\gamma\mathbf{1}_{Q+\ell(Q)z} \big)(x+\ell(Q)z) \Big)_{l\in\IZ}\Big\|_{X^{\frac{q}{\gamma},\alpha}}.\nonumber
\end{align}
Observe that for a fixed $z\in\IZ^d$ the term
\begin{align*}
    \Big\| \Big(\sum\limits_{m=l}^\infty 2^{-(m-l)(N+\beta-d)\gamma} 2^{m\beta\gamma}   \cM\big(|\widetilde{\Phi}_{m}f|^\gamma\mathbf{1}_{Q+\ell(Q)z} \big)(x+\ell(Q)z) \Big)_{l\in\IZ}\Big\|_{X^{\frac{q}{\gamma},\alpha}}
\end{align*}
is of the same form as the term of Step~1.1. Thus, one can analogously repeat the same calculation as before to obtain the bound
\begin{align*}
    &\Big\| \Big(\sum\limits_{m=l}^\infty 2^{-(m-l)(N+\beta-d)\gamma} 2^{m\beta\gamma}   \cM\big(|\widetilde{\Phi}_{m}f|^\gamma\mathbf{1}_{Q+\ell(Q)z} \big)(x+\ell(Q)z) \Big)_{l\in\IZ}\Big\|_{X^{\frac{q}{\gamma},\alpha}}\lesssim \|\Phi_s\ast f\|_{\T^{\infty,q,r}_\beta}^\gamma
\end{align*}
for each $z\in \IZ^d$. Plugging this into \eqref{eq: F5} gives us
\begin{align*}
    \Big\| \Big(\sum\limits_{m=l}^\infty 2^{-(m-l)(N+\beta)\gamma} 2^{m(d+\beta\gamma)} I_2^m \Big)_{l\in\IZ}\Big\|_{X^{\frac{q}{\gamma},\alpha}}
    &\lesssim \sum\limits_{\substack{z\in\IZ^d \\ |z|_\infty\geq 2}} |z|^{-a\gamma}\|\Phi_s\ast f\|_{\T^{\infty,q,r}_\beta}^\gamma
    \lesssim\|\Phi_s\ast f\|_{\T^{\infty,q,r}_\beta}^\gamma,
\end{align*}
where we used $a\gamma>d$ in the last inequality. Combining this estimate with \eqref{eq: F6} of Step~1.1 and \eqref{eq: F4} yields
\begin{align*}
    I\lesssim \|\Phi_s\ast f\|_{\T^{\infty,q,r}_\beta}^\gamma.
\end{align*}
Finally, using this estimate in \eqref{eq: rht0} gives
\begin{align*}
     \|f\|_{\dot \F^\beta_{\infty,q}}\lesssim \|\Phi_s\ast f\|_{\T^{\infty,q,r}_\beta},
\end{align*}
which finishes Step~1.\\

\noindent \textbf{Step 2: $\dot \F^\beta_{\infty,q}(\IR^d)\subset X$.}
To attack this inclusion,  we first note that averages over Whitney boxes can be bounded by Peetre's maximal function, as discussed in Step~2 in the proofs of \cite[Thm.~2.2+2.4]{Auscher_Bechtel_Haardt}. More precisely, for fixed $t>0$ and $x\in \IR^d$, we have
\begin{align*}
     \bigg(\fint\limits_{\frac{t}{2}}^{t} \fint\limits_{B(x,t)} |(\Phi_{s}\ast f)(z) |^r \,\d z \d s \bigg)^\frac{1}{r} & \lesssim\sup\limits_{\frac{t}{2}<s \leq t} (\Phi_s^*f)_a(x).
\end{align*}
Thus, to show the inclusion $\dot \F^\beta_{\infty,q}(\IR^d)\subset X$, it suffices to show
\begin{align*}
     \sup\limits_{y\in\IR^d}\sup\limits_{\tau>0}\bigg(\int\limits_{0}^{\tau}\fint\limits_{B(y,\tau)}   t^{-q\beta}\Big(\sup\limits_{\frac{t}{2}<s \leq t} (\Phi_s^*f)_a(x)\Big)^q\,\frac{\d x\d t}{t}\bigg)^\frac{1}{q}\lesssim \|f\|_{\dot \F^\beta_{\infty,q}}
\end{align*}
or, equivalently by the John--Nirenberg-type property (Proposition~\ref{prop: John-Nirenberg continuous}),
\begin{align*}
     \sup\limits_{y\in\IR^d}\sup\limits_{\tau>0}\bigg(\fint\limits_{B(y,\tau)}\bigg(\int\limits_{0}^{\tau}   t^{-q\beta}\Big(\sup\limits_{\frac{t}{2}<s \leq t} (\Phi_s^*f)_a(x)\Big)^q\,\frac{\d t}{t}\bigg)^\frac{\alpha\gamma}{q}\,\d x \bigg)^\frac{1}{\alpha\gamma}\lesssim \|f\|_{\dot \F^\beta_{\infty,q}}
\end{align*}
for  $f\in \dot \F^\beta_{\infty,q}(\IR^d)$, where $\alpha>1$ is defined as in Step~1. However, this follows from \cite[Thm.\@ 3.2]{Ullrich2} together with the observation \cite[Rem.\@ 3.7]{Ullrich2}. 
\end{proof}

As an illustration of Theorem~\ref{thm: char of lifted Triebel data}, we characterize the space $\dot \F^\beta_{\infty,q}(\IR^d)$ for $\beta <0$ via the Gauss--Weierstrass semigroup. Specifically, by choosing $R=-1$, we see that the heat kernel $\Phi(x) = (4\pi)^{-\frac{d}{2}} e^{-\frac{|x|^2}{4}}$ satisfies both \eqref{property 1} and \eqref{property 2}. Moreover, for $\beta<0$, we realize the space $\dot \F^\beta_{\infty,q}(\IR^d)$ within $\cS'(\IR^d)$, see Remark~\ref{rem:characterization}. 

\begin{proposition}[Gauss--Weierstrass characterization of $\dot \F^\beta_{\infty,q}$]
\label{prop: Gauss--Weierstrass char}
Let $0<q< \infty$ and $\beta<0$. A tempered distribution $f\in\cS'(\IR^d)$ belongs to $\dot \F^\beta_{\infty,q}(\IR^d)$ if and only if $e^{t^2\Delta}f \in \T^{\infty,q,r}_\beta$ for any $0<r\leq \infty$, and there holds the equivalence of (quasi-)norms
\begin{align*}
    \|f\|_{\dot \F^\beta_{\infty,q}} \simeq \|e^{t^2\Delta}f \|_{\T^{\infty,q,r}_\beta}.
\end{align*}

\end{proposition}

\section{Functional analytic properties of tent spaces}
\label{sec: tent spaces}

\noindent As we saw in the previous section, homogeneous endpoint Triebel--Lizorkin spaces $\dot\F^\beta_{\infty,q}(\IR^d)$ can be characterized via extensions to the upper half-space $\IR^{d+1}_+$ belonging to a certain extension space. In the case of homogeneous Triebel--Lizorkin spaces $\dot\F^\beta_{p,q}(\IR^d)$ for $0<p<\infty$, analogous characterizations were shown in \cite[Sec.\@ 2]{Auscher_Bechtel_Haardt}. More precisely, the extension spaces of $ \dot\F^{\beta}_{p,q}(\IR^d)$ are characterized in terms of the four-parameter tent spaces $\T^{p,q,r}_{\beta}$, see Definition~\ref{def: tent spaces} below. Thus, Theorem~\ref{thm: char of lifted Triebel data} can be seen as an extension of their characterization to the case $p=\infty$. The purpose of this section is to extend the scale $\T^{p,q,r}_{\beta}$ to its endpoint $p=\infty$ and study the full scale from a function space theoretical point of view. First, we provide concrete definitions of these spaces and compare them to known spaces from the literature and collect basic properties such as completeness, density, duality and complex interpolation (Section~\ref{subsec: basic prop}). Afterwards, we develop a discrete descriptions of these spaces (Section~\ref{sec: dyadic char}), which are powerful tools for duality theory in the quasi-Banach range (Section~\ref{sec: duality}), embedding theory (Section~\ref{sec: embeddings}) and real interpolation theory (Section~\ref{sec: real interpolation}).

\subsection{Definitions, relations to known spaces and basic properties}
\label{subsec: basic prop}

Here, we recall the definitions of the function spaces $\T^{p,q,r}_\beta$ and $\Z^{p,q,r}_\beta$ that appear in \cite[Sec.\@ 3]{Auscher_Bechtel_Haardt} and extend the scale of tent spaces to its endpoint $p=\infty$. Moreover, we provide a link to existing function spaces from the literature and collect basic properties, including completeness, density, duality and complex interpolation.

\begin{definition}[Tent spaces]
\label{def: tent spaces}
    For $0<p< \infty$, $0<q,r\leq \infty$ and $\beta \in\IR$ define the space $\T^{p,q,r}_\beta$ as the set of all measurable functions $f$ on $\IR^{d+1}_+$ such that
    \begin{align*}
        \|f\|_{\T^{p,q,r}_\beta} \coloneqq \bigg\|\bigg( \int\limits_0^\infty \bigg(\fint\limits_{\frac{t}{2}}^t \fint\limits_{B(\cdot,t)} |s^{-\beta}f(s,y)|^r \,\d y \d s\bigg)^\frac{q}{r}\,\frac{\d t}{t}\bigg)^\frac{1}{q} \bigg\|_{\L^p}< \infty
    \end{align*}
    with the usual replacement of an integral by an $\esssup$ if the corresponding parameters $q,r$ are infinite. Moreover, we define the space $\T^{\infty,q,r}_\beta$ as the set of all measurable functions $f$ on $\IR^{d+1}_+$ such that
    \begin{align*}
        \|f\|_{\T^{\infty,q,r}_\beta} \coloneqq \sup\limits_{\tau>0} \sup\limits_{y\in\IR^d} 
        \bigg(\int\limits_{0}^\tau  \fint\limits_{B(y,\tau)} \bigg(\fint\limits_{\frac{t}{2}}^t \fint\limits_{B(x,t)} |s^{-\beta}f(s,z)|^r \,\d z \d s\bigg)^\frac{q}{r}\,\frac{\d x \d t}{t}\bigg)^\frac{1}{q} < \infty,
    \end{align*}
    if the parameter $q$ is finite, and
    \begin{align*}
        \|f\|_{\T^{\infty,\infty,r}_\beta} \coloneqq \sup\limits_{t>0} \sup\limits_{x\in\IR^d} 
        \bigg(\fint\limits_{\frac{t}{2}}^t \fint\limits_{B(x,t)} |s^{-\beta}f(s,z)|^r \,\d z \d s\bigg)^\frac{1}{r}< \infty
    \end{align*}
    with the same replacements in both cases if $r=\infty$, as before.
\end{definition}

\begin{definition}[$\Z$-spaces]
\label{def: Z-space}
    For $0<p,q,r\leq \infty$ and $\beta\in\IR$ define the space $\Z^{p,q,r}_\beta$ as the set of all measurable functions $f$ on $\IR^{d+1}_+$ such that
    \begin{align*}
        \|f\|_{\Z^{p,q,r}_\beta} \coloneqq \bigg(\int\limits_{0}^\infty \bigg\|\bigg(\fint\limits_{\frac{t}{2}}^t \fint\limits_{B(\cdot,t)} |s^{-\beta}f(s,y)|^r \, \d y \d s\bigg)^\frac{1}{r} \bigg\|_{\L^p}^q\,\frac{\d t}{t}\bigg)^\frac{1}{q}<\infty,
    \end{align*}
    with the usual replacement of an integral by an $\esssup$ if the corresponding parameters $q,r$ are infinite. 
\end{definition}

\begin{remark}
    In \cite[Rem.\@ 3.2]{Auscher_Bechtel_Haardt}, the authors illustrated that the spaces $\Z^{p,q,r}_\beta$ and $\T^{p,q,r}_\beta$ (for $p<\infty$) are generalizations of the well-known tent spaces $\T^{p,q}_\beta$ and $\Z$-spaces $\Z^{p,q}_\beta$. In the case $p=\infty$, our introduced spaces $\T^{\infty,q,r}_\beta$ are generalizations of the endpoint tent spaces $\T^{\infty,q}_\beta$ (see \cite[Def.\@ 1.1]{Huang}) in the sense that $\T^{\infty,q,q}_\beta = \T^{\infty,q}_\beta$ with equivalence of (quasi-)norms.
\end{remark}

In \cite{Huang}, the author introduced weighted tent spaces with Whitney average, denoted by $\T^{p,r}_{q,\beta}$. They consists of all measurable functions $f$ on $\IR^{d+1}_+$ such that
\begin{align*}
    &\|f\|_{\T^{p,r}_{q,\beta}} \coloneqq \bigg(\int\limits_{\IR^d} \bigg( \int\limits_0^\infty\fint\limits_{B(x,t)} \bigg(\fint\limits_{\frac{t}{2}}^{2t} \fint\limits_{B(y,t)} |s^{-\beta}f(s,z)|^r \,\d z \d s\bigg)^\frac{q}{r}\,\frac{\d y\d t}{t}\bigg)^\frac{p}{q}\,\d x \bigg)^\frac{1}{p}\qquad \text{(if $p<\infty$),}\\
    &\|f\|_{\T^{\infty,r}_{q,\beta}} \coloneqq \sup\limits_{\tau>0} \sup\limits_{y\in\IR^d} 
    \bigg(\iint\limits_{\hat{B}(y,\tau)}\bigg(\fint\limits_{\frac{t}{2}}^{2t} \fint\limits_{B(x,t)} |s^{-\beta}f(s,z)|^r \,\d z \d s\bigg)^\frac{q}{r}\,\frac{\d x \d t}{t}\bigg)^\frac{1}{q} 
    \end{align*}
are finite, where $\hat{B}(y,\tau) \coloneqq  \{(z,s)\in\IR^{d+1}_+: B(z,s)\subset B(y,\tau)\}$ denotes the tent with base $B(y,\tau)$.
We observe that the spaces $\T^{p,r}_{q,\beta}$ have an additional average compared to the (quasi-)norms of $\T^{p,q,r}_\beta$. However, as demonstrated in \cite[Lem.\@ 3.3]{Auscher_Bechtel_Haardt}, these spaces are equivalent for $p<\infty$. In the endpoint case $p=\infty$, a standard covering argument allows for the comparison of tents $\hat{B}(y,\tau)$ with Carleson boxes $(0,\tau)\times B(y,\tau)$, and a change of parameters in the average integrals, establishing the equivalence of $\T^{\infty,r}_{q,\beta}$ and $\T^{\infty,q,r}_\beta$.

While the spaces $\T^{p,r}_{q,\beta}$ and $\T^{p,q,r}_\beta$ essentially describe the same scale of function spaces, Definition~\ref{def: tent spaces} is more natural in the context for extension spaces for the Besov and Triebel--Lizorkin scales. Just as Besov and Triebel--Lizorkin spaces are distinguished primarily by the order of their underlying norms, this same structural symmetry is reflected in our definitions of $\Z^{p,q,r}_\beta$ and $\T^{p,q,r}_\beta$. Beyond this conceptual consistency, our definition reduces notational overhead and simplifies practical calculations. 

However, we first leverage the established properties of $\T^{p,r}_{q,\beta}$ from \cite{Huang}, transferring them to our tent spaces $\T^{p,q,r}_\beta$ via the aforementioned equivalence. In particular, we present three results that follow as direct consequences. First, we summarize basic properties of $\T^{p,q,r}_\beta$, such as completeness, density and separability in the following statement. The corresponding properties for $\T^{p,r}_{q,\beta}$-spaces can be found in \cite[Sec.\@ 1]{Huang}.

\begin{proposition}
\label{prop: basic props}
    Let $0<p,q,r\leq \infty$ and $\beta \in\IR$. Then the space $\T^{p,q,r}_\beta$ is a (quasi-)Banach space. Moreover, $\T^{p,q,r}_\beta$ is separable and the space of compactly supported functions in $\L^r\big(\IR^{d+1}_+,\frac{\d y \d s}{s}\big)$ is a dense subspace of $\T^{p,q,r}_\beta$ if $\max\{p,q,r\}<\infty$. Finally, for $K\subset\IR^{d+1}_+$ compact and $f\in\L^0(\IR^{d+1}_+)$, we have the equivalence
    \begin{align*}
        \|f\|_{\L^r\big(K,\frac{\d y\d s}{s}\big)}\simeq \|\mathbf{1}_Kf\|_{\T^{p,q,r}_\beta}
    \end{align*}
    where the implicit constants depend on $p, q, r, \beta, d$ and $K$.
\end{proposition}

Although the last property of the above statement was not proven in \cite{Huang}, it follows the same reasoning as in  \cite[Lem.\@ 3.8]{Auscher_Bechtel_Haardt} and will be omitted here. Next, we present a duality result in the Banach range $1\leq p,q,r<\infty$. The corresponding result for $\T^{p,r}_{q,\beta}$-spaces was proven in \cite[Thm.\@ 5.4]{Huang}.

\begin{proposition}
\label{prop: duality Banach range}
    Let $1\leq p,q,r<\infty$ and $\beta\in \IR$. Then, we have
    \begin{align*}
        \int\limits_0^\infty \int\limits_{\IR^d} |f(s,y)| \cdot |g(s,y)| \,\frac{\d y \d s}{s} \leq  \|f\|_{\T^{\vphantom{p',q',r'}p,q,r}_\beta } \|g\|_{\T^{p',q',r'}_{-\beta} }
    \end{align*}
    for all measurable functions $f,g$. Moreover, we can identify $(\T^{p,q,r}_\beta)' \simeq \T^{p',q',r'}_{-\beta}$ with equivalent norms via the $\L^2$ duality pairing.
\end{proposition}

Finally, \cite[Thm.\@ 4.3]{Huang} provides a complete complex interpolation theory for $\T^{p,r}_{q,\beta}$-spaces. The corresponding result for $\T^{p,q,r}_\beta$-spaces reads as follows.

\begin{proposition}
\label{prop: complex int}
    Let $0< p_0,p_1,q_0,q_1, r_0,r_1\leq \infty$ be such that $\max\{p_0, q_0 , r_0\}<\infty$ or $\max\{p_1, q_1 ,r_1\}<\infty$. Furthermore, let $\beta_0,\beta_1\in\IR$ and $\theta\in (0,1)$. Define
    \begin{align*}
        \frac{1}{p_\theta} = \frac{1-\theta}{p_0} + \frac{\theta}{p_1},\quad \frac{1}{q_\theta} = \frac{1-\theta}{q_0} + \frac{\theta}{q_1},\quad \frac{1}{r_\theta} = \frac{1-\theta}{r_0} + \frac{\theta}{r_1}, \quad \beta_\theta = (1-\theta)\beta_0 + \theta \beta_1.
    \end{align*}
    Then, we have
    \begin{align*}
        [\T^{p_0,q_0,r_0}_{\beta_0} , \T^{p_1,q_1,r_1}_{\beta_1} ]_{\theta} = \T^{p_\theta, q_\theta,r_\theta}_{\beta_\theta}
    \end{align*}
    with equivalent (quasi-)norms.
\end{proposition}

\begin{remark}
    Without introducing complex interpolation spaces properly, we would like to mention that the spaces in Proposition~\ref{prop: complex int} should be interpreted according to Kalton and Mitrea's complex interpolation method \cite{Kalton_Mitrea}. This method is well-defined for all quasi-Banach couples and agrees with the usual complex interpolation method of Calderón \cite{Calderón} (see also \cite[Chp.\@ 4]{Berg_Löfström}) on couples of Banach spaces.
\end{remark}

We continue with a convexity property for tent spaces, which will be utilized at several stages of our analysis.

\begin{lemma}
\label{lem: convex tent space}
    Let $f\in \L^0(\IR^{d+1}_+)$ and $M>0$. Then, for $0<p,q,r\leq \infty$ and $\beta\in \IR$, we have
\begin{align}
\label{eq: convex reduction}
    \|f\|_{\T^{p,q,r}_\beta} = \||f|^M\|^\frac{1}{M}_{\T^{p/M,q/M,r/M}_{M\beta}}.
\end{align}
In particular, $f\in \T^{p,q,r}_\beta$ iff $|f|^M \in \T^{p/M,q/M,r/M}_{M\beta}$.
\end{lemma}

To conclude this section, we present a \enquote{change of angle} formulas for $\T^{p,q,r}_\beta$-spaces. These formulas quantitatively describe the change of Whitney boxes in the norm of $\T^{p,q,r}_\beta$-spaces. They are motivated by the corresponding change of angle formulas for tent spaces $\T^{p,q}_\beta$, see \cite{Auscher}.

\begin{lemma}
\label{lem: change of angle p<infty}
    Let $0<p<\infty$, $0<q,r\leq \infty$ and $\beta\in\IR$. Then for $\lambda> 1$ and $0<M<\min\{p,q,r\}$ we have
    \begin{align*}
         \bigg\|\bigg( \int\limits_0^\infty \bigg(\fint\limits_{\frac{t}{2}}^t t^{-d}\int\limits_{B(\cdot,\lambda t)} |s^{-\beta}f(s,y)|^r \,\d y \d s\bigg)^\frac{q}{r}\,\frac{\d t}{t}\bigg)^\frac{1}{q}\bigg\|_{\L^p}\lesssim \lambda^{\frac{d}{M}}\|f\|_{\T^{p,q,r}_\beta},
    \end{align*}
    where the implicit constant depends on $d,M,p,q,r$.
\end{lemma}

\begin{proof}
    For simplicity, we assume $q,r$ to be finite and $\beta=0$. We divide the proof into two steps.\\

    \noindent \textbf{Step 1:} Assume $ 1< p,q,r<\infty$. By an averaging trick, we have the identity
    \begin{align*}
        f(s,y) = \fint\limits_s^{2s} \fint\limits_{B(y,\tau)} f(s,y)\,\d z \d \tau= \fint\limits_s^{2s} \fint\limits_{B(y,\lambda\tau)} \lambda^d\mathbf{1}_{W(\tau,z)}(s,y) f(s,y)\,\d z \d \tau
    \end{align*}
    for all $(s,y)\in\IR^{d+1}_+$. Set 
    \begin{align*}
        g(\tau,z,s,y) \coloneqq \lambda^d\mathbf{1}_{W(\tau,z)}(s,y) f(s,y)
    \end{align*}
    and 
    \begin{align*}
        R_\lambda(g)(s,y) \coloneqq \fint\limits_s^{2s} \fint\limits_{B(y,\lambda\tau)} g(\tau,z,s,y)\,\d z \d \tau.
    \end{align*}
    Moreover, fix $(t,x)\in \IR^{d+1}_+$. Observe that for $\frac{t}{2}< s <t$ and $s<\tau <2s$ we have $(s,2s) \times B(y,\lambda \tau) \subset (\frac{t}{2}, 2t) \times B(x,3\lambda t)$ for all $y\in B(x,\lambda t)$. Minkowski's integral inequality then yields
    \begin{align*}
        \bigg(\fint\limits_{\frac{t}{2}}^{t} t^{-d}\int\limits_{B(x,\lambda t)} |f(s,y)|^{r} \,\d y\d s \bigg)^\frac{1}{{r}}&=\bigg(\fint\limits_{\frac{t}{2}}^{t} t^{-d}\int\limits_{B(x,\lambda t)} |R_\lambda(g)(s,y)|^{r} \,\d y\d s \bigg)^\frac{1}{{r}}\\
        &\lesssim \bigg(\fint\limits_{\frac{t}{2}}^{t} t^{-d}\int\limits_{B(x,\lambda t)} \bigg|\fint\limits_{\frac{t}{2}}^{2t} \fint\limits_{B(x,3\lambda t)} |g(\tau, z, s, y)| \, \d z \d \tau\bigg|^{r} \,\d y\d s \bigg)^\frac{1}{r}\\
        &\lesssim  \fint\limits_{\frac{t}{2}}^{2t} \fint\limits_{B(x,3\lambda t)}\bigg(\fint\limits_{\frac{t}{2}}^{t} t^{-d}\int\limits_{B(x,\lambda t)} | g(\tau, z, s, y)|^{r} \,\d y\d s  \,\bigg)^\frac{1}{r} \d z \d \tau\\
        &\lesssim \fint\limits_{\frac{t}{2}}^{2t} \fint\limits_{B(x,3\lambda t)}\bigg(\fint\limits_{\frac{\tau}{4}}^{2\tau} \tau^{-d}\int\limits_{B(z,8\lambda \tau)} |g(\tau, z, s, y)|^{r} \,\d y\d s  \,\bigg)^\frac{1}{r} \d z \d \tau\\
        &= \fint\limits_{\frac{t}{2}}^{2t} \fint\limits_{B(x,3\lambda t)} h(\tau,z)\, \d z \d \tau,
    \end{align*}
    where we set
    \begin{align}
        \label{eq:duality_def_h}
        h(\tau,z) \coloneqq \bigg(\fint\limits_{\frac{\tau}{4}}^{2\tau} \tau^{-d}\int\limits_{B(z,8\lambda \tau)} | g(\tau, z, s, y)|^{r} \,\d y\d s  \,\bigg)^\frac{1}{r}.
    \end{align}
    Taking first the $\L^q$ norm in $t$ and then the $\L^p$ norm in $x$ on both sides yields the estimate
    \begin{align*}
         \bigg\|\bigg( \int\limits_0^\infty \bigg(\fint\limits_{\frac{t}{2}}^t t^{-d}\int\limits_{B(\cdot,\lambda t)} |f(s,y)|^r \,\d y \d s\bigg)^\frac{q}{r}\,\frac{\d t}{t}\bigg)^\frac{1}{q}\bigg\|_{\L^p} \lesssim \bigg\| \bigg(\int\limits_0^\infty \bigg(\fint\limits_{\frac{t}{2}}^{2t} \fint\limits_{B(x,3\lambda t)} h(\tau,z)\, \d z \d \tau\bigg)^q\,\frac{\d t}{t} \bigg)^\frac{1}{q}\bigg\|_{\L^p}.
    \end{align*}
    Splitting the time integral dyadically on the right-hand side implies
    \begin{align*}
         &\bigg\|\bigg( \int\limits_0^\infty \bigg(\fint\limits_{\frac{t}{2}}^t t^{-d}\int\limits_{B(\cdot,\lambda t)} |f(s,y)|^r \,\d y \d s\bigg)^\frac{q}{r}\,\frac{\d t}{t}\bigg)^\frac{1}{q}\bigg\|_{\L^p} \\
         &\lesssim \bigg\| \bigg(\sum\limits_{k\in\IZ}\, \int\limits_{2^{-k-1}}^{2^{-k}} \bigg(\fint\limits_{\frac{t}{2}}^{2t} \fint\limits_{B(x,3\lambda t)} h(\tau,z)\, \d z \d \tau\bigg)^q\,\frac{\d t}{t} \bigg)^\frac{1}{q}\bigg\|_{\L^p}\\
        &\lesssim \bigg\| \bigg(\sum\limits_{k\in\IZ}\, \int\limits_{2^{-k-1}}^{2^{-k}} \bigg(\fint\limits_{2^{-k-2}}^{2^{-k+1}} \fint\limits_{B(x,3\lambda 2^{-k})} h(\tau,z)\, \d z \d \tau\bigg)^q\,\frac{\d t}{t} \bigg)^\frac{1}{q}\bigg\|_{\L^p}\\
        &\lesssim \bigg\| \bigg(\sum\limits_{k\in\IZ}\big(\cM(f_k)(x)\big)^q \bigg)^\frac{1}{q}\bigg\|_{\L^p},
    \end{align*}
    where $\cM$ is the Hardy-Littlewood maximal operator and
    \begin{align*}
        f_k(z)\coloneqq \fint\limits_{2^{-k-2}}^{2^{-k+1}}  h(\tau,z)\,  \d \tau.
    \end{align*}
    Since $1<p,q<\infty$, we may apply Fefferman--Stein's vector valued inequality (see for example \cite[Thm.\@ 2.1]{Ullrich1}) and Jensen's inequality to obtain
    \begin{align*}
        \bigg\|\bigg( \int\limits_0^\infty \bigg(\fint\limits_{\frac{t}{2}}^t t^{-d}\int\limits_{B(\cdot,\lambda t)} |f(s,y)|^r \,\d y \d s\bigg)^\frac{q}{r}\,\frac{\d t}{t}\bigg)^\frac{1}{q}\bigg\|_{\L^p} &\lesssim \bigg\| \bigg(\sum\limits_{k\in\IZ}|f_k(x)|^q \bigg)^\frac{1}{q}\bigg\|_{\L^p}\\
        &\lesssim \bigg\| \bigg(\sum\limits_{k\in\IZ}\bigg( \fint\limits_{2^{-k-2}}^{2^{-k+1}}  h(\tau,x)\,  \d \tau \bigg)^q \bigg)^\frac{1}{q}\bigg\|_{\L^p}\\
        &\lesssim \bigg\| \bigg(\sum\limits_{k\in\IZ}\, \fint\limits_{2^{-k-2}}^{2^{-k+1}}  |h(\tau,x)|^q\,  \d \tau \bigg)^\frac{1}{q}\bigg\|_{\L^p}\\
        &\lesssim \bigg\| \bigg(\int\limits_{0}^{\infty}  |h(\tau,x)|^q\, \frac{\d \tau }{\tau}\bigg)^\frac{1}{q}\bigg\|_{\L^p}.
    \end{align*}
    Finally, by the definition of the auxiliary functions $h$ and $g$, we conclude
    \begin{align*}
        &\bigg\|\bigg( \int\limits_0^\infty \bigg(\fint\limits_{\frac{t}{2}}^t t^{-d}\int\limits_{B(\cdot,\lambda t)} |f(s,y)|^r \,\d y \d s\bigg)^\frac{q}{r}\,\frac{\d t}{t}\bigg)^\frac{1}{q}\bigg\|_{\L^p} \\
        &\lesssim \bigg\| \bigg(\int\limits_{0}^{\infty}  \bigg(\fint\limits_{\frac{\tau}{4}}^{2\tau} \tau^{-d}\int\limits_{B(x,8\lambda \tau)} | g(\tau, x, s, y)|^{r} \,\d y\d s  \bigg)^\frac{q}{r}\, \frac{\d \tau }{\tau}\bigg)^\frac{1}{q}\bigg\|_{\L^p}\\
        &\lesssim \lambda^{d}\|f\|_{\T^{p,q,r}_0}
    \end{align*}
    for some implicit constant depending on $p,q,r$ and $d$.\\

    \noindent \textbf{Step 2:} Let $ 0< p,q,r<\infty$ and $0<M< \min\{p,q,r\}$. Then, by Lemma~\ref{lem: convex tent space} and Step~1, we obtain

    \begin{align*}
        &\bigg\|\bigg( \int\limits_0^\infty \bigg(\fint\limits_{\frac{t}{2}}^t t^{-d}\int\limits_{B(\cdot,\lambda t)} |f(s,y)|^r \,\d y \d s\bigg)^\frac{q}{r}\,\frac{\d t}{t}\bigg)^\frac{1}{q}\bigg\|_{\L^p}\\
        &=\bigg\|\bigg( \int\limits_0^\infty \bigg(\fint\limits_{\frac{t}{2}}^t t^{-d}\int\limits_{B(\cdot,\lambda t)} ||f(s,y)|^M|^\frac{r}{M} \,\d y \d s\bigg)^\frac{qM}{rM}\,\frac{\d t}{t}\bigg)^\frac{M}{q}\bigg\|_{\L^\frac{p}{M}}^\frac{1}{M}\\
        &\leq \Big(C\lambda^d\||f|^M\|_{\T^{p/M,q/M,r/M}_0}\Big)^\frac{1}{M}\\
        &= C^\frac{1}{M}\lambda^\frac{d}{M} \|f\|_{\T^{p,q,r}_0}
    \end{align*}
    for some constant $C>0$ depending on $d,M,p,q,r$.
\end{proof}

Similarly to the previous proof, the following change of angle formula holds for the endpoint space ${\T^{\infty,q,r}_\beta}$.

\begin{lemma}
\label{lem: change of angle p=infty}
     Let $0<q,r\leq \infty$ and $\beta\in\IR$. Then for $\lambda> 1$ and $0<M<\min\{q,r\}$ we have
    \begin{align*}
        \sup\limits_{\tau>0} \sup\limits_{y\in\IR^d} 
        \bigg(\int\limits_{0}^\tau  \fint\limits_{B(y,\tau)} \bigg(\fint\limits_{\frac{t}{2}}^t t^{-d}\int\limits_{B(x,\lambda t)} |s^{-\beta}f(s,z)|^r \,\d z \d s\bigg)^\frac{q}{r}\,\frac{\d x \d t}{t}\bigg)^\frac{1}{q} \lesssim \lambda^{\frac{d}{M}}\|f\|_{\T^{\infty,q,r}_\beta},
    \end{align*}
    where the implicit constant depends on $d,M,q,r$.
\end{lemma}

\begin{remark}
    Unlike the classical change of angle formula for $\T^{p,q}_\beta$, the growth rates of Lemma~\ref{lem: change of angle p<infty} and \ref{lem: change of angle p=infty} are suboptimal. Specifically, our implicit constants blow up as $M$ approaches its upper limit due to a maximal operator argument.
\end{remark}

As an illustration, we can use Lemma~\ref{lem: change of angle p<infty} and \ref{lem: change of angle p=infty} to show that the definition of the $\T^{p,q,r}_\beta$-spaces is independent of the underlying Whitney box.
\begin{proposition}
\label{prop: independent Whitney cube}
    Let $0<p,q,r\leq \infty$ and $\beta\in \IR$. If $p<\infty$, then we have for all $f\in \T^{p,q,r}_\beta$, $0<a<b<\infty$ and $c>0$ that
    \begin{align*}
        \bigg\|\bigg(\int\limits_0^\infty\bigg(  \fint\limits_{at}^{bt} \fint\limits_{B(\cdot, ct)} |s^{-\beta}f(s,y)|^r \, \d y \d s\bigg)^\frac{q}{r} \,\frac{\d t}{t}\bigg)^\frac{1}{q}\bigg\|_{\L^p}\simeq \|f\|_{\T^{p,q,r}_\beta},
    \end{align*}
    where the implicit constants depend on $a, b, c, p, q, r$, and the dimension $d$. Furthermore, for $f\in \T^{\infty,q,r}_\beta$ we have
    \begin{align*}
        \sup\limits_{\tau>0} \sup\limits_{y\in\IR^d} 
        \bigg(\int\limits_{0}^\tau  \fint\limits_{B(y,\tau)}\bigg(  \fint\limits_{at}^{bt} \fint\limits_{B(x, ct)} |s^{-\beta}f(s,y)|^r \, \d y \d s\bigg)^\frac{q}{r}\,\frac{\d t \d x}{t}\bigg)^\frac{1}{q}\simeq \|f\|_{\T^{\infty,q,r}_\beta},
    \end{align*}
    where the implicit constants depend on $a, b, c, q, r$, and the dimension $d$. 
\end{proposition}

The proofs of both cases follow the same line of reasoning as in \cite[Prop.\@ 3.7]{Auscher_Bechtel_Haardt} for $\Z^{p,q,r}_\beta$-spaces. We will omit them here and leave the details to the reader.

\subsection{Discrete descriptions}
\label{sec: dyadic char}

In this section, we provide several equivalent characterizations of $\T^{p,q,r}_\beta$, which turn out to be very useful in the upcoming sections. We start with the following discrete characterization.

\begin{proposition}
\label{prop: dyadic char}
    Let $0<p<\infty$, $0<q,r\leq \infty$ and $\beta\in\IR$. Then, for $f\in \T^{p,q,r}_\beta$ we have 
    \begin{align*}
        \|f\|_{\T^{p,q,r}_\beta} \simeq \bigg\| \Big(\sum\limits_{Q\in\square} \mathbf{1}_{Q}(\cdot)|Q|^{-\frac{\beta q}{d}} \|f\|_{\L^r\big(\bar Q , \frac{\d y \d s}{s^{d+1}}\big)}^q \Big)^\frac{1}{q} \bigg\|_{\L^p}
    \end{align*}
    with implicit constants depending only on $d,p,q,r,\beta$.
\end{proposition}

\begin{proof}
    To simplify notation, we concentrate on the case $q,r<\infty$. The case of infinite parameters follows analogously.\\
    
    \noindent \textbf{Step 1: We show \enquote{$\lesssim$}.}
    Our starting point is the equation
    \begin{align*}
        \|f\|_{\T^{p,q,r}_\beta}  &\simeq \bigg\| \bigg( \int\limits_0^\infty \bigg(\int\limits_{\frac{t}{2}}^t \int\limits_{B(\cdot,t)} |s^{-\beta}f(s,y)|^r \,\frac{\d y \d s}{s^{d+1}}\bigg)^\frac{q}{r}\,\frac{\d t}{t}\bigg)^\frac{1}{q}\bigg\|_{\L^p}.
    \end{align*}
    Fix some $x\in\IR^d$. Then for each $k\in\IZ$, there exists a unique $Q_k\in\square_k$ with $x\in Q_k$. Hence, splitting the $\L^q$ integral dyadically, we get
    \begin{align*}
        &\bigg( \int\limits_0^\infty \bigg(\int\limits_{\frac{t}{2}}^{t} \int\limits_{B(x,t)} |s^{-\beta}f(s,y)|^r \,\frac{\d y \d s}{s^{d+1}}\bigg)^\frac{q}{r}\,\frac{\d t}{t}\bigg)^\frac{1}{q}\\
        &=\bigg( \sum\limits_{k\in\IZ} \mathbf{1}_{Q_k}(x)\int\limits_{\ell(Q_k)}^{2\ell(Q_k)} \bigg(\int\limits_{\frac{t}{2}}^{t} \int\limits_{B(x,t)} |s^{-\beta}f(s,y)|^r \,\frac{\d y \d s}{s^{d+1}}\bigg)^\frac{q}{r}\,\frac{\d t}{t}\bigg)^\frac{1}{q}.
    \intertext{Observe that, if $t \in (\ell(Q_k), 2 \ell(Q_k)]$ and $x\in Q_k$, then $B(x,t)\subset 5Q_k$. Thus, we further estimate}
        &\leq \bigg( \sum\limits_{k\in\IZ} \mathbf{1}_{Q_k}(x)\int\limits_{\ell(Q_k)}^{2\ell(Q_k)} \bigg(\int\limits_{\frac{\ell(Q_k)}{2}}^{2\ell(Q_k)} \int\limits_{5Q_k} |s^{-\beta}f(s,y)|^r \,\frac{\d y \d s}{s^{d+1}}\bigg)^\frac{q}{r}\,\frac{\d t}{t}\bigg)^\frac{1}{q}\\
        &\lesssim \bigg( \sum\limits_{k\in\IZ} \mathbf{1}_{Q_k}(x) \bigg(\int\limits_{\frac{\ell(Q_k)}{2}}^{2\ell(Q_k)} \int\limits_{5Q_k} |s^{-\beta}f(s,y)|^r \,\frac{\d y \d s}{s^{d+1}}\bigg)^\frac{q}{r}\bigg)^\frac{1}{q}.
    \end{align*}
    Next, we cover the dilated Whitney box
    \begin{align*}
        \Big(\frac{\ell(Q_k)}{2},2\ell(Q_k)\Big] \times 5Q_k 
    \end{align*}
    by a finite number of Whitney boxes $(\bar Q_k^j)_{j\in J}$ associated to certain dyadic cubes $( Q_k^j)_{j\in J}$, see Section~\ref{subsec: notation} for a review of notation. Indeed, there exist dyadic cubes $(Q_k^j)_{j\in J}$ with following the properties:
    \begin{enumerate}
        \item $|J|$ is bounded by a constant depending only on $d$,
        
        \item $Q^j_k \in \square_k\cup \square_{k+1}$,

        \item $\Big(\frac{\ell(Q_k)}{2},2\ell(Q_k)\Big] \times 5Q_k \subset \bigcup_{j\in J} \bar Q^j_k$. 
    \end{enumerate}
    Using this covering yields
    \begin{align*}
        &\bigg( \sum\limits_{k\in\IZ} \mathbf{1}_{Q_k}(x) \bigg(\int\limits_{\frac{\ell(Q_k)}{2}}^{2\ell(Q_k)} \int\limits_{5Q_k} |s^{-\beta}f(s,y)|^r \,\frac{\d y \d s}{s^{d+1}}\bigg)^\frac{q}{r}\bigg)^\frac{1}{q}\\
        &\leq \bigg( \sum\limits_{k\in\IZ} \mathbf{1}_{Q_k}(x) \bigg(\sum\limits_{j\in J}\iint\limits_{\bar Q^j_k} |s^{-\beta}f(s,y)|^r \,\frac{\d y \d s}{s^{d+1}}\bigg)^\frac{q}{r}\bigg)^\frac{1}{q}\\
        &\lesssim  \bigg( \sum\limits_{k\in\IZ}  \mathbf{1}_{Q_k}(x)\sum\limits_{j\in J}\bigg(\iint\limits_{\bar Q^j_k} |s^{-\beta}f(s,y)|^r \,\frac{\d y \d s}{s^{d+1}}\bigg)^\frac{q}{r}\bigg)^\frac{1}{q}\\
        &\lesssim \Big( \sum\limits_{k\in\IZ}\sum\limits_{j\in J}  \mathbf{1}_{Q^j_k}(x) \|s^{-\beta} f(s,y)\|_{\L^r\big(\bar Q_k^j , \frac{\d y \d s}{s^{d+1}}\big)}^q\Big)^\frac{1}{q}.
    \end{align*}
    Since $Q^j_k \in \square_k\cup \square_{k+1}$, we take the sum over all $Q\in \square_k\cup \square_{k+1}$ to get
    \begin{align*}
        &\leq \Big( \sum\limits_{k\in\IZ}\sum\limits_{Q\in \square_k\cup \square_{k+1}} \mathbf{1}_{Q}(x)\|s^{-\beta} f(s,y)\|_{\L^r\big(\bar Q , \frac{\d y \d s}{s^{d+1}}\big)}^q\Big)^\frac{1}{q}.
    \intertext{By splitting the second series into two, one running over $Q\in\square_k$ and the other over $Q\in\square_{k+1}$, we can use an index shift on the latter to obtain}
        &\lesssim\Big( \sum\limits_{k\in\IZ}\sum\limits_{Q\in \square_k} \mathbf{1}_{Q}(x)\|s^{-\beta} f(s,y)\|_{\L^r\big(\bar Q , \frac{\d y \d s}{s^{d+1}}\big)}^q\Big)^\frac{1}{q}\\
        &= \Big( \sum\limits_{Q\in\square} \mathbf{1}_{Q}(x)\|s^{-\beta} f(s,y)\|_{\L^r\big(\bar Q , \frac{\d y \d s}{s^{d+1}}\big)}^q\Big)^\frac{1}{q}\\
        &\simeq \Big( \sum\limits_{Q\in\square} \mathbf{1}_{Q}(x)|Q|^{-\frac{\beta q}{d}}\|f\|_{\L^r\big(\bar Q , \frac{\d y \d s}{s^{d+1}}\big)}^q\Big)^\frac{1}{q}.
    \end{align*}
    In summary, we have shown
    \begin{align*}
        \bigg( \int\limits_0^\infty \bigg(\int\limits_{\frac{t}{2}}^{t} \int\limits_{B(x,t)} |s^{-\beta}f(s,y)|^r \,\frac{\d y \d s}{s^{d+1}}\bigg)^\frac{q}{r}\,\frac{\d t}{t}\bigg)^\frac{1}{q} \lesssim \Big( \sum\limits_{Q\in\square} \mathbf{1}_{Q}(x)|Q|^{-\frac{\beta q}{d}}\|f\|_{\L^r\big(\bar Q , \frac{\d y \d s}{s^{d+1}}\big)}^q\Big)^\frac{1}{q}
    \end{align*}
    for every $x\in\IR^d$. Taking $\L^p$-norms on both sides yields the first direction.\\

    \noindent \textbf{Step 2: We show \enquote{$\gtrsim$}.} Fix  $x\in\IR^d$. Again, for each $k\in\IZ$, there exist a unique $Q_k\in\square_k$ with $x\in Q_k$. Hence, we have
    \begin{align*}
        \Big( \sum\limits_{Q\in\square} \mathbf{1}_{Q}(x)|Q|^{-\frac{\beta q}{d}}\|f\|_{\L^r\big(\bar Q , \frac{\d y \d s}{s^{d+1}}\big)}^q\Big)^\frac{1}{q}&\simeq \Big( \sum\limits_{Q\in\square} \mathbf{1}_{Q}(x)\|s^{-\beta} f(s,y)\|_{\L^r\big(\bar Q , \frac{\d y \d s}{s^{d+1}}\big)}^q\Big)^\frac{1}{q}\\
        &=  \Big( \sum\limits_{k\in\IZ} \mathbf{1}_{Q_k}(x)\|s^{-\beta} f(s,y)\|_{\L^r\big(\bar Q_k , \frac{\d y \d s}{s^{d+1}}\big)}^q\Big)^\frac{1}{q}.
    \end{align*}
    By introducing an extra integral, we get
    \begin{align}
    \label{eq: D11}
        &\Big( \sum\limits_{k\in\IZ} \mathbf{1}_{Q_k}(x)\|s^{-\beta} f(s,y)\|_{\L^r\big(\bar Q_k , \frac{\d y \d s}{s^{d+1}}\big)}^q\Big)^\frac{1}{q}\bigg\|_{\L^p} \nonumber\\
        &\simeq \bigg( \sum\limits_{k\in\IZ} \mathbf{1}_{Q_k}(x) \int\limits_{\frac{\ell(Q_k)}{2}}^{\ell(Q_k)} \bigg(\int\limits_{\frac{\ell(Q_k)}{2}}^{\ell(Q_k)}\int\limits_{ Q_k } |s^{-\beta} f(s,y)|^r\,\frac{\d y \d s}{s^{d+1}}\bigg)^\frac{q}{r}\,\frac{\d t}{t}\bigg)^\frac{1}{q}.
    \end{align}
    Note that for $x\in Q_k$ and $t\in (\frac{\ell(Q_k)}{2},\ell(Q_k)]$, we have
    \begin{align*}
        \Big(\frac{\ell(Q_k)}{2},\ell(Q_k)\Big] \times Q_k \subset \Big(\frac{t}{2},2t\Big] \times B(x,2\sqrt{d}t).
    \end{align*}
    Thus, we further estimate the last term in \eqref{eq: D11} by
    \begin{align*}
        &  \bigg( \sum\limits_{k\in\IZ} \mathbf{1}_{Q_k}(x) \int\limits_{\frac{\ell(Q_k)}{2}}^{\ell(Q_k)} \bigg(\int\limits_{\frac{\ell(Q_k)}{2}}^{\ell(Q_k)}\int\limits_{ Q_k } |s^{-\beta} f(s,y)|^r\,\frac{\d y \d s}{s^{d+1}}\bigg)^\frac{q}{r}\,\frac{\d t}{t}\bigg)^\frac{1}{q}\\
        &\lesssim \bigg( \sum\limits_{k\in\IZ} \mathbf{1}_{Q_k}(x) \int\limits_{\frac{\ell(Q_k)}{2}}^{\ell(Q_k)} \bigg(\int\limits_{\frac{t}{2}}^{2t}\int\limits_{B(x,2\sqrt{d}t)} |s^{-\beta} f(s,y)|^r\,\frac{\d y \d s}{s^{d+1}}\bigg)^\frac{q}{r}\,\frac{\d t}{t}\bigg)^\frac{1}{q}\\
        &\lesssim \int\limits_{0}^{\infty} \bigg(\fint\limits_{\frac{t}{2}}^{2t}\fint\limits_{B(x,2\sqrt{d}t)} |s^{-\beta} f(s,y)|^r\,\d y \d s\bigg)^\frac{q}{r}\,\frac{\d t}{t}\bigg)^\frac{1}{q}.
    \end{align*}
    In summary, we showed 
    \begin{align*}
        \Big( \sum\limits_{Q\in\square} \mathbf{1}_{Q}(x)|Q|^{-\frac{\beta q}{d}}\|f\|_{\L^r\big(\bar Q , \frac{\d y \d s}{s^{d+1}}\big)}^q\Big)^\frac{1}{q} \lesssim \bigg(\int\limits_{0}^{\infty} \bigg(\fint\limits_{\frac{t}{2}}^{2t}\fint\limits_{B(x,2\sqrt{d}t)} |s^{-\beta} f(s,y)|^r\,\d y \d s\bigg)^\frac{q}{r}\,\frac{\d t}{t}\bigg)^\frac{1}{q}.
    \end{align*}
    for every $x\in\IR^d$. Finally, by taking $\L^p$-norms in $x$ on both sides and using a change of Whitney parameters (Proposition~\ref{prop: independent Whitney cube}), we conclude the reverse inequality.
\end{proof}

\begin{proposition}
\label{prop: dyadic char p=infty}
    Let $0<q,r\leq \infty$ and $\beta\in \IR$. Then, for $f\in \T^{\infty,q,r}_\beta$ we have
    \begin{align}
    \label{eq: dyadic char for p=infty}
        \|f\|_{\T^{\infty,q,r}_\beta} \simeq \sup\limits_{P\in\square} \bigg(\fint\limits_{P} \sum\limits_{Q\subset P} \mathbf{1}_{Q}(x)|Q|^{-\frac{\beta q}{d}} \|f\|_{\L^r\big(\bar Q , \frac{\d y \d s}{s^{d+1}}\big)}^q \,\d x \bigg)^\frac{1}{q},
    \end{align}
    with the usual modification if $q,r$ are infinite. Here, the implicit constants depend only on $d,q,r,\beta$ and the sum runs over all dyadic cubes $Q\in \square$ with $Q\subset P$.
\end{proposition}

\begin{proof}
    We simplify notation by focusing on the case $0<q,r <\infty$. The case of infinite parameters follows in a similar way.\\
    
    \noindent \textbf{Step 1: We show \enquote{$\lesssim$}.}
    Fix $\tau>0$ and $y\in \IR^d$. Then, there exist $k\in \IZ$ and $P\in\square_k$ with $y\in P$ and $\frac{\ell(P)}{2}\leq \tau<\ell(P) $. It follows
    \begin{align} 
    \label{eq: D0}
        &\bigg(\int\limits_{0}^\tau  \fint\limits_{B(y,\tau)} \bigg(\fint\limits_{\frac{t}{2}}^t \fint\limits_{B(x,t)} |s^{-\beta}f(s,z)|^r \,\d z \d s\bigg)^\frac{q}{r}\,\frac{\d x \d t}{t}\bigg)^\frac{1}{q}\nonumber\\
        &\lesssim \bigg(\int\limits_{0}^{\ell(P)}  \fint\limits_{3P} \bigg(\fint\limits_{\frac{t}{2}}^t \fint\limits_{B(x,t)} |s^{-\beta}f(s,z)|^r \,\d z \d s\bigg)^\frac{q}{r}\,\frac{\d x \d t}{t}\bigg)^\frac{1}{q}
    \end{align}
    Pick a finite sequence $(P_i)_{i\in I}\subset \square_k$ with the properties
    \begin{align*}
        3P \subset \bigcup_{i\in I} P_i \quad \text{and} \quad |I| \leq C
    \end{align*}
    for some constant $C>0$ depending only on the dimension $d$. This covering allows us to estimate
    \begin{align}
    \label{eq: dyad0}
        &\bigg(\int\limits_{0}^{\ell(P)}  \fint\limits_{3P} \bigg(\fint\limits_{\frac{t}{2}}^t \fint\limits_{B(x,t)} |s^{-\beta}f(s,z)|^r \,\d z \d s\bigg)^\frac{q}{r}\,\frac{\d x \d t}{t}\bigg)^\frac{1}{q} \nonumber\\
        &\lesssim \sum\limits_{i\in I} \bigg(\int\limits_{0}^{\ell(P_i)}  \fint\limits_{P_i} \bigg(\fint\limits_{\frac{t}{2}}^t \fint\limits_{B(x,t)} |s^{-\beta}f(s,z)|^r \,\d z \d s\bigg)^\frac{q}{r}\,\frac{\d x \d t}{t}\bigg)^\frac{1}{q}.
    \end{align}
    Fix $i\in I$ and decompose the Carleson box $(0,\ell(P_i)]\times P_i$ into 
    \begin{align}
    \label{eq: Carleson decomp}
        (0,\ell(P_i)]\times P_i =\bigcup_{\substack{Q\in\square \\ Q\subset P_i}} \bar Q.
    \end{align}
    With this observation, we estimate a single term of the right side of \eqref{eq: dyad0} as follows:
    \begin{align*}
        &\bigg(\int\limits_{0}^{\ell(P_i)}  \fint\limits_{P_i} \bigg(\fint\limits_{\frac{t}{2}}^t \fint\limits_{B(x,t)} |s^{-\beta}f(s,z)|^r \,\d z \d s\bigg)^\frac{q}{r}\,\frac{\d x \d t}{t}\bigg)^\frac{1}{q}\\
        &=\bigg(\frac{1}{|P_i|}\sum\limits_{Q\subset P_i}\int\limits_{\frac{\ell(Q)}{2}}^{\ell(Q)}  \int\limits_{Q} \bigg(\fint\limits_{\frac{t}{2}}^t \fint\limits_{B(x,t)} |s^{-\beta}f(s,z)|^r \,\d z \d s\bigg)^\frac{q}{r}\,\frac{\d x \d t}{t}\bigg)^\frac{1}{q}\\
        &\simeq\bigg(\sum\limits_{Q\subset P_i}\int\limits_{\frac{\ell(Q)}{2}}^{\ell(Q)}  \fint\limits_{P_i} \mathbf{1}_{Q}(x)\bigg(\int\limits_{\frac{t}{2}}^t \int\limits_{B(x,t)} |s^{-\beta}f(s,z)|^r \,\frac{\d z \d s}{s^{d+1}}\bigg)^\frac{q}{r}\,\frac{\d x \d t}{t}\bigg)^\frac{1}{q}.
    \intertext{For $t\in (\ell(Q)/2 , \ell(Q)]$ and $x\in Q$, we have $(t/2,t]\times B(x,t)\subset(\ell(Q)/4,\ell(Q)]\times 3Q$. Hence, we further estimate}
        &\lesssim\bigg(\sum\limits_{Q\subset P_i}\int\limits_{\frac{\ell(Q)}{2}}^{\ell(Q)}  \fint\limits_{P_i}\mathbf{1}_{Q}(x) \bigg(\int\limits_{\frac{\ell(Q)}{4}}^{\ell(Q)} \int\limits_{3Q} |s^{-\beta}f(s,z)|^r \,\frac{\d z \d s}{s^{d+1}}\bigg)^\frac{q}{r}\,\frac{\d x \d t}{t}\bigg)^\frac{1}{q}\\
        &\lesssim\bigg(\sum\limits_{Q\subset P_i}  \fint\limits_{P_i}\mathbf{1}_{Q}(x)|Q|^{\frac{-\beta q}{d}} \bigg(\int\limits_{\frac{\ell(Q)}{4}}^{\ell(Q)} \int\limits_{3Q} |f(s,z)|^r \,\frac{\d z \d s}{s^{d+1}}\bigg)^\frac{q}{r}\,\d x \bigg)^\frac{1}{q}\\
        &\lesssim\sup\limits_{P\in\square}\bigg(\sum\limits_{Q\subset P}\fint\limits_{P} \mathbf{1}_{Q}(x)|Q|^{\frac{-\beta q}{d}} \bigg(\int\limits_{\frac{\ell(Q)}{4}}^{\ell(Q)} \int\limits_{3Q} |f(s,z)|^r \,\frac{\d z \d s}{s^{d+1}}\bigg)^\frac{q}{r}\,\d x\bigg)^\frac{1}{q}.
    \end{align*}
    Notice that this bound holds uniformly in $i$. Since the index set $I$ is finite, we can plug the uniform bound back in \eqref{eq: dyad0}, so that \eqref{eq: D0} yields
    \begin{align*}
        &\bigg(\int\limits_{0}^\tau  \fint\limits_{B(y,\tau)} \bigg(\fint\limits_{\frac{t}{2}}^t \fint\limits_{B(x,t)} |s^{-\beta}f(s,z)|^r \,\d z \d s\bigg)^\frac{q}{r}\,\frac{\d x \d t}{t}\bigg)^\frac{1}{q}\\
        &\lesssim\sup\limits_{P\in\square}\bigg(\sum\limits_{Q\subset P}\fint\limits_{P} \mathbf{1}_{Q}(x)|Q|^{\frac{-\beta q}{d}} \bigg(\int\limits_{\frac{\ell(Q)}{4}}^{\ell(Q)} \int\limits_{3Q} |f(s,z)|^r \,\frac{\d z \d s}{s^{d+1}}\bigg)^\frac{q}{r}\,\d x\bigg)^\frac{1}{q}
    \end{align*}
    for all $\tau >0$ and $y\in \IR^d$. Taking suprema in $\tau$ and $y$ yields
    \begin{align*}
        \|f\|_{\T^{\infty,q,r}_\beta} \lesssim\sup\limits_{P\in\square}\bigg(\sum\limits_{Q\subset P}\fint\limits_{P} \mathbf{1}_{Q}(x)|Q|^{\frac{-\beta q}{d}} \bigg(\int\limits_{\frac{\ell(Q)}{4}}^{\ell(Q)} \int\limits_{3Q} |f(s,z)|^r \,\frac{\d z \d s}{s^{d+1}}\bigg)^\frac{q}{r}\,\d x\bigg)^\frac{1}{q}.
    \end{align*}
    Finally, a similar covering argument as in Step~1 of the proof in Proposition~\ref{prop: dyadic char} allows us to recover our Whitney box $(\ell(Q)/2,\ell(Q)]\times Q$, resulting in
    \begin{align*}
        \|f\|_{\T^{\infty,q,r}_\beta} \lesssim\sup\limits_{P\in\square} \bigg(\fint\limits_{P} \sum\limits_{Q\subset P} \mathbf{1}_{Q}(x)|Q|^{-\frac{\beta q}{d}} \|f\|_{\L^r\big(\bar Q , \frac{\d y \d s}{s^{d+1}}\big)}^q \,\d x \bigg)^\frac{1}{q},
    \end{align*}
    which proves the first direction.\\

    \noindent \textbf{Step 2: We show \enquote{$\gtrsim$}.}
    This direction follows the same ideas as in Step~2 of the proof of Proposition~\ref{prop: dyadic char}.
    Fix a dyadic cube $P\in\square$. Then, we calculate
    \begin{align*}
        &\bigg(\fint\limits_{P} \sum\limits_{Q\subset P} \mathbf{1}_{Q}(x)|Q|^{-\frac{\beta q}{d}} \|f\|_{\L^r\big(\bar Q , \frac{\d y \d s}{s^{d+1}}\big)}^q \,\d x \bigg)^\frac{1}{q}\\
        &\lesssim \bigg(\frac{1}{|P|}\sum\limits_{Q\subset P}\int\limits_{Q}  \bigg(\int\limits_{\frac{\ell(Q)}{2}}^{\ell(Q)} \int\limits_{Q} |s^{-\beta}f(s,z)|^r \,\frac{\d z \d s}{s^{d+1}}\bigg)^\frac{q}{r} \,\d x \bigg)^\frac{1}{q}.
    \intertext{By introducing an extra integral, we get}
        &\lesssim\bigg(\frac{1}{|P|}\sum\limits_{Q\subset P}\int\limits_{\frac{\ell(Q)}{2}}^{\ell(Q)}\int\limits_{Q}  \bigg(\int\limits_{\frac{\ell(Q)}{2}}^{\ell(Q)} \int\limits_{Q} |s^{-\beta}f(s,z)|^r \,\frac{\d z \d s}{s^{d+1}}\bigg)^\frac{q}{r} \,\frac{\d x\d t}{t} \bigg)^\frac{1}{q}.
    \intertext{For $t\in(\ell(Q)/2, \ell(Q)]$ and $x\in Q$ we have $(\ell(Q)/2, \ell(Q)]\times Q \subset (t/2,2t]\times B(x,2\sqrt{d}t)$. Thus, we further estimate}
        &\lesssim\bigg(\frac{1}{|P|}\sum\limits_{Q\subset P}\int\limits_{\frac{\ell(Q)}{2}}^{\ell(Q)}\int\limits_{Q}  \bigg(\int\limits_{\frac{t}{2}}^{2 t} \int\limits_{B(x,2\sqrt{d}t)} |s^{-\beta}f(s,z)|^r \,\frac{\d z \d s}{s^{d+1}}\bigg)^\frac{q}{r} \,\frac{\d x\d t}{t} \bigg)^\frac{1}{q}.
    \intertext{Finally, the decomposition \eqref{eq: Carleson decomp} yields}
        &=\bigg(\frac{1}{|P|}\int\limits_{0}^{\ell(P)}\int\limits_{P}  \bigg(\int\limits_{\frac{t}{2}}^{2 t} \int\limits_{B(x,2\sqrt{d}t)} |s^{-\beta}f(s,z)|^r \,\frac{\d z \d s}{s^{d+1}}\bigg)^\frac{q}{r} \,\frac{\d x\d t}{t} \bigg)^\frac{1}{q}.\\
        &\lesssim \sup\limits_{\tau>0}\sup\limits_{y\in\IR^d}\bigg(\int\limits_{0}^{\tau}\fint\limits_{B(y,\tau)}  \bigg(\fint\limits_{\frac{t}{2}}^{2 t} \fint\limits_{B(x,2\sqrt{d}t)} |s^{-\beta}f(s,z)|^r \,\d z \d s\bigg)^\frac{q}{r} \,\frac{\d x\d t}{t} \bigg)^\frac{1}{q}.
    \end{align*}
    Since this estimate holds for all dyadic cubes $P\in\square$, it follows
    \begin{align*}
        &\sup\limits_{P\in\square} \bigg(\fint\limits_{P} \sum\limits_{Q\subset P} \mathbf{1}_{Q}(x)|Q|^{-\frac{\beta q}{d}} \|f\|_{\L^r\big(\bar Q , \frac{\d y \d s}{s^{d+1}}\big)}^q \,\d x \bigg)^\frac{1}{q}\\
        &\lesssim \sup\limits_{\tau>0}\sup\limits_{y\in\IR^d}\bigg(\int\limits_{0}^{\tau}\fint\limits_{B(y,\tau)}  \bigg(\fint\limits_{\frac{t}{2}}^{2 t} \fint\limits_{B(x,2\sqrt{d}t)} |s^{-\beta}f(s,z)|^r \,\d z \d s\bigg)^\frac{q}{r} \,\frac{\d x\d t}{t} \bigg)^\frac{1}{q}.
    \end{align*}
By a change of Whitney parameters (Proposition~\ref{prop: independent Whitney cube}), we conclude the reverse inequality.
\end{proof}

Next, we want to provide another equivalent characterization of $\T^{p,q,r}_\beta$ by ``discrete local square functions". To this end, for $0<q,r\leq \infty$, $\beta\in\IR$ and $f \in \L^0( \IR^{d+1}_+)$ we define
\begin{align*}
    G^{q,r}_{\beta,P}(f)(x) \coloneqq \Big(\sum\limits_{Q\subset P} \mathbf{1}_Q(x)|Q|^{-\frac{\beta q}{d}} \|f\|_{\L^r\big(\bar Q , \frac{\d y \d s}{s^{d+1}}\big)}^q \Big)^\frac{1}{q}
\end{align*}
for some fixed $P\in\square$ (with the usual modification in the infinite cases). Next, for $0<c<1$ define
\begin{align}
\label{eq: def of m1/4}
    m_{\beta,P,c}^{q,r}(f) \coloneqq \inf\Big\{t: |\{x\in P: G^{q,r}_{\beta,P}(f)(x)>t\}|<c|P| \Big\},
\end{align}
which is the so-called \enquote{$c$-median} of $G^{q,r}_{\beta,P}(f)$ in $P$, and 
\begin{align}
\label{eq: def of m}
     m^{q,r}_{\beta,c}(f)(x) \coloneqq \sup\limits_{P\in\square}  m^{q,r}_{\beta,P,c}(f)\mathbf{1}_P(x).
\end{align}
These three objects have a counterpart on the level of sequence spaces, see \cite[Sec.\@ 5]{Frazier_Jawerth}. To be more precise, for $0<q\leq \infty$, $\beta\in\IR$ and $s= (s_Q)_{Q\in \square}\subset \IC$ define
\begin{align*}
    G^{\beta q}_{P}(s)(x) \coloneqq \Big(\sum\limits_{Q\subset P} \mathbf{1}_Q(x)|Q|^{-\frac{q}{2}-\frac{\beta q}{d}} |s_Q|^q \Big)^\frac{1}{q}
\end{align*}
for some fixed $P\in\square$ with the usual modification in the infinite case. Furthermore, for $0<c<1$, we define
\begin{align*}
    m_{P,c}^{\beta q}(s) \coloneqq \inf\Big\{t: |\{x\in Q: G^{\beta q}_{P}(s)(x)>t|<c|P|\Big\},
\end{align*}
as well as
\begin{align*}
     m^{\beta q}_c(s)(x) \coloneqq \sup\limits_{P\in\square}  m^{\beta q}_{P,c}(s)\mathbf{1}_P(x).
\end{align*}
By setting $s_Q \coloneqq |Q|^\frac{1}{2}\|f\|_{\L^r\big(\bar Q , \frac{\d y \d s}{s^{d+1}}\big)}$ for every $Q\in\square$, we have for all $x\in \IR^d$
\begin{align}
\label{eq: equiv Gf = Gs}
   G^{q,r}_{\beta,P}(f)(x) = G^{\beta q}_{P}(s)(x),
\end{align}
as well as
\begin{align}
\label{eq: equiv mf = ms}
    m_{\beta,P,c}^{q,r}(f) = m_{P,c}^{\beta q}(s)   \quad \text{and}\quad  m^{q,r}_{\beta,c}(f)(x) =m_c^{\beta q}(s)(x).
\end{align}
Thus, the results of \cite[Sec.\@ 5]{Frazier_Jawerth} can be utilized to establish the following characterizations, which will play an important role in the duality theory (Section~\ref{sec: duality}) and real interpolation theory (Section~\ref{sec: real interpolation}).

\begin{lemma}
\label{lem: dyadic char with subset}
    Let $0<p,q,r\leq \infty$. Furthermore, let $\varepsilon>0$. Suppose that for each dyadic cube $Q\in\square$ there exists a measurable set $E_Q\subset Q$ with $|E_Q|>\varepsilon|Q|$.  Then, for $f\in\T^{p,q,r}_\beta$ we have
    \begin{align}
    \label{eq: dyad char with subset}
        \|f\|_{\T^{p,q,r}_\beta} \simeq \bigg\| \Big(\sum\limits_{Q\in\square} \mathbf{1}_{E_Q}(\cdot)|Q|^{-\frac{\beta q}{d}} \|f\|_{\L^r\big(\bar Q , \frac{\d y \d s}{s^{d+1}}\big)}^q \Big)^\frac{1}{q} \bigg\|_{\L^p}
    \end{align}
    if $p<\infty$, and
    \begin{align}
    \label{eq: dyad char with subset p=infty}
        \|f\|_{\T^{\infty,q,r}_\beta} \simeq \sup\limits_{P\in\square} \bigg(\fint\limits_{P} \sum\limits_{Q\subset P} \mathbf{1}_{E_Q}(x)|Q|^{-\frac{\beta q}{d}} \|f\|_{\L^r\big(\bar Q , \frac{\d y \d s}{s^{d+1}}\big)}^q \,\d x \bigg)^\frac{1}{q}.
    \end{align}
\end{lemma}

\begin{proof}
     For every $Q\in\square$, set $s_Q \coloneqq |Q|^\frac{1}{2}\|f\|_{\L^r\big(\bar Q , \frac{\d y \d s}{s^{d+1}}\big)}$. Then, we can rewrite the results of Proposition~\ref{prop: dyadic char} and ~\ref{prop: dyadic char p=infty} into
     \begin{align*}
         \|f\|_{\T^{p,q,r}_\beta} \simeq \bigg\| \Big(\sum\limits_{Q\in\square} \mathbf{1}_{Q}(\cdot)|Q|^{-\frac{q}{2}-\frac{\beta q}{d}} |s_Q|^q \Big)^\frac{1}{q} \bigg\|_{\L^p}
     \end{align*}
     if $0<p<\infty$, and
     \begin{align*}
         \|f\|_{\T^{\infty,q,r}_\beta} \simeq \sup\limits_{P\in\square} \bigg(\fint\limits_{P} \sum\limits_{Q\subset P} \mathbf{1}_{Q}(x)|Q|^{-\frac{q}{2}-\frac{\beta q}{d}} |s_Q|^q \,\d x \bigg)^\frac{1}{q}.
     \end{align*}
     In the case $p<\infty$, we use \cite[Prop.\@ 2.7]{Frazier_Jawerth} to get
     \begin{align*}
         \|f\|_{\T^{p,q,r}_\beta} \simeq \bigg\| \Big(\sum\limits_{Q\in\square} \mathbf{1}_{E_Q}|Q|^{-\frac{q}{2}-\frac{\beta q}{d}} |s_Q|^q \Big)^\frac{1}{q} \bigg\|_{\L^p}
         = \bigg\| \Big(\sum\limits_{Q\in\square} \mathbf{1}_{E_Q}|Q|^{-\frac{\beta q}{d}} \|f\|_{\L^r\big(\bar Q , \frac{\d y \d s}{s^{d+1}}\big)}^q \Big)^\frac{1}{q} \bigg\|_{\L^p},
     \end{align*}
     which proves \eqref{eq: dyad char with subset}. Similarly, if $p=\infty$, we use \cite[Prop.\@ 5.4]{Frazier_Jawerth} to obtain
     \begin{align*}
         \|f\|_{\T^{\infty,q,r}_\beta} &\simeq \sup\limits_{P\in\square} \bigg(\fint\limits_{P} \sum\limits_{Q\subset P} \mathbf{1}_{E_Q}(x)|Q|^{-\frac{q}{2}-\frac{\beta q}{d}} |s_Q|^q \,\d x \bigg)^\frac{1}{q}\\
         &=\sup\limits_{P\in\square} \bigg(\fint\limits_{P} \sum\limits_{Q\subset P} \mathbf{1}_{E_Q}(x)|Q|^{-\frac{\beta q}{d}} \|f\|_{\L^r\big(\bar Q , \frac{\d y \d s}{s^{d+1}}\big)}^q \,\d x \bigg)^\frac{1}{q},
     \end{align*}
     which proves \eqref{eq: dyad char with subset p=infty}.
\end{proof}

\begin{proposition}
\label{prop: char with m}
    Let $0<p,q,r\leq \infty$, $\beta\in\IR$ and $0<c<1$. Then, for $f\in \T^{p,q,r}_\beta$ we have
    \begin{align*}
        \|f\|_{\T^{p,q,r}_\beta} \simeq \| m^{q,r}_{\beta,c}(f) \|_{\L^p}
    \end{align*}
    with implicit constants depending only on $d,q,r,\beta$. 
\end{proposition}

\begin{proof}
    Define the sequence $s_Q=  |Q|^\frac{1}{2}\|f\|_{\L^r\big(\bar Q , \frac{\d y \d s}{s^{d+1}}\big)}$ for every $Q\in\square$. Then, Proposition~\ref{prop: dyadic char} and Proposition~\ref{prop: dyadic char p=infty} can be rewritten as
    \begin{align*}
        \|f\|_{\T^{p,q,r}_\beta} \simeq \bigg\| \Big(\sum\limits_{Q\in\square} \mathbf{1}_{Q}|Q|^{-\frac{q}{2}-\frac{\beta q}{d}} |s_Q|^q \Big)^\frac{1}{q} \bigg\|_{\L^p}
    \end{align*}
    and
    \begin{align*}
        \|f\|_{\T^{\infty,q,r}_\beta} \simeq \sup\limits_{P\in\square} \bigg(\fint\limits_{P} \sum\limits_{Q\subset P} \mathbf{1}_{Q}(x)|Q|^{-\frac{q}{2}-\frac{\beta q}{d}} |s_Q|^q \,\d x \bigg)^\frac{1}{q}.
    \end{align*}
    This characterization, together with \cite[Prop.\@ 5.5]{Frazier_Jawerth} (which holds for every $0<c<1$) and the observation \eqref{eq: equiv mf = ms}, yields
    \begin{align*}
        \|f\|_{\T^{p,q,r}_\beta} &\simeq \| m^{\beta q}_c(s)\|_{\L^p}=\| m^{q,r}_{\beta,c}(f)\|_{\L^p},
    \end{align*}
    for all $0<p,q,r\leq \infty$, which proves the claim.
\end{proof}

We end this section by proving the John--Nirenberg-type property for the endpoint spaces $\T^{\infty,q,r}_\beta$ used in the characterization of $\dot \F^\beta_{\infty,q}(\IR^d)$ in Theorem~\ref{thm: char of lifted Triebel data}. The proof is based on the discrete characterization in Proposition~\ref{prop: dyadic char p=infty} and a John--Nirenberg-type property for sequence spaces.

\begin{lemma}
\label{lem: John-Nirenberg lemma}
    Let $0<q,r\leq \infty$, $0<\alpha<\infty$ and $\beta\in \IR$. Then, for $f\in\T^{\infty,q,r}_\beta$ we have
    \begin{align*}
        \|f\|_{\T^{\infty,q,r}_\beta}
        &\simeq \sup\limits_{P\in\square} \bigg(\fint\limits_{P} \Big(\sum\limits_{Q\subset P} \mathbf{1}_{Q}(x)|Q|^{-\frac{\beta q}{d}} \|f\|_{\L^r\big(\bar Q , \frac{\d y \d s}{s^{d+1}}\big)}^q \Big)^\frac{\alpha}{q}\,\d x \bigg)^\frac{1}{\alpha}
    \end{align*}
    with implicit constants depending only on $d,q,r,\alpha,\beta$. 
\end{lemma}

\begin{proof}
    For every $Q\in\square$, set $s_Q \coloneqq |Q|^\frac{1}{2}\|f\|_{\L^r\big(\bar Q , \frac{\d y \d s}{s^{d+1}}\big)}$. Then, we can rewrite the result of Proposition~\ref{prop: dyadic char p=infty} into
     \begin{align*}
         \|f\|_{\T^{\infty,q,r}_\beta} \simeq \sup\limits_{P\in\square} \bigg(\fint\limits_{P} \sum\limits_{Q\subset P} \mathbf{1}_{Q}(x)|Q|^{-\frac{q}{2}-\frac{\beta q}{d}} |s_Q|^q \,\d x \bigg)^\frac{1}{q}
     \end{align*}
     with implicit constants depending only on $d,q,r,\beta$. Now, by a John--Nirenberg-type property of sequence spaces \cite[Cor.\@ 5.7]{Frazier_Jawerth}, it follows
     \begin{align*}
         \|f\|_{\T^{\infty,q,r}_\beta} &\simeq \sup\limits_{P\in\square} \bigg(\fint\limits_{P} \Big(\sum\limits_{Q\subset P} \mathbf{1}_{Q}(x)|Q|^{-\frac{q}{2}-\frac{\beta q}{d}} |s_Q|^q\Big)^\frac{\alpha}{q} \,\d x \bigg)^\frac{1}{\alpha}\\
         &= \sup\limits_{P\in\square} \bigg(\fint\limits_{P} \Big(\sum\limits_{Q\subset P} \mathbf{1}_{Q}(x)|Q|^{-\frac{\beta q}{d}} \|f\|_{\L^r\big(\bar Q , \frac{\d y \d s}{s^{d+1}}\big)}^q \Big)^\frac{\alpha}{q} \,\d x \bigg)^\frac{1}{\alpha}
     \end{align*}
    for all $\alpha\in(0,\infty)$, where the implicit constants depend only on $d,q,r,\alpha,\beta$. 
\end{proof}

\begin{proposition}[John--Nirenberg-type property]
\label{prop: John-Nirenberg continuous}
    Let $0<q,r\leq \infty$, $0<\alpha<\infty$ and $\beta\in \IR$. Then, for $f\in \T^{\infty,q,r}_\beta$ we have
    \begin{align*}
        \|f\|_{\T^{\infty,q,r}_\beta}
        &\simeq  \sup\limits_{\tau>0} \sup\limits_{y\in\IR^d} 
        \bigg(  \fint\limits_{B(y,\tau)} \bigg(\int\limits_{0}^\tau\bigg(\fint\limits_{\frac{t}{2}}^t \fint\limits_{B(x,t)} |s^{-\beta}f(s,z)|^r \,\d z \d s\bigg)^\frac{q}{r}\,\frac{ \d t}{t}\bigg)^\frac{\alpha}{q}\d x\bigg)^\frac{1}{\alpha}
    \end{align*}
    with implicit constants depending only on $d,q,r,\alpha,\beta$. 
\end{proposition}

\begin{proof}
    Due to Lemma~\ref{lem: John-Nirenberg lemma}, it remains to show that
     \begin{align*}
         &\sup\limits_{P\in\square} \bigg(\fint\limits_{P} \Big(\sum\limits_{Q\subset P} \mathbf{1}_{Q}(x)|Q|^{-\frac{\beta q}{d}} \|f\|_{\L^r\big(\bar Q , \frac{\d y \d s}{s^{d+1}}\big)}^q\Big)^\frac{\alpha}{q} \,\d x \bigg)^\frac{1}{\alpha}\\
         &\simeq \sup\limits_{\tau>0} \sup\limits_{y\in\IR^d} 
        \bigg(  \fint\limits_{B(y,\tau)} \bigg(\int\limits_{0}^\tau\bigg(\fint\limits_{\frac{t}{2}}^t \fint\limits_{B(x,t)} |s^{-\beta}f(s,z)|^r \,\d z \d s\bigg)^\frac{q}{r}\,\frac{ \d t}{t}\bigg)^\frac{\alpha}{q}\d x\bigg)^\frac{1}{\alpha}.
     \end{align*}
     The procedure is similar to the proof of the Proposition~\ref{prop: dyadic char p=infty} and consists of two steps.\\

     \noindent \textbf{Step 1: We show ``$\lesssim$"}. Fix a dyadic cube $P\in\square$ and $x\in P$. Recall the notation $j_P = -\log_2(\ell(P))$. Then, there exists a unique sequence $(Q_k)_{k=-\infty}^{-j_P}\subset P$ of dyadic cubes with $Q_k\in\square_k $ and $x\in Q_k$ for every $k\leq - j_P$. It follows
     \begin{align*}
         & \Big(\sum\limits_{Q\subset P} \mathbf{1}_{Q}(x)|Q|^{-\frac{\beta q}{d}} \|f\|_{\L^r\big(\bar Q , \frac{\d y \d s}{s^{d+1}}\big)}^q\Big)^\frac{1}{q}\\
         &= \Big(\sum\limits_{k=-\infty}^{-j_P} |Q_k|^{-\frac{\beta q}{d}} \|f\|_{\L^r\big(\bar Q_k , \frac{\d y \d s}{s^{d+1}}\big)}^q\Big)^\frac{1}{q}.
        \intertext{By introducing an extra integral, we get}
         &\simeq \bigg(\sum\limits_{k=-\infty}^{-j_P}\, \int\limits_{\frac{\ell(Q_k)}{2}}^{\ell(Q_k)} \bigg(\fint\limits_{\frac{\ell(Q_k)}{2}}^{\ell(Q_k)} \fint\limits_{Q_k} |s^{-\beta}f(s,z)|^r\,\d z \d s \bigg)^\frac{q}{r} \,\frac{\d t}{t}\bigg)^\frac{1}{q}.
     \end{align*}
    Note that for each $t\in (\frac{\ell(Q_k)}{2},\ell(Q_k)]$, we have
    \begin{align*}
        \Big(\frac{\ell(Q_k)}{2},\ell(Q_k)\Big] \times Q_k \subset \Big(\frac{t}{2},2t\Big] \times B(x,2\sqrt{d}t).
    \end{align*}
    Hence, we get
    \begin{align*}
        &\Big(\sum\limits_{Q\subset P} \mathbf{1}_{Q}(x)|Q|^{-\frac{\beta q}{d}} \|f\|_{\L^r\big(\bar Q , \frac{\d y \d s}{s^{d+1}}\big)}^q\Big)^\frac{1}{q}.\\
        &\lesssim\bigg(\sum\limits_{k=-\infty}^{-j_P}\, \int\limits_{\frac{\ell(Q_k)}{2}}^{\ell(Q_k)} \bigg(\fint\limits_{\frac{t}{2}}^{2t} \fint\limits_{B(x,2\sqrt{d}t)} |s^{-\beta}f(s,z)|^r\,\d z \d s \bigg)^\frac{q}{r} \,\frac{\d t}{t}\bigg)^\frac{1}{q}\\
        &=\bigg( \int\limits_{0}^{\ell(P)} \bigg(\fint\limits_{\frac{t}{2}}^{2t} \fint\limits_{B(x,2\sqrt{d}t)} |s^{-\beta}f(s,z)|^r\,\d z \d s \bigg)^\frac{q}{r} \,\frac{\d t}{t}\bigg)^\frac{1}{q}.
    \end{align*}
    Since this holds for all dyadic cubes $P\in\square$ and $x\in P$, we take $\L^\alpha$-averages of both sides and then supremum over all $P$ to obtain
    \begin{align*}                          &\sup\limits_{P\in\square}\bigg(\fint\limits_P\Big(\sum\limits_{Q\subset P} \mathbf{1}_{Q}(x)|Q|^{-\frac{\beta q}{d}} \|f\|_{\L^r\big(\bar Q , \frac{\d y \d s}{s^{d+1}}\big)}^q\Big)^\frac{\alpha}{q}\,\d x\bigg)^\frac{1}{\alpha}\\
    &\lesssim\sup\limits_{P\in\square}\bigg(\fint\limits_P\bigg( \int\limits_{0}^{\ell(P)} \bigg(\fint\limits_{\frac{t}{2}}^{2t} \fint\limits_{B(x,2\sqrt{d}t)} |s^{-\beta}f(s,z)|^r\,\d z \d s \bigg)^\frac{q}{r} \,\frac{\d t}{t}\bigg)^\frac{\alpha}{q}\,\d x\bigg)^\frac{1}{\alpha}\\
    &\lesssim\sup\limits_{\tau>0} \sup\limits_{y\in\IR^d} 
        \bigg(  \fint\limits_{B(y,\tau)} \bigg(\int\limits_{0}^\tau \bigg(\fint\limits_{\frac{t}{2}}^{2t} \fint\limits_{B(x,2\sqrt{d}t)} |s^{-\beta}f(s,z)|^r\,\d z \d s \bigg)^\frac{q}{r} \,\frac{\d t}{t}\bigg)^\frac{\alpha}{q}\,\d x\bigg)^\frac{1}{\alpha}.
    \end{align*}
    Finally, a straight forward change of Whitney parameters in the norm on the right-hand side (see Remark~\ref{prop: independent Whitney cube}) concludes
    \begin{align*}
         &\sup\limits_{P\in\square} \bigg(\fint\limits_{P} \Big(\sum\limits_{Q\subset P} \mathbf{1}_{Q}(x)|Q|^{-\frac{\beta q}{d}} \|f\|_{\L^r\big(\bar Q , \frac{\d y \d s}{s^{d+1}}\big)}^q\Big)^\frac{\alpha}{q} \,\d x \bigg)^\frac{1}{\alpha}\\
         &\lesssim \sup\limits_{\tau>0} \sup\limits_{y\in\IR^d} 
        \bigg(  \fint\limits_{B(y,\tau)} \bigg(\int\limits_{0}^\tau\bigg(\fint\limits_{\frac{t}{2}}^t \fint\limits_{B(x,t)} |s^{-\beta}f(s,z)|^r \,\d z \d s\bigg)^\frac{q}{r}\,\frac{ \d t}{t}\bigg)^\frac{\alpha}{q}\d x\bigg)^\frac{1}{\alpha}.
    \end{align*}

    \noindent \textbf{Step 2: We show ``$\gtrsim$"}.
    Fix $\tau>0$ and $y\in \IR^d$. Then, there exist $k\in \IZ$ and $P\in\square_k$ with $y\in P$ and $\frac{\ell(P)}{2}\leq \tau<\ell(P) $. It follows
    \begin{align*}
        &\bigg(\fint\limits_{B(y,\tau)}\bigg(\int\limits_{0}^\tau   \bigg(\fint\limits_{\frac{t}{2}}^t \fint\limits_{B(x,t)} |s^{-\beta}f(s,z)|^r \,\d z \d s\bigg)^\frac{q}{r}\,\frac{\d t}{t}\bigg)^\frac{\alpha}{q}\,\d x\bigg)^\frac{1}{\alpha}\\
        &\lesssim \bigg( \fint\limits_{3P} \bigg(\int\limits_{0}^{\ell(P)} \bigg(\fint\limits_{\frac{t}{2}}^t \fint\limits_{B(x,t)} |s^{-\beta}f(s,z)|^r \,\d z \d s\bigg)^\frac{q}{r}\,\frac{\d t}{t}\bigg)^\frac{\alpha}{q}\,\d x\bigg)^\frac{1}{\alpha}.
    \end{align*}
    Since $P\in\square_k$, we can pick a finite sequence $(P_i)_{i\in I}\subset \square_k$ with the properties
    \begin{align*}
        3P \subset \bigcup_{i\in I} P_i \quad \text{and} \quad |I| \leq C
    \end{align*}
    for some constant $C>0$ depending only on the dimension $d$. This covering allows us to estimate
    \begin{align}
    \label{eq: idk 1}
        &\bigg(\fint\limits_{B(y,\tau)}\bigg(\int\limits_{0}^\tau   \bigg(\fint\limits_{\frac{t}{2}}^t \fint\limits_{B(x,t)} |s^{-\beta}f(s,z)|^r \,\d z \d s\bigg)^\frac{q}{r}\,\frac{\d t}{t}\bigg)^\frac{\alpha}{q}\,\d x\bigg)^\frac{1}{\alpha} \nonumber\\
        &\lesssim \sum\limits_{i\in I} \bigg( \fint\limits_{P_i} \bigg(\int\limits_{0}^{\ell(P_i)} \bigg(\fint\limits_{\frac{t}{2}}^t \fint\limits_{B(x,t)} |s^{-\beta}f(s,z)|^r \,\d z \d s\bigg)^\frac{q}{r}\,\frac{\d t}{t}\bigg)^\frac{\alpha}{q}\,\d x\bigg)^\frac{1}{\alpha}.
    \end{align}
    We focus on the single terms
    \begin{align*}
        \bigg(\int\limits_{0}^{\ell(P_i)} \bigg(\fint\limits_{\frac{t}{2}}^t \fint\limits_{B(x,t)} |s^{-\beta}f(s,z)|^r \,\d z \d s\bigg)^\frac{q}{r}\,\frac{\d t}{t}\bigg)^\frac{1}{q}.
    \end{align*}
    Fix an index $i$ and $x\in P_i$. Then, there exists a sequence of dyadic cubes $(Q_k)_{k=-\infty}^{-j_{P_i}}\subset P_i$ with $Q_k\in\square_k$ and $x\in Q_k$ for every $k\leq - j_{P_i}$. It follows
    \begin{align*}
        &\bigg(\int\limits_{0}^{\ell(P_i)} \bigg(\fint\limits_{\frac{t}{2}}^t \fint\limits_{B(x,t)} |s^{-\beta}f(s,z)|^r \,\d z \d s\bigg)^\frac{q}{r}\,\frac{\d t}{t}\bigg)^\frac{1}{q}\\
        &= \bigg(\sum\limits_{k=-\infty}^{-j_{P_i}}\mathbf{1}_{Q_k}(x)\int\limits_{\frac{\ell(Q_k)}{2}}^{\ell(Q_k)} \bigg(\fint\limits_{\frac{t}{2}}^t \fint\limits_{B(x,t)} |s^{-\beta}f(s,z)|^r \,\d z \d s\bigg)^\frac{q}{r}\,\frac{\d t}{t}\bigg)^\frac{1}{q}.
    \intertext{For $t\in (\ell(Q_k)/2 , \ell(Q_k)]$ we have $(t/2,t]\times B(x,t)\subset(\ell(Q_k)/4,\ell(Q_k)]\times 3Q_k$. Hence,}
        &\lesssim \bigg(\sum\limits_{k=-\infty}^{-j_{P_i}}\mathbf{1}_{Q_k}(x)\int\limits_{\frac{\ell(Q_k)}{2}}^{\ell(Q_k)} \bigg(\fint\limits_{\frac{\ell(Q_k)}{4}}^{\ell(Q_k)} \fint\limits_{3Q_k} |s^{-\beta}f(s,z)|^r \,\d z \d s\bigg)^\frac{q}{r}\,\frac{\d t}{t}\bigg)^\frac{1}{q}\\
        &\lesssim  \bigg(\sum\limits_{k=-\infty}^{-j_{P_i}}\mathbf{1}_{Q_k}(x)|Q_k|^{-\frac{\beta q}{d}}\bigg( \fint\limits_{\frac{\ell(Q_k)}{4}}^{\ell(Q_k)} \fint\limits_{3Q_k} |f(s,z)|^r \,\d z \d s\bigg)^\frac{q}{r}\bigg)^\frac{1}{q}\\
        &\lesssim \bigg(\sum\limits_{Q\subset {P_i}}\mathbf{1}_{Q}(x)|Q|^{-\frac{\beta q}{d}} \bigg(\fint\limits_{\frac{\ell(Q)}{4}}^{\ell(Q)} \fint\limits_{3Q} |f(s,z)|^r \,\d z \d s\bigg)^\frac{q}{r}\bigg)^\frac{1}{q},
    \end{align*}
    where we took the sum over all dyadic cubes $Q\subset P_i$ in the last step. Since this estimate holds for all $x\in P_i$, taking the $\L^\alpha$-average of both sides yields
    \begin{align*}
         &\bigg(\fint\limits_{P_i} \bigg(\int\limits_{0}^{\ell(P_i)} \bigg(\fint\limits_{\frac{t}{2}}^t \fint\limits_{B(x,t)} |s^{-\beta}f(s,z)|^r \,\d z \d s\bigg)^\frac{q}{r}\,\frac{\d t}{t}\bigg)^\frac{\alpha}{q} \,\d x\bigg)^\frac{1}{\alpha}\\
         &\lesssim \bigg(\fint\limits_{P_i}\bigg(\sum\limits_{Q\subset {P_i}}\mathbf{1}_{Q}(x)|Q|^{-\frac{\beta q}{d}} \bigg(\fint\limits_{\frac{\ell(Q)}{4}}^{\ell(Q)} \fint\limits_{3Q} |f(s,z)|^r \,\d z \d s\bigg)^\frac{q}{r}\bigg)^\frac{\alpha}{q}\,\d x\bigg)^\frac{1}{\alpha}\\
         &\lesssim \sup\limits_{P\in\square}\bigg(\fint\limits_{P}\bigg(\sum\limits_{Q\subset {P}}\mathbf{1}_{Q}(x)|Q|^{-\frac{\beta q}{d}} \bigg(\fint\limits_{\frac{\ell(Q)}{4}}^{\ell(Q)} \fint\limits_{3Q} |f(s,z)|^r \,\d z \d s\bigg)^\frac{q}{r}\bigg)^\frac{\alpha}{q}\,\d x\bigg)^\frac{1}{\alpha}
    \end{align*}
    Plugging this estimate into \eqref{eq: idk 1} and using finiteness of the index set $I$, we obtain
    \begin{align*}
        &\bigg(\fint\limits_{B(y,\tau)}\bigg(\int\limits_{0}^\tau   \bigg(\fint\limits_{\frac{t}{2}}^t \fint\limits_{B(x,t)} |s^{-\beta}f(s,z)|^r \,\d z \d s\bigg)^\frac{q}{r}\,\frac{\d t}{t}\bigg)^\frac{\alpha}{q}\,\d x\bigg)^\frac{1}{\alpha}\\
        &\lesssim \sup\limits_{P\in\square}\bigg( \fint\limits_{P} \bigg(\sum\limits_{Q\subset P}\mathbf{1}_{Q}(x)|Q|^{-\frac{\beta q}{d}} \bigg(\fint\limits_{\frac{\ell(Q)}{4}}^{\ell(Q)} \fint\limits_{3Q} |f(s,z)|^r \,\d z \d s\bigg)^\frac{q}{r}\bigg)^\frac{\alpha}{q}\,\d x\bigg)^\frac{1}{\alpha}.
    \end{align*}
    Using a similar covering argument as in Step~1 of the proof in Proposition~\ref{prop: dyadic char}, we can recover the Whitney box $(\ell(Q)/2,\ell(Q)]\times Q$ in the inner average integral on the right-hand side to obtain
    \begin{align*}
        &\bigg(\fint\limits_{B(y,\tau)}\bigg(\int\limits_{0}^\tau   \bigg(\fint\limits_{\frac{t}{2}}^t \fint\limits_{B(x,t)} |s^{-\beta}f(s,z)|^r \,\d z \d s\bigg)^\frac{q}{r}\,\frac{\d t}{t}\bigg)^\frac{\alpha}{q}\,\d x\bigg)^\frac{1}{\alpha}\\
        &\lesssim \sup\limits_{P\in\square} \bigg(\fint\limits_{P} \Big(\sum\limits_{Q\subset P} \mathbf{1}_{Q}(x)|Q|^{-\frac{\beta q}{d}} \|f\|_{\L^r\big(\bar Q , \frac{\d y \d s}{s^{d+1}}\big)}^q\Big)^\frac{\alpha}{q} \,\d x \bigg)^\frac{1}{\alpha}.
    \end{align*}
    Because this holds for all $\tau>0$ and $y\in\IR^d$, we conclude by taking the suprema.
\end{proof}

\subsection{Embedding theory}
\label{sec: embeddings}

\noindent In this section, we show varies embeddings of $\T^{p,q,r}_{\beta}$-spaces. In particular, we show Hardy--Littlewood--Sobolev-type embeddings (Theorem~\ref{thm: Hardy-Sobolev embedding}) and mixed-type embeddings of $\T^{p,q,r}_{\beta}$ and $\Z^{p,q,r}_{\beta}$-spaces (Theorem~\ref{thm: mixed embedding}). They are in line with Besov--Triebel--Lizorkin theory as discussed in Section~\ref{sec: introduction}.

\begin{theorem}[Weighted tent space embeddings]
\label{thm: Hardy-Sobolev embedding}
    Let $0<p_0<p_1\leq \infty$, $0<q_0,q_1 \leq \infty$, $0<r_1\leq r_0\leq \infty$ and $\beta_0,\beta_1\in \IR$ satisfying  $\beta_0-\beta_1 = \frac{d}{p_0}-\frac{d}{p_1}$. Then, we have the continuous embedding
    \begin{align}
    \label{eq: HSE}
        \T^{p_0,q_0,r_0}_{\beta_0} \hookrightarrow \T^{p_1,q_1,r_1}_{\beta_1}.
    \end{align}
\end{theorem}

Before we prove Theorem~\ref{thm: Hardy-Sobolev embedding}, we need the following two lemmata. The first lemma reassembles a nesting property of $\T^{p,q,r}_\beta$ in the parameter $q$ and $r$. The second one is a specific endpoint embedding, which follows from the weighted $\Z$-space embeddings in \cite[Thm.\@ 3.14]{Auscher_Bechtel_Haardt}.

\begin{lemma}
\label{lem: embedding in q}
    Let $0<p\leq \infty$, $0<q_0\leq q_1\leq \infty$ , $0<r_1\leq r_0\leq \infty$ and $\beta\in\IR$. Then we have the continuous embedding
    \begin{align*}
        \T^{p,q_0,r_0}_{\beta} \hookrightarrow \T^{p,q_1,r_1}_{\beta}.
    \end{align*}
\end{lemma}

\begin{proof}
    By Proposition~\ref{prop: char with m}, it is enough to show that
    \begin{align*}
        m^{q_1,r_1}_{\beta,c}(f)(x) \lesssim m^{q_0,r_0}_{\beta,c}(f)(x)
    \end{align*}
    for some $0<c<1$ and all $x\in\IR^d$. First, observe that for a fixed $P\in \square$, we have by Jensen's inequality and the embedding $\ell^{q_0}\subset \ell^{q_1}$ that
    \begin{align*}
        G^{q_1,r_1}_{\beta,P}(f)(x) &= \Big(\sum\limits_{Q\subset P} \mathbf{1}_Q(x)|Q|^{-\frac{\beta q_1}{d}} \|f\|_{\L^{r_1}\big(\bar Q , \frac{\d y \d s}{s^{d+1}}\big)}^{q_1} \Big)^\frac{1}{q_1}\\
         &\leq \Big(\sum\limits_{Q\subset P} \mathbf{1}_Q(x)|Q|^{-\frac{\beta q_0}{d}} \|f\|_{\L^{r_0}\big(\bar Q , \frac{\d y \d s}{s^{d+1}}\big)}^{q_0} \Big)^\frac{1}{q_0}=\G^{q_0,r_0}_{\beta,P}(f)(x).
    \end{align*}
    Thus, we have
    \begin{align*}
        \Big|\Big\{x\in P: G^{q_1,r_1}_{\beta,P}(f)(x)>t\Big\}\Big|\leq \Big|\Big\{x\in P: G^{q_0,r_0}_{\beta,P}(f)(x)>t\Big\}\Big|
    \end{align*}
    for all $t>0$. By definition \eqref{eq: def of m1/4}, it follows
    \begin{align*}
         m_{\beta,P,c}^{q_1,r_1}(f) \leq \,m_{\beta,P,c}^{q_0,r_0}(f),
    \end{align*}
    for all $0<c<1$ and thus
    \begin{align*}
        m^{q_1,r_1}_{\beta,c}(f)(x) =\sup\limits_{P\in\square}  m^{q_1,r_1}_{\beta,P,c}(f)\mathbf{1}_P(x) \leq \sup\limits_{P\in\square}  m^{q_0,r_0}_{\beta,P,c}(f)\mathbf{1}_P(x) = m^{q_0,r_0}_{\beta,c}(f)(x).
    \end{align*}
    This concludes the proof.
\end{proof}

\begin{lemma}
\label{lem: endpoint HSE for fixed q}
    Let $0<p<\infty$, $0<r\leq \infty$ and $\beta_0,\beta_1\in \IR$ satisfying $\beta_0-\beta_1 = \frac{d}{p}$. Then, we have the continuous embedding
    \begin{align*}
        \T^{p,\infty,r}_{\beta_0} \hookrightarrow \T^{\infty,\infty,r}_{\beta_1}.
    \end{align*}
\end{lemma}

\begin{proof}
    Combining Lemma~\ref{lem: embedding tent into Z} from below with the Hardy--Littlewood--Sobolev embedding for $\Z$-spaces \cite[Thm.\@ 3.14]{Auscher_Bechtel_Haardt}, we obtain
    \begin{align*}
        \T^{p,\infty,r}_{\beta_0} \hookrightarrow \Z^{p,\infty,r}_{\beta_0}\hookrightarrow\Z^{\infty,\infty,r}_{\beta_1}=\T^{\infty,\infty,r}_{\beta_1}
    \end{align*}
    where the last step follows by the Definitions~\ref{def: tent spaces} and~\ref{def: Z-space} of these spaces.
\end{proof}

\begin{proof}[Proof of Theorem~\ref{thm: Hardy-Sobolev embedding}]
    We divide the proof into two parts. In the first part, we prove the continuous embedding \eqref{eq: HSE} in the case $p_1<\infty$. The proof is similar to the one of \cite[Thm.\@ 2.1]{Jawerth}, where the author shows these type of embeddings for homogeneous Triebel--Lizorkin spaces. Afterwards, we show \eqref{eq: HSE} in the endpoint case $p_1=\infty$. For this, we first use a duality argument to prove the embedding for all parameters in the Banach range. Then, we apply a convex reduction argument (see Lemma ~\ref{lem: convex tent space}) to conclude the result for the full range of parameters.\\
    
    \noindent \textbf{Step 1: We show \eqref{eq: HSE} for $p_1<\infty$.} First, using Jensen's inequality, we have
    \begin{align*}
        \|f\|_{\T^{p_1,q_1,r_1}_{\beta_1}}\lesssim \|f\|_{\T^{p_1,q_1,r_0}_{\beta_1}}.
    \end{align*}
    Hence, it is sufficient to show
    \begin{align*}
        \|f\|_{\T^{p_1,q_1,r}_{\beta_1}} \lesssim\|f\|_{\T^{p_0,q_0,r}_{\beta_0}}
    \end{align*}
    for some fixed $0<r\leq\infty$. Second, it is enough to assume that $\beta_0 =0$ because otherwise we consider $s^{-\beta_0}f(s,\cdot)$ instead of $f(s,\cdot)$. Finally, due to the nesting property of $\T^{p,q,r}_\beta$-spaces in the parameter $q$, see Lemma~\ref{lem: embedding in q}, it is enough to assume $q_0=\infty$. Therefore, the embedding \eqref{eq: HSE} follows from
    \begin{align}
    \label{eq: H0}
        \|f\|_{\T^{p_1,q_1,r}_{\beta_1}} \lesssim\|f\|_{\T^{p_0,\infty,r}_{0}}
    \end{align}
    To show this embedding, we first fix $f\in \T^{p_0,\infty,r}_0 $ with $\|f\|_{\T^{p_0,\infty,r}_0} =1$. By Lemma~\ref{lem: endpoint HSE for fixed q}, we have
    \begin{align}
    \label{eq: embedding}
        \sup\limits_{x\in\IR^d}\sup\limits_{t>0}\bigg(\fint\limits_{\frac{t}{2}}^t\fint\limits_{B(x,t)} |s^{\frac{d}{p_0}}f(s,y)|^r\,\d y \d s\bigg)^\frac{1}{r} = \|f\|_{\T^{\infty,\infty,r}_{-\frac{d}{p_0}}} \lesssim \|f\|_{\T^{p_0,\infty,r}_0} =1.
    \end{align}
    Next, for fixed $x\in\IR^d$, we have
    \begin{align*}
        &\bigg(\int\limits_0^\infty \bigg(\fint\limits_{\frac{t}{2}}^t \fint\limits_{B(x,t)} |s^{-\beta_1} f(s,y)|^r \,\d y \d s\bigg)^\frac{q_1}{r} \,\frac{\d t}{t} \bigg)^\frac{1}{q_1}\\
        &\simeq\bigg(\int\limits_0^\infty t^{-\beta_1q_1} \bigg(\fint\limits_{\frac{t}{2}}^t \fint\limits_{B(x,t)} | f(s,y)|^r \,\d y \d s\bigg)^\frac{q_1}{r} \,\frac{\d t}{t} \bigg)^\frac{1}{q_1} \\
        &=\bigg(\int\limits_0^\infty t^{-\beta_1q_1} |F(t,x)|^{q_1} \,\frac{\d t}{t} \bigg)^\frac{1}{q_1},
    \end{align*}
    where we define
    \begin{align*}
        F(t,x) \coloneqq \bigg(\fint\limits_{\frac{t}{2}}^t \fint\limits_{B(x,t)} | f(s,y)|^r \,\d y \d s\bigg)^\frac{1}{r}.
    \end{align*}
    By \cite[Prop.\@ 1.1.4]{Grafakos}, we can use the description of the $\L^{p_1}$-norm in terms of the distribution function to obtain
    \begin{align}
    \label{eq: H}
        \|f\|^{p_1}_{\T^{p_1,q_1,r}_{\beta_1}} \simeq \int\limits_0^\infty \lambda^{p_1-1} \bigg| \bigg\{x\in\IR^d:  \bigg(\int\limits_0^\infty t^{-\beta_1q_1} |F(t,x)|^{q_1} \,\frac{\d t}{t} \bigg)^\frac{1}{q_1} >\lambda \bigg\} \bigg| \,\d \lambda.
    \end{align}
    Thus, we will further analyze the distribution function
    \begin{align}
    \label{eq: distr func}
        \lambda\mapsto \bigg| \bigg\{x\in\IR^d: \bigg(\int\limits_0^\infty t^{-\beta_1q_1} |F(t,x)|^{q_1} \,\frac{\d t}{t} \bigg)^\frac{1}{q_1} >\lambda \bigg\} \bigg|.
    \end{align}
    For some fixed but arbitrary $c>0$, we calculate
    \begin{align}
    \label{eq: H2}
        \bigg(\int\limits_0^\infty t^{-\beta_1q_1} |F(t,x)|^{q_1} \,\frac{\d t}{t} \bigg)^\frac{1}{q_1} \leq \bigg(\int\limits_0^c t^{-\beta_1q_1} |F(t,x)|^{q_1} \,\frac{\d t}{t} \bigg)^\frac{1}{q_1} + \bigg(\int\limits_c^\infty t^{-\beta_1q_1} |F(t,x)|^{q_1} \,\frac{\d t}{t} \bigg)^\frac{1}{q_1}.
    \end{align}
    Since we assumed $\beta_0=0$, we have $\beta_1=\beta_1-\beta_0 = \frac{d}{p_1}-\frac{d}{p_0}<0$. Thus, we estimate the first integral on the right-hand side
    \begin{align*}
        \bigg(\int\limits_0^c t^{-\beta_1q_1} |F(t,x)|^{q_1} \,\frac{\d t}{t} \bigg)^\frac{1}{q_1}&\leq \sup\limits_{t>0}|F(t,x)| \bigg(\int\limits_0^c t^{-\beta_1q_1-1}  \,\d t \bigg)^\frac{1}{q_1}\\
        &=A_1 c^{-\beta_1}\sup\limits_{t>0}|F(t,x)|
    \end{align*}
    for some generic constant $A_1>0$ which does not depend on $c$ and $f$. Furthermore, by \eqref{eq: embedding}, we can bound the second integral on the right-hand side of \eqref{eq: H2}
    \begin{align*}
        \bigg(\int\limits_c^\infty t^{-\beta_1q_1} |F(t,x)|^{q_1} \,\frac{\d t}{t} \bigg)^\frac{1}{q_1} &= \bigg(\int\limits_c^\infty t^{-(\frac{d}{p_1}-\frac{d}{p_0})q_1} |F(t,x)|^{q_1} \,\frac{\d t}{t} \bigg)^\frac{1}{q_1}\\
        &\leq \sup\limits_{x\in\IR^d}\sup\limits_{t>0} \Big(t^{\frac{d}{p_0}}|F(t,x)|\Big) \bigg(\int\limits_c^\infty t^{-\frac{dq_1}{p_1}} \,\frac{\d t}{t} \bigg)^\frac{1}{q_1}\\
        &\leq A_2c^{-\frac{d}{p_1}}
    \end{align*}
    for some generic constant $A_2>0$ which does not depend on $c$ and $f$. Plug both integral estimates back into \eqref{eq: H2} to obtain
    \begin{align}
    \label{eq: H3}
        &\bigg| \bigg\{x\in\IR^d: \bigg(\int\limits_0^\infty t^{-\beta_1q_1} |F(t,x)|^{q_1} \,\frac{\d t}{t} \bigg)^\frac{1}{q_1} >\lambda \bigg\} \bigg| \nonumber\\
        &\leq \bigg| \bigg\{x\in\IR^d: A_1 c^{-\beta_1}\sup\limits_{t>0}|F(t,x)| >\frac{\lambda}{2} \bigg\} \bigg|+\bigg| \bigg\{x\in\IR^d: A_2c^{-\frac{d}{p_1}} >\frac{\lambda}{2}  \bigg\} \bigg|
    \end{align}
    for all $\lambda,c>0$. Now, for a fixed $\lambda>0$, we choose $c = (\frac{\lambda}{2A_2})^{-\frac{p_1}{d}}$. Then, it follows
    \begin{align*}
        \bigg| \bigg\{x\in\IR^d: A_2c^{-\frac{d}{p_1}} >\frac{\lambda}{2}  \bigg\} \bigg| = 0,
    \end{align*}
    and we have
    \begin{align*}
        A_1c^{-\beta_1} = A_1\Big(\frac{\lambda}{2A_2}\Big)^{\frac{p_1\beta_1}{d}} = A\lambda^{1-\frac{p_1}{p_0}}
    \end{align*}
    for some new generic constant $A>0$. In this case, the inequality \eqref{eq: H3} reduces to
    \begin{align*}
        \bigg| \bigg\{x\in\IR^d: \bigg(\int\limits_0^\infty t^{-\beta_1q_1} |F(t,x)|^{q_1} \,\frac{\d t}{t} \bigg)^\frac{1}{q_1} >\lambda \bigg\} \bigg| &\leq \bigg| \bigg\{x\in\IR^d: A\lambda^{1-\frac{p_1}{p_0}}\sup\limits_{t>0}|F(t,x)| >\frac{\lambda}{2} \bigg\} \bigg|\\
        &=\bigg| \bigg\{x\in\IR^d: \sup\limits_{t>0}|F(t,x)| >\frac{\lambda^{\frac{p_1}{p_0}}}{2A} \bigg\} \bigg|.
    \end{align*}
    Using this estimate in the distributional description \eqref{eq: H}, we get
    \begin{align*}
        \|f\|^{p_1}_{\T^{p_1,q_1,r}_{\beta_1}} &\lesssim \int\limits_0^\infty \lambda^{p_1-1} \bigg| \bigg\{x\in\IR^d: \sup\limits_{t>0}|F(t,x)| >\frac{\lambda^{\frac{p_1}{p_0}}}{2A} \bigg\} \bigg| \,\d \lambda\\
        &\simeq \int\limits_0^\infty \mu^{p_0-1} \bigg| \bigg\{x\in\IR^d: \sup\limits_{t>0}|F(t,x)| >\mu \bigg\} \bigg| \,\d \mu\\
        &\simeq \|\sup\limits_{t>0}|F(t,\cdot)| \|_{\L^{p_0}}^{p_0}\\
        &= \|f\|^{p_0}_{\T^{p_0,\infty,r}_{0}}\\
        &=1.
    \end{align*}
    Since this estimate holds for all $f\in \T^{p_0,\infty,r}_0 $ with $\|f\|_{\T^{p_0,\infty,r}_0} =1$, we conclude with a homogeneity argument
    \begin{align*}
        \|f\|_{\T^{p_1,q_1,r}_{\beta_1}} \lesssim \|f\|_{\T^{p_0,\infty,r}_0}
    \end{align*}
    for all $f\in \T^{p_0,\infty,r}_0 $. This proves \eqref{eq: H0}, and therefore the embedding \eqref{eq: HSE} in the case $p_1<\infty$.\\

    \noindent \textbf{Step 2: We show \eqref{eq: HSE} for $p_1=\infty$.} 
    First, we assume $1<p_0<\infty$ and $1<q_0,q_1 ,r_0,r_1\leq \infty$. Then, $1<p_0'<\infty$ and $1\leq q_0',q_1' ,r_0',r_1' < \infty$ and
    \begin{align*}
        -\beta_1 - (-\beta_0) = \beta_0-\beta_1 = \frac{d}{p_0} = d\Big(1- \frac{1}{p_0'} \Big).
    \end{align*}
    Thus, by \eqref{eq: HSE}, we have
    \begin{align*}
        \T^{1,q_1',r_1'}_{-\beta_1} \hookrightarrow \T^{p_0',q_0',r_0'}_{-\beta_0}.
    \end{align*}
    Since all parameters $p_0',q_i',r_i'\geq  1$, we can use Proposition~\ref{prop: duality Banach range} to get
    \begin{align}
    \label{eq: HSE1}
        \T^{p_0,q_0,r_0}_{\beta_0} = (\T^{p_0',q_0',r_0'}_{-\beta_0})' \hookrightarrow (\T^{1,q_1',r_1'}_{-\beta_1})' = \T^{\infty,q_1,r_1}_{\beta_1}.
    \end{align}
    This proves \eqref{eq: HSE} for $1<p_0< p_1=\infty$ and $1<q_0,q_1 ,r_0,r_1\leq \infty$.\\
    Now, assume $0<p_0<\infty$, $0<q_0,q_1 ,r_0,r_1\leq \infty$. Choose $M>0$ such that
    \begin{align*}
        1<\frac{p_0}{M},\frac{q_0}{M},\frac{q_1}{M},\frac{r_0}{M},\frac{r_1}{M} \leq\infty.
    \end{align*}
    Due to Lemma~\ref{lem: convex tent space} in conjunction with \eqref{eq: HSE1}, we get
    \begin{align*}
        \|f\|_{\T^{\infty,q_1,r_1}_{\beta_1}} = \||f|^M\|^{\frac{1}{M}}_{\T^{\infty,\frac{q_1}{M},\frac{r_1}{M}}_{M\beta_1}} \lesssim \||f|^M\|^{\frac{1}{M}}_{\T^{\frac{p_0}{M},\frac{q_0}{M},\frac{r_0}{M}}_{M\beta_0}} = \|f\|_{{\T^{p,q_0,r_0}_{\beta_0}}},
    \end{align*}
    which proves \eqref{eq: HSE} in the full range of parameters.
\end{proof}

We finish this section by presenting mixed-type embeddings for tent and $\Z$-spaces. There are two of them. The first (Lemma~\ref{lem: embedding tent into Z}) is a mixed-type embedding in the parameter $q$ that resembles a nesting property. The second (Theorem~\ref{thm: mixed embedding}) is a mixed-type Hardy--Littlewood--Sobolev embedding. Its proof relies on the real interpolation results for tent spaces in Theorem~\ref{thm: real interpolation in p} and Proposition~\ref{prop: real interpolation tent spaces full range}, which will be proven in Section~\ref{sec: real interpolation} below.

\begin{lemma}
\label{lem: embedding tent into Z}
    For $0<p, q, r\leq \infty$ and $\beta\in\IR$ we have the continuous embeddings
    \begin{align}
    \label{eq: embedding tent into Z}
         \Z^{p,\min\{p,q\},r}_\beta\hookrightarrow \T^{p,q,r}_\beta \hookrightarrow \Z^{p,\max\{p,q\},r}_\beta.
    \end{align}
\end{lemma}

\begin{proof}
    In the case $0<p<\infty$, the result was already proven in \cite[Lem.\@ 4.3]{Auscher_Bechtel_Haardt}. Hence, we focus on the case $p=\infty$. First, observe that Lemma~\ref{lem: embedding in q} implies
    \begin{align*}
         \T^{\infty,q,r}_\beta \hookrightarrow \T^{\infty,\infty,r}_\beta = \Z^{\infty,\infty,r}_\beta.
    \end{align*}
    This shows the second embedding of \eqref{eq: embedding tent into Z} since $\max\{p,q\} = \infty$. To show the first embedding of \eqref{eq: embedding tent into Z}, we estimate
    \begin{align*}
        \|f\|_{\T^{\infty,q,r}_\beta}&=\sup\limits_{\tau>0} \sup\limits_{y\in\IR^d} 
        \bigg(\int\limits_{0}^\tau  \fint\limits_{B(y,\tau)} \bigg(\fint\limits_{\frac{t}{2}}^t \fint\limits_{B(x,t)} |s^{-\beta}f(s,z)|^r \,\d z \d s\bigg)^\frac{q}{r}\,\frac{\d x\d t }{t}\bigg)^\frac{1}{q}\\
        &\leq \sup\limits_{\tau>0}  
        \bigg(\int\limits_{0}^\tau  \sup\limits_{x\in\IR^d} \bigg(\fint\limits_{\frac{t}{2}}^t \fint\limits_{B(x,t)} |s^{-\beta}f(s,z)|^r \,\d z \d s\bigg)^\frac{q}{r}\,\frac{\d t}{t}\bigg)^\frac{1}{q}\\
        &\leq  
        \bigg(\int\limits_{0}^\infty  \sup\limits_{x\in\IR^d} \bigg(\fint\limits_{\frac{t}{2}}^t \fint\limits_{B(x,t)} |s^{-\beta}f(s,z)|^r \,\d z \d s\bigg)^\frac{q}{r}\,\frac{\d t}{t}\bigg)^\frac{1}{q}\\
        &=\|f\|_{\Z^{\infty,q,r}_\beta},
    \end{align*}
    which concludes the claim.
\end{proof}

\begin{theorem}[Mixed-type embeddings]
\label{thm: mixed embedding}
    Let $0<p_0<p_1\leq \infty$, $0<q \leq \infty$, $0<r_1\leq r_0\leq \infty$ and $\beta_0,\beta_1\in \IR$ satisfying $\beta_0-\beta_1 = \frac{d}{p_0}-\frac{d}{p_1}$. Then, we have the continuous embedding
    \begin{align}
    \label{eq: mixed embed 1}
        \T^{p_0,q,r_0}_{\beta_0} \hookrightarrow \Z^{p_1,p_0,r_1}_{\beta_1},
    \end{align}
    as well as
    \begin{align}
    \label{eq: mixed embed 2}
        \Z^{p_0,p_1,r_0}_{\beta_0} \hookrightarrow \T^{p_1,q,r_1}_{\beta_1}.
    \end{align}
\end{theorem}

\begin{proof}
\noindent \textbf{Step 1: We show \eqref{eq: mixed embed 1}.}
    By Jensen's inequality and Lemma~\ref{lem: embedding in q}, the embedding \eqref{eq: mixed embed 1} can be reduced to show
    \begin{align*}
        \T^{p_0,\infty,r}_{\beta_0} \hookrightarrow \Z^{p_1,p_0,r}_{\beta_1}
    \end{align*}
    for any $0<r\leq \infty$. Now, by using Theorem~\ref{thm: real interpolation in p}, we obtain
    \begin{align}
    \label{eq: tent real int}
        \T^{p_0,\infty,r}_{\beta_0} = (\T^{\tilde p_0,\infty,r}_{\beta_0} ,\T^{\tilde p_1,\infty,r}_{\beta_0} )_{\theta,p_0},
    \end{align}
    for some $\theta\in(0,1)$ and $0<\tilde p_0<p_0<\tilde p_1<p_1<\infty$ such that $\frac{1}{p_0} = \frac{1-\theta}{\tilde p_0} + \frac{\theta}{\tilde p_1}$. Using the embedding of Lemma~\ref{lem: embedding tent into Z} and the weighted $\Z$-space embedding \cite[Thm.\@ 3.14]{Auscher_Bechtel_Haardt}, we get
    \begin{align*}
         \T^{\tilde p_0,\infty,r}_{\beta_0}\hookrightarrow \Z^{\tilde p_0,\infty,r}_{\beta_0} \hookrightarrow  \Z^{p_1,\infty,r}_{\tilde \beta_0}
    \end{align*}
    for $\tilde \beta_0 \coloneqq \beta_0+ \frac{d}{p_1} - \frac{d}{\tilde p_0} $, and 
    \begin{align*}
         \T^{\tilde p_1,\infty,r}_{\vphantom{\tilde \beta_1}\beta_0}\hookrightarrow \Z^{\tilde p_1,\infty,r}_{\vphantom{\tilde \beta_1}\beta_0} \hookrightarrow  \Z^{p_1,\infty,r}_{\tilde \beta_1}
    \end{align*}
    for $\tilde \beta_1= \beta_0+ \frac{d}{p_1} - \frac{d}{\tilde p_1} $. These embeddings, combined with \eqref{eq: tent real int} and the interpolation result for $\Z$-spaces in \cite[Thm.\@ 4.1]{Auscher_Bechtel_Haardt}, yield
    \begin{align*}
        \T^{p_0,\infty,r}_{\vphantom{\tilde \beta_1}\beta_0} = (\T^{\tilde p_0,\infty,r}_{\vphantom{\tilde \beta_1}\beta_0} ,\T^{\tilde p_1,\infty,r}_{\vphantom{\tilde \beta_1}\beta_0} )_{\theta,p_0}
        \hookrightarrow(\Z^{p_1,\infty,r}_{\tilde \beta_0} ,\Z^{p_1,\infty,r}_{\tilde \beta_1} )_{\theta,p_0}
        =\Z^{p_1,p_0,r}_{\vphantom{\tilde \beta_1}\beta_1},
    \end{align*}
    where we used
    \begin{align*}
        (1-\theta)\tilde \beta_0 + \theta\tilde \beta_1 &=(1-\theta) \Big(\beta_0+ \frac{d}{p_1}-\frac{d}{\tilde p_0} \Big) + \theta\Big(\beta_0+ \frac{d}{p_1}-\frac{d}{\tilde p_1} \Big)\\
        &= \beta_0 +\frac{d}{p_1}  -d\Big(\frac{1-\theta}{\tilde p_0} + \frac{\theta}{\tilde p_1}\Big) \\
        &= \beta_0+ \frac{d}{p_1}  -\frac{d}{p_0}\\
        &=\beta_1
    \end{align*}
    in the last line. This proves \eqref{eq: mixed embed 1}.\\
    
    \noindent \textbf{Step 2: We show \eqref{eq: mixed embed 2}.}
    By Jensen's inequality, the embedding \eqref{eq: mixed embed 1} can be reduced to
    \begin{align*}
        \Z^{p_0,p_1,r}_{\beta_0} \hookrightarrow \T^{p_1,q,r}_{\beta_1}
    \end{align*}
    for any $0<r\leq \infty$. Fix $0<\bar p_0<\bar p_1<\infty$ and $\theta\in (0,1)$ such that $0<p_0<\bar p_0<p_1< \bar p_1<\infty$ and
    \begin{align*}
        \frac{1}{p_1} = \frac{1-\theta}{\bar p_0} + \frac{\theta}{\bar p_1}.
    \end{align*}
    Moreover, we define
    \begin{align*}
        \bar \beta_0 \coloneqq \beta_1 + \frac{d}{p_0}- \frac{d}{\bar p_0}  \quad \text{and} \quad \bar \beta_1 \coloneqq \beta_1+ \frac{d}{p_0} - \frac{d}{\bar p_1} .
    \end{align*}
    Then, by Theorem~\ref{thm: Hardy-Sobolev embedding}, we have the embeddings
    \begin{align}
    \label{eq: mixed 1}
        \T^{p_0,q,r}_{\bar \beta_0} \hookrightarrow \T^{\bar p_0,q,r}_{\beta_1}\quad \text{and} \quad  \T^{p_0,q,r}_{\bar \beta_1} \hookrightarrow \T^{\bar p_1,q,r}_{\beta_1}.
    \end{align}
    Finally, notice that
    \begin{align*}
        (1-\theta)\bar \beta_0  + \theta \bar \beta_1 &=  \beta_1 + \frac{d}{p_0} - d\Big(\frac{1-\theta}{\bar p_0} + \frac{\theta}{\bar p_1}\Big)\\
        &=  \beta_1 + \frac{d}{p_0} - \frac{d}{p_1}\\
        &= \beta_0.
    \end{align*}
    Hence, the real interpolation results Theorem~\ref{thm: real interpolation in p} and Proposition~\ref{prop: real interpolation tent spaces full range}, together with \eqref{eq: mixed 1}, yield
    \begin{align*}
        \Z^{p_0,p_1,r}_{\beta_0} = (\T^{ p_0,q,r}_{\bar \beta_0} ,\T^{ p_0,q,r}_{\bar \beta_1} )_{\theta,p_1} \hookrightarrow(\T^{\bar p_0,q,r}_{\beta_1} ,\T^{\bar  p_1,q,r}_{ \beta_1} )_{\theta,p_1} = \T^{p_1,q,r}_{\beta_1},
    \end{align*}
    which finishes the proof
\end{proof}

\subsection{Duality: non-Banach range}
\label{sec: duality}

\noindent In this section, we develop a complete duality theory for the tent spaces $\T^{p,q,r}_\beta$. Specifically, we identify the dual spaces $(\T^{p,q,r}_\beta)'$ for parameters $p$ and/or $q$ in the non-Banach range. This complements Proposition~\ref{prop: duality Banach range}, where the parameters belong to the Banach range $1\leq p,q<\infty$. We begin with a duality theorem in which $1\leq p <\infty$ remains in the Banach range but $q$ does not. Recall the interpretation $q'=\infty$ in the case $0<q\leq 1$.

\begin{theorem}
\label{thm: duality p>1 but q<1}
    Let $1\leq p,r<\infty$, $0<q<\infty$ and $\beta\in\IR$. Then, we have
    \begin{align}
    \label{eq: dual estimate}
        \int\limits_0^\infty \int\limits_{\IR^d} |f(s,y)| \cdot |g(s,y)|\,\frac{\d y \d s}{s} 
        \lesssim\|f\|_{\T^{p',q',r'}_{-\beta} } \|g\|_{\T^{\vphantom{p',q',r'}p,q,r}_{\beta} }
    \end{align}
    for all measurable functions $f,g$. Moreover, we can identify $(\T^{p,q,r}_\beta)'\simeq \T^{\infty,q',r'}_{-\beta}$ with equivalent norms via the $\L^2$-duality pairing.
\end{theorem}

For the proof we need the following auxiliary lemma, which is a direct consequence of  Proposition~\ref{prop: dyadic char}.

\begin{lemma}
\label{lem: embedding into Lplq}
    Let $0<p,q,r<\infty$ and $\beta\in\IR$. Then, the map
    \begin{align*}
        i: \T^{p,q,r}_\beta \to \L^p(\ell^q), \quad f\mapsto  \mathbf{1}_{Q}(x)|Q|^{-\frac{\beta }{d}} \|f\|_{\L^r\big(\bar Q , \frac{\d y \d s}{s^{d+1}}\big)}
    \end{align*}
    satisfies $\|i(f)\|_{\L^p(\ell^q)}\simeq \|f\|_{{\T^{p,q,r}_\beta }}$. In particular, it is bounded, injective and has closed range.
\end{lemma}

\begin{proof}[Proof of Theorem~\ref{thm: duality p>1 but q<1}]
    According to Proposition~\ref{prop: duality Banach range}, it suffices to assume $0<q<1$. We will divide the proof into two steps. First, we show the inclusion $\T^{p',\infty,r'}_{-\beta} \subset(\T^{p,q,r}_\beta)'$. Here, we use the nesting property (Lemma~\ref{lem: embedding in q}) to embed the quasi-Banach tent space into a tent space in the Banach range. For the reverse inclusion $(\T^{p,q,r}_\beta)'\subset \T^{p',\infty,r'}_{-\beta} $, we will embed the tent space into a vector-valued Lebesgue space using Lemma~\ref{lem: embedding into Lplq}, for which duality results are already known.\\
    
    \noindent \textbf{Step 1: We show $\T^{p',\infty,r'}_{-\beta} \subset(\T^{p,q,r}_\beta)'$.}
    By Lemma~\ref{lem: embedding in q}, we have the continuous embedding
    \begin{align*}
        \T^{p,q,r}_\beta \subset \T^{p,1,r}_\beta.
    \end{align*}
    This, together with Proposition~\ref{prop: duality Banach range}, yields
    \begin{align*}
        \int\limits_0^\infty \int\limits_{\IR^d} |f(s,y)| \cdot |g(s,y)|\,\frac{\d y \d s}{s} \leq \|f\|_{\T^{p',\infty,r'}_{-\beta} } \|g\|_{\T^{\vphantom{p',q',r'}p,1,r}_{\beta} }\lesssim\|f\|_{\T^{p',\infty,r'}_{-\beta} }  \|g\|_{\T^{\vphantom{p',q',r'}p,q,r}_{\beta} }.
    \end{align*}
    This proves \eqref{eq: dual estimate} and the inclusion $\T^{p',\infty,r'}_{-\beta} \subset(\T^{p,q,r}_\beta)'$.\\

    \noindent \textbf{Step 2: We show $(\T^{p,q,r}_\beta)'\subset \T^{p',\infty,r'}_{-\beta} $.}
    Let $\Phi \in (\T^{p,q,r}_\beta)'$. Since compactly supported $\L^r$-functions embed continuously into $\T^{p,q,r}_\beta$, see Proposition~\ref{prop: basic props}, we observe that $\Phi \in \big(\L^r\big(K,\frac{\d y \d s}{s}\big)\big)'$ for each compact set $K\subset \IR^{d+1}_+$. Thus, we can find $f_K\in \L^{r'}\big(K,\frac{\d y \d s}{s}\big)$ such that
    \begin{align*}
        \Phi(g) = \iint\limits_{K} f_K(s,y) \, \overline{g(s,y)} \, \frac{\d y \d s}{s}, 
    \end{align*}
    for every $g\in \L^r\big(K,\frac{\d y \d s}{s}\big)$. Exhausting the whole $\IR^{d+1}_+$ by compact sets, we may find a global function $f\in \L^{r'}_{\loc}\big(\IR^{d+1}_+,\frac{\d y \d s}{s}\big)$ with the property
    \begin{align}
    \label{eq: rep 0}
        \Phi(g) = \iint\limits_{\IR^{d+1}_+} f(s,y) \, \overline{g(s,y)} \, \frac{\d y \d s}{s}
    \end{align}
    for all compactly supported $g\in \L^r\big(\IR^{d+1}_+,\frac{\d y \d s}{s}\big)$. Hence, by density (see Proposition~\ref{prop: basic props}), it suffices to show $f\in \T^{p',\infty,r'}_{-\beta}$. To do so, we first use Lemma~\ref{lem: embedding into Lplq} to see that $\Phi\circ i^{-1}$ is a bounded linear functional on $\Rg(i)$. Then, by the Hahn--Banach theorem, we can find an extension $\tilde \Phi \in (\L^p(\ell^q))'$ with $\|\tilde \Phi\| = \|\Phi\circ i^{-1}\|$ and $\tilde \Phi = \Phi\circ i^{-1}$ on $\Rg(i)$, or equivalently,
    \begin{align}
    \label{eq: rep 1}
        \tilde \Phi\circ i = \Phi \quad \text{on} \quad \T^{p,q,r}_\beta.
    \end{align}
    Due to the duality result \cite[Prop.\@ 2.11.1]{Triebel1}, we may identify $(\L^p(\ell^q))' \simeq \L^{p'}(\ell^\infty)$. Hence, we find a sequence $(F_Q)_{Q\in\square}\in \L^{p'}(\ell^{\infty})$ such that $\|F_Q\|_{\L^{p'}(\ell^{\infty})}=\|\tilde \Phi\|$ and $\tilde \Phi$ can be represented uniquely as
    \begin{align}
    \label{eq: rep 2}
        \tilde \Phi ((g_Q)_{Q}) = \int\limits_{\IR^d} \sum\limits_{Q\in\square} F_Q(x) \, \overline{g_Q(x)}\,\d x
    \end{align}
    for all $(g_Q)_{Q\in\square}\in \L^p(\ell^q)$. Combining \eqref{eq: rep 0}, \eqref{eq: rep 1} and \eqref{eq: rep 2}, we get
    \begin{align}
    \label{eq: rep 4}
        \iint\limits_{\IR^{d+1}_+} f(s,y) \, \overline{g(s,y)} \, \frac{\d y \d s}{s} = \int\limits_{\IR^d} \sum\limits_{Q\in\square} F_Q(x) \, \mathbf{1}_{Q}(x)|Q|^{-\frac{\beta }{d}} \|g\|_{\L^r\big(\bar Q , \frac{\d y \d s}{s^{d+1}}\big)}\,\d x
    \end{align}
    for all compactly supported $g\in \L^r\big(\IR^{d+1}_+,\frac{\d y \d s}{s}\big)$.
    Next, fix $Q\in\square$ arbitrary and $g\in \L^{r}\big(\bar Q, \frac{\d y \d s}{s}\big)$ with norm $\|g\|_{\L^{r}(\bar Q, \frac{\d y \d s}{s})} \leq 1 $. With \eqref{eq: rep 4} it follows
    \begin{align}
    \label{eq: id D}
       \bigg|\iint\limits_{\bar Q} f(s,y) \, \overline{g(s,y)} \, \frac{\d y \d s}{s}\bigg|&=\bigg|\int\limits_{Q}  F_Q(x) \, |Q|^{-\frac{\beta }{d}} \| g\|_{\L^{r}\big(\bar Q , \frac{\d y \d s}{s^{d+1}}\big)}\,\d x\bigg|\nonumber \\
       &\lesssim \int\limits_{Q}  |F_Q(x)| \,|Q|^{-\frac{1}{r}-\frac{\beta }{d}} \| g\|_{\L^{r}\big(\bar Q , \frac{\d y \d s}{s}\big)}\,\d x\\
       &\leq |Q|^{-\frac{1}{r}-\frac{\beta }{d}}\int\limits_{Q}  |F_Q(x)| \,\d x.\nonumber
    \end{align}
    Since $g$ was arbitrary, we conclude, by taking the supremum over all such $g$, that
    \begin{align}
    \label{eq: est D}
        |Q|^{\frac{\beta}{d}}\|f\|_{\L^{r'}\big(\bar Q , \frac{\d y \d s}{s^{d+1}}\big)}\lesssim |Q|^{\frac{\beta}{d}-\frac{1}{r'}}\|f\|_{\L^{r'}\big(\bar Q , \frac{\d y \d s}{s}\big)}
        \lesssim  \fint\limits_{Q}  |F_Q(x)| \,\d x.
    \end{align}
    Since $Q\in\square$ was arbitrary, \eqref{eq: est D} in conjunction with Proposition~\ref{prop: dyadic char}   yields
    \begin{align*} 
        \|f\|_{\T^{p',\infty,r'}_{-\beta}} &\simeq \bigg\|\sup\limits_{Q\in\square} \mathbf{1}_{Q}(\cdot)|Q|^{\frac{\beta}{d}}\|f\|_{\L^{r'}\big(\bar Q , \frac{\d y \d s}{s^{d+1}}\big)} \bigg\|_{\L^{p'}}\\
        &\lesssim  \bigg\|\sup\limits_{Q\in\square} \mathbf{1}_{Q}(\cdot)\fint\limits_{Q}  |F_Q(y)| \,\d y \bigg\|_{\L^{p'}}\\
        &\leq  \bigg\|\sup\limits_{Q\in\square} \mathcal{M}\big( |F_Q|\big)(\cdot) \bigg\|_{\L^{p'}}.
    \end{align*}
    Because $1<p'\leq \infty$, we can apply Fefferman--Stein's inequality, see \cite[Thm.\@ 2.1]{Ullrich1}, to obtain
    \begin{align*}
        \|f\|_{\T^{p',\infty,r'}_{-\beta}} &\lesssim\Big\|\sup\limits_{Q\in\square} |F_Q(\cdot)\big| \Big\|_{\L^{p'}} = \|(F_Q)_{Q\in\square}\|_{\L^{p'}(\ell^\infty)}= \|\tilde\Phi\|\leq \|\Phi\|.
    \end{align*}
    Hence, $f\in \T^{p',\infty,r'}_{-\beta}$, which proves the claim.
\end{proof}

Finally, we complete the duality theory for $\T^{p,q,r}_\beta$ by treating the case $0<p<1$. In that case, we can identify $(\T^{p,q,r}_\beta)'$ with a $\Z$-space from Definition~\ref{def: Z-space}. This is in line with the Besov--Triebel--Lizorkin duality theory, see Section~\ref{sec: introduction}.

\begin{theorem}
\label{thm: duality p<1}
    Let $0<p<1$, $0<q<\infty$, $1\leq r<\infty$ and $\beta\in\IR$. Then, we have
    \begin{align}
    \label{eq: duality}
        \int\limits_0^\infty \int\limits_{\IR^d} |f(s,y)| \cdot |g(s,y)|\,\frac{\d y \d s}{s} 
        \lesssim\|f\|_{\Z^{\infty,\infty,r'}_{-\beta+d(\frac{1}{p}-1)} } \|g\|_{\T^{\vphantom{\infty,\infty,r'}p,q,r}_{\beta} }
    \end{align}
    for all measurable functions $f,g$. Moreover, we can identify $(\T^{p,q,r}_\beta)'\simeq \Z^{\infty,\infty,r'}_{-\beta+d(\frac{1}{p}-1)}$ with equivalent norms via the $\L^2$-duality pairing.
\end{theorem}

\begin{proof}
    First, we observe with Theorem~\ref{thm: Hardy-Sobolev embedding} and Lemma~\ref{lem: embedding tent into Z} that
    \begin{align}
    \label{eq: DZ}
        \Z^{p,\min\{p,q\},r}_\beta \subset \T^{p,q,r}_{\beta} \subset \T^{1,1,r}_{\beta-d(\frac{1}{p}-1)}.
    \end{align}
    This, in conjunction with the duality theory of $\Z$-spaces (see \cite[Thm.\@ 3.16]{Auscher_Bechtel_Haardt}) and Proposition~\ref{prop: duality Banach range}, we get
    \begin{align*}
        \Z^{\infty,\infty,r'}_{-\beta+d(\frac{1}{p}-1)} =\T^{\infty,\infty,r'}_{-\beta+d(\frac{1}{p}-1)} &= (\T^{1,1,r}_{\beta-d(\frac{1}{p}-1)})' \\
        &\subset (\T^{p,q,r}_{\beta})'\\
        &\subset (\Z^{p,\min\{p,q\},r}_\beta)' =\Z^{\infty,\infty,r'}_{-\beta+d(\frac{1}{p}-1)}.
    \end{align*}
    Hence, we can identify $(\T^{p,q,r}_\beta)'$ with $\Z^{\infty,\infty,r'}_{-\beta+d(\frac{1}{p}-1)}$ via the $\L^2$-duality pairing. To show \eqref{eq: duality}, we introduce an extra average, use Fubini's theorem and Hölder's inequality, to get
    \begin{align*}
        &\int\limits_0^\infty \int\limits_{\IR^d} |f(s,y)| \cdot |g(s,y)|\,\frac{\d y \d s}{s} = \int\limits_0^\infty \int\limits_{\IR^d} \fint\limits_{s}^{2s} \fint\limits_{B(y,s)} |f(s,y)| \cdot |g(s,y)|\,\d x \d t \,\frac{\d y \d s}{s}\\
        &\lesssim \int\limits_0^\infty \int\limits_{\IR^d} \fint\limits_{\frac{t}{2}}^t \fint\limits_{B(x,t)} |f(s,y)| \cdot |g(s,y)|\,\d y \d s \,\frac{\d x \d t}{t}\\
        &\leq \int\limits_0^\infty \int\limits_{\IR^d} \bigg(\fint\limits_{\frac{t}{2}}^t \fint\limits_{B(x,t)} |s^{\beta-d(\frac{1}{p}-1)}f(s,y)|^{r'}\,\d y \d s\bigg)^\frac{1}{r'} \bigg(\fint\limits_{\frac{t}{2}}^t \fint\limits_{B(x,t)} |s^{-\beta+d(\frac{1}{p}-1)}g(s,y)|^{r}\,\d y \d s\bigg)^\frac{1}{r} \,\frac{\d x \d t}{t}\\
        &\leq \|f\|_{\Z^{\infty,\infty,r'}_{-\beta+d(\frac{1}{p}-1)}} \|g\|_{\T^{1,1,r}_{\beta-d(\frac{1}{p}-1)}}.
    \end{align*}
    Finally, by the second embedding of \eqref{eq: DZ}, we get
    \begin{align*}
        \int\limits_0^\infty \int\limits_{\IR^d} |f(s,y)| \cdot |g(s,y)|\,\frac{\d y \d s}{s}\leq \|f\|_{\Z^{\infty,\infty,r'}_{-\beta+d(\frac{1}{p}-1)}} \|g\|_{\T^{\vphantom{\infty,\infty,r'}1,1,r}_{\beta-d(\frac{1}{p}-1)}}\lesssim \|f\|_{\Z^{\infty,\infty,r'}_{-\beta+d(\frac{1}{p}-1)}}\ \|g\|_{\T^{\vphantom{\infty,\infty,r'}p,q,r}_{\beta}},
    \end{align*}
    which proves the claim.
\end{proof}

\begin{remark}
    It is known that the dual space of tent spaces $\T^{p,2}_\beta$ for $0<p<1$ can be identified with a tent space \enquote{beyond infinity}, see for example \cites{Amenta2, Auscher_Egert}. For $0<q\leq \infty$ and $\alpha,\beta\in \IR$, these spaces are denoted by $\T^{\infty,q}_{\beta;\alpha}$ and consist of all measurable functions $f$ such that
    \begin{align*}
        \|f\|_{\T^{\infty,q}_{\beta;\alpha}} \coloneqq \sup\limits_{t>0}\sup\limits_{x\in\IR^d}\frac{1}{t^\alpha} \bigg(\int\limits_0^t \fint\limits_{B(x,t)} |s^{-\beta}f(s,y)|^q\,\frac{\d y \d s}{s}\bigg)^\frac{1}{q}<\infty 
    \end{align*}
    with the usual modification in the infinite case.
    On the one hand, it was proven in \cite[Thm.\@ 1.11]{Amenta2} that, for $0<p<1$, we can identify $(\T^{p,2}_\beta)' \simeq \T^{\infty,2}_{-\beta;d(\frac{1}{p}-1)}$ with equivalence of norms. On the other hand, Theorem~\ref{thm: duality p<1} suggests that the dual space of $\T^{p,2}_\beta$ can be identified with a $\Z$-space. Indeed, for $q=r=2$, we have $\T^{p,q,r}_\beta = \T^{p,2}_\beta$, see \cite[Rem.\@ 3.2]{Auscher_Bechtel_Haardt}, such that
    \begin{align*}
        \T^{\infty,2}_{-\beta;d(\frac{1}{p}-1)}=(\T^{p,2}_\beta)' = (\T^{p,2,2}_\beta)' = \Z^{\infty,\infty,2}_{-\beta+d(\frac{1}{p}-1)}.
    \end{align*}
    This shows, that tent spaces \enquote{beyond infinity} are equivalent to certain $\Z$-spaces. This observation will be recorded in the next proposition. Additionally, we provide a direct proof of this phenomenon including a broader range of parameters.
\end{remark}

\begin{proposition}
    For $0<q\leq \infty$, $\alpha>0$ and $\beta\in \IR$ we have $\Z^{\infty,\infty,q}_{\beta+\alpha} = \T^{\infty,q}_{\beta;\alpha}$ with equivalent norms.
\end{proposition}

\begin{proof}
    We will show the assertion for $0<q<\infty$ since the case $q=\infty$ follows directly by definition of the spaces.\\
    
    \noindent \textbf{Step 1: $\T^{\infty,q}_{\beta;\alpha}\subset \Z^{\infty,\infty,q}_{\beta+\alpha}$.}
    Let $f\in \T^{\infty,q}_{\beta;\alpha}$ and fix $x\in \IR^d$ and $t>0$. Then, we get by a direct computation
    \begin{align*}
        \bigg(\fint\limits_{\frac{t}{2}}^t \fint\limits_{B(x,t)} |s^{-(\beta+\alpha)}f(s,y)|^q\,\d y \d s\bigg)^\frac{1}{q} \lesssim\frac{1}{t^\alpha} \bigg(\int\limits_0^t \fint\limits_{B(x,t)} |s^{-\beta}f(s,y)|^q\,\frac{\d y \d s}{s}\bigg)^\frac{1}{q}.
    \end{align*}
    Hence, taking suprema in $x$ and $t$ on both sides yields $\|f\|_{\Z^{\infty,\infty,q}_{\beta+\alpha}}\lesssim \|f\|_{\T^{\infty,q}_{\beta;\alpha}}$.\\

    \noindent \textbf{Step 2: $\Z^{\infty,\infty,q}_{\beta+\alpha} \subset \T^{\infty,q}_{\beta;\alpha}$.}
    Let $f\in \Z^{\infty,\infty,q}_{\beta+\alpha}$ and fix  $x\in \IR^d$ and $t>0$. Using a dyadic splitting gives us
    \begin{align}
    \label{eq: E1}
        \frac{1}{t^\alpha} \bigg(\int\limits_0^t \fint\limits_{B(x,t)} |s^{-\beta}f(s,y)|^q\,\frac{\d y \d s}{s}\bigg)^\frac{1}{q} &=\frac{1}{t^\alpha} \bigg(\sum\limits_{k=0}^\infty \,\int\limits_{2^{-k-1}t}^{2^{-k}t}  \fint\limits_{B(x,t)} |s^{-\beta}f(s,y)|^q\,\frac{\d y \d s}{s}\bigg)^\frac{1}{q}.
    \end{align}
    Now, for each $k\in\IN_0$ there exist $J_k\subset \IN$ and $(x_j)_{j\in J_k}\subset \IR^d$ such that
    \begin{align*}
        B(x, t) \subset \bigcup_{j\in J_k} B(x_j, 2^{-k}t) \quad \text{and} \quad |J_k| \leq C2^{kd}
    \end{align*}
    for some constant $C>0$ depending only on $d$. This covering allows us to estimate \eqref{eq: E1} by
    \begin{align*}
       \frac{1}{t^\alpha} \bigg(\int\limits_0^t \fint\limits_{B(x,t)} |s^{-\beta}f(s,y)|^q\,\frac{\d y \d s}{s}\bigg)^\frac{1}{q} &\lesssim \frac{1}{t^\alpha} \bigg(\sum\limits_{k=0}^\infty \, \int\limits_{2^{-k-1}t}^{2^{-k}t} \sum\limits_{j\in J_k} t^{-d}\int\limits_{B(x_j,2^{-k}t)} |s^{-\beta}f(s,y)|^q\,\frac{\d y \d s}{s}\bigg)^\frac{1}{q}.
    \end{align*}
    Since $s\simeq 2^{-k}t$, we can introduce average integrals to estimate the term on the right-hand side by
    \begin{align*}
        &\frac{1}{t^\alpha} \bigg(\sum\limits_{k=0}^\infty \, \int\limits_{2^{-k-1}t}^{2^{-k}t} \sum\limits_{j\in J_k} t^{-d}\int\limits_{B(x_j,2^{-k}t)} |s^{-\beta}f(s,y)|^q\,\frac{\d y \d s}{s}\bigg)^\frac{1}{q}\\
        & \lesssim  \bigg(\sum\limits_{k=0}^\infty 2^{-k(d+q\alpha)} \sum\limits_{j\in J_k}\frac{1}{(2^{-k}t)^{\alpha q}}\fint\limits_{2^{-k-1}t}^{2^{-k}t}  \fint\limits_{B(x_j,2^{-k}t)} |s^{-\beta}f(s,y)|^q\,\d y \d s\bigg)^\frac{1}{q}\\
        &\lesssim \bigg(\sum\limits_{k=0}^\infty 2^{-k(d+q\alpha)} \sum\limits_{j\in J_k} \,\fint\limits_{2^{-k-1}t}^{2^{-k}t}  \fint\limits_{B(x_j,2^{-k}t)} |s^{-(\beta+\alpha)}f(s,y)|^q\,\d y \d s\bigg)^\frac{1}{q}.
    \intertext{Taking suprema with respect to the average integrals and using the property $|J_k| \leq C2^{kd}$ implies}
        &\leq \bigg(\sum\limits_{k=0}^\infty 2^{-k(d+q\alpha)} \sum\limits_{j\in J_k} \sup\limits_{\tau>0}\sup\limits_{z \in\IR^d}\bigg(\fint\limits_{\frac{\tau}{2}}^{\tau}  \fint\limits_{B(z,\tau)} |s^{-(\beta+\alpha)}f(s,y)|^q\,\d y \d s\bigg)\bigg)^\frac{1}{q}\\
        &\lesssim \bigg(\sum\limits_{k=0}^\infty 2^{-kq\alpha} \|f\|^q_{\Z^{\infty,\infty,q}_{\beta+\alpha}}\bigg)^\frac{1}{q}\\
        &\lesssim \|f\|_{\Z^{\infty,\infty,q}_{\beta+\alpha}},
    \end{align*}
    where we used the fact that $\alpha>0$ in the last step. Thus, we have shown that
    \begin{align*}
         \frac{1}{t^\alpha} \bigg(\int\limits_0^t \fint\limits_{B(x,t)} |s^{-\beta}f(s,y)|^q\,\frac{\d y \d s}{s}\bigg)^\frac{1}{q}\lesssim \|f\|_{\Z^{\infty,\infty,q}_{\beta+\alpha}}
    \end{align*}
    holds for all $x\in\IR^d$ and $t>0$. Taking suprema in $x$ and $t$ concludes $\Z^{\infty,\infty,q}_{\beta+\alpha} \subset \T^{\infty,q}_{\beta;\alpha}$.
\end{proof}

\subsection{Real interpolation of tent spaces}
\label{sec: real interpolation}

\noindent In this section, we present a real interpolation theory for $\T^{p,q,r}_\beta$-spaces. First, we show that $\T^{p,q,r}_\beta$ forms a real interpolation scale in the parameter $p$, see Theorem~\ref{thm: real interpolation in p}. Its proof is based on the characterization via discrete local means (Proposition~\ref{prop: char with m}) and is inspired by \cite[Thm.\@ 6.4 + Cor.\@ 6.7]{Frazier_Jawerth}, where the authors show analogous results for Triebel–Lizorkin spaces. Second, it was already shown in \cite[Prop.\@ 4.4]{Auscher_Bechtel_Haardt} that $\Z^{p,q,r}_\beta$-spaces for $p<\infty$ can be obtained by real interpolation of $\T^{p,q,r}_\beta$-spaces, in the same way as Besov spaces can be derived from Triebel--Lizorkin spaces, see Section~\ref{sec: introduction}. However, there does not exist a corresponding characterization of the endpoint space $\Z^{\infty,q,r}_\beta$. We will treat the missing endpoint characterization in Proposition~\ref{prop: real interpolation tent spaces full range}.

We start with a brief introduction to real interpolation theory and refer to the monograph \cite[Chp.\@ 3]{Berg_Löfström} for a detailed exposition.  Let $(X_0,X_1)$ be an interpolation couple, that is, a pair of (quasi-)Banach spaces $X_0$ and $X_1$ that are continuous embedded in a Hausdorff topological vector space. Then we define the (quasi-)Banach space $X_0+X_1 = \{x_0+x_1: x_0\in X_0, x_1\in X_1\}$  equipped with the (quasi-)norm
\begin{align*}
    \|x\|_{X_0+X_1} \coloneqq \inf\limits_{\substack{x_0\in X_0, x_1\in X_1\\ x=x_0+x_1}} \|x_0\|_{X_0} + \|x_1\|_{X_1}.
\end{align*}
For every $x\in X_0 + X_1$ and $t>0$ we define the $K$-functional
\begin{align*}
    K(t,x,X_0,X_1) \coloneqq \inf\limits_{\substack{x_0\in X_0, x_1\in X_1\\ x=x_0+x_1}} \|x_0\|_{X_0} + t\|x_1\|_{X_1}.
\end{align*}
Moreover, we define the $K_\infty$-functional 
\begin{align*}
    K_\infty(t,x,X_0,X_1) \coloneqq \inf\limits_{\substack{x_0\in X_0, x_1\in X_1\\ x=x_0+x_1}}\max\Big\{ \|x_0\|_{X_0} , t\|x_1\|_{X_1}\Big\}.
\end{align*}
It is easy to see that $K_\infty(t,x,X_0,X_1) \simeq K(t,x,X_0,X_1)$. Lastly, we define the $E$-functional
\begin{align*}
    E(t,x,X_0,X_1) \coloneqq  \inf\limits_{\substack{x_1\in X_1\\ \|x_1\|_{X_1}\leq t}} \|x-x_1\|_{X_0}.
\end{align*}
It was observed in \cite[Sec.\@ 6]{Frazier_Jawerth}, that $K_\infty$ is the right inverse of $E$ and that, consequently, one has
\begin{align}
\label{eq: K E}
    K_\infty(t,x,X_0,X_1)\leq C K_\infty(t,y,Y_0,Y_1)
\end{align}
for all $t>0$ if and only if
\begin{align}
\label{eq: E K}
    E(Ct,x,X_0,X_1)\leq C E(t,y,Y_0,Y_1)
\end{align}
for all $t>0$.
Finally, for $0<\theta<1$ and $0<q\leq \infty$, we define the real interpolation space $(X_0,X_1)_{\theta,q}$ as the set of all $x\in X_0+X_1$ such that
\begin{align}
\label{eq: def real int space}
    \|x\|_{\theta,q} \coloneqq \bigg(\int\limits_0^\infty t^{-\theta q}  \big[K(t,x,X_0,X_1)\big]^q \, \frac{\d t}{t}\bigg)^\frac{1}{q}<\infty.
\end{align}
It is known that $(X_0,X_1)_{\theta,q}$ is a (quasi-)Banach space equipped with the (quasi-)norm $\|\cdot\|_{\theta,q}$. Our main result reads as follows.

\begin{theorem}
\label{thm: real interpolation in p}
    Let $0< p_0<p_1\leq \infty$ and $0<q,r\leq \infty$. Furthermore, let $\beta\in\IR$ and $\theta\in (0,1)$. Then, we have
    \begin{align*}
        (\T^{p_0,q,r}_{\beta} , \T^{p_1,q,r}_{\beta} )_{\theta, p} = \T^{p,q,r}_\beta,
    \end{align*}
    where $\frac{1}{p} = \frac{1-\theta}{p_0} + \frac{\theta}{p_1}$.
\end{theorem}

The strategy to prove this theorem is similar to the corresponding result for Triebel--Lizorkin spaces, see \cite[Thm.\@ 6.4]{Frazier_Jawerth}. There, one uses the so-called $\varphi$-transform to reduce matters to sequence spaces and standard interpolation theory for $\L^p$-spaces. In our situation, we will use the discrete characterization from Section~\ref{sec: dyadic char} to use their derived identities and imitate their procedure. However, the reader is advised to keep a copy of the mentioned reference at hand. We start with the following key result.

\begin{theorem}
\label{thm: intermediate result}
    Let $0< p_0<\infty$, $0<q,r\leq \infty$ and $\beta\in\IR$. Then, for every $f \in \T^{p_0,q,r}_{\beta} + \T^{\infty,q,r}_{\beta}$ we have
    \begin{align*}
        K(t,f,\T^{p_0,q,r}_{\beta} , \T^{\infty,q,r}_{\beta}) \simeq K(t,m^{q,r}_{\beta,\frac{1}{4}}(f),\L^{p_0} , \L^{\infty}).
    \end{align*}
\end{theorem}

\begin{proof}
    For simplicity, we present the proof only for finite exponents $q$ and $r$. The case of infinite exponents follows with the usual modifications. We divide the proof into two steps.\\

    \noindent \textbf{Step 1: $ K(t,m^{q,r}_{\beta,\frac{1}{4}}(f),\L^{p_0} , \L^{\infty})\lesssim K(t,f,\T^{p_0,q,r}_{\beta} , \T^{\infty,q,r}_{\beta}) $}.
    Let $f= f_0 + f_1 \in \T^{p_0,q,r}_{\beta} + \T^{\infty,q,r}_{\beta}$ be an arbitrary but fixed decomposition. Recall the definitions of $m^{q,r}_{\beta,Q,\frac{1}{4}}(f)$ and $m^{q,r}_{\beta,\frac{1}{4}}(f)$ in \eqref{eq: def of m1/4} and \eqref{eq: def of m}. Then, by subadditivity, we get
    \begin{align*}
        m^{q,r}_{\beta,Q,\frac{1}{4}}(f) \leq C_q \big(m^{q,r}_{\beta,Q,\frac{1}{8}}(f_0) + m^{q,r}_{\beta,Q,\frac{1}{8}}(f_1) \big)
    \end{align*}
    for some constant $C_q>0$ depending only on $q$, which implies
    \begin{align}
    \label{eq: m subadd}
        m^{q,r}_{\beta,\frac{1}{4}}(f)(x) \leq C_q \big(m^{q,r}_{\beta,\frac{1}{8}}(f_0)(x) + m^{q,r}_{\beta,\frac{1}{8}}(f_1)(x) \big)
    \end{align}
    for all $x\in\IR^d$. Next, define the functions
    \begin{align*}
        g_0(x)\coloneqq \min\{m^{q,r}_{\beta,\frac{1}{4}}(f)(x) , C_q m^{q,r}_{\beta,\frac{1}{8}}(f_0)(x) \}
        \quad \text{and} \quad g_1(x) = m^{q,r}_{\beta,\frac{1}{4}}(f)(x)- g_0(x).
    \end{align*}
    First, notice that $m^{q,r}_{\beta,\frac{1}{4}}(f) = g_0 + g_1$. Second, by Proposition~\ref{prop: char with m}, we have
    \begin{align*}
        \|g_0\|_{\L^{p_0}} \leq C_q\|  m^{q,r}_{\beta,\frac{1}{8}}(f_0)\|_{\L^{p_0}}\simeq\|f_0\|_{ \T^{p_0,q,r}_{\beta}}.
    \end{align*}
    Third, by a case distinction for $g_0$, and \eqref{eq: m subadd}, we have
    \begin{align*}
        |g_1(x)|\leq |m^{q,r}_{\beta,\frac{1}{4}}(f)(x)-   C_q m^{q,r}_{\beta,\frac{1}{8}}(f_0)(x) |\leq C_q |m^{q,r}_{\beta,\frac{1}{8}}(f_1)(x)|
    \end{align*}
    for all $x\in\IR^d$. Thus, by Proposition~\ref{prop: char with m}, it follows
    \begin{align*}
        \|g_1\|_{\L^{\infty}} \leq   C_q \|m^{q,r}_{\beta,\frac{1}{8}}(f_1)\|_{\L^{\infty}} \simeq\|f_1\|_{ \T^{\infty,q,r}_{\beta}}.
    \end{align*}
    Using both norm estimates of $g_0$ and $g_1$ yields
    \begin{align*}
         K(t,m^{q,r}_{\beta,\frac{1}{4}}(f),\L^{p_0} , \L^{\infty}) &\leq \|g_0\|_{\L^{p_0}}  + t\|g_1\|_{\L^{\infty}}\lesssim \|f_0\|_{ \T^{p_0,q,r}_{\beta}}  + t\|f_1\|_{ \T^{\infty,q,r}_{\beta}}.
    \end{align*}
    Since this holds for every decomposition $f= f_0 + f_1 \in \T^{p_0,q,r}_{\beta} + \T^{\infty,q,r}_{\beta}$, we get
    \begin{align*}
        K(t,m^{q,r}_{\beta,\frac{1}{4}}(f),\L^{p_0} , \L^{\infty})\lesssim K(t,f,\T^{p_0,q,r}_{\beta} , \T^{\infty,q,r}_{\beta}),
    \end{align*}
    which proves the first direction.\\
    
    \noindent \textbf{Step 2: $ K(t,f,\T^{p_0,q,r}_{\beta} , \T^{\infty,q,r}_{\beta}) \lesssim K(t,m^{q,r}_{\beta,\frac{1}{4}}(f),\L^{p_0} , \L^{\infty})$.}
    Due to the equivalence of \eqref{eq: K E} and \eqref{eq: E K}, as well as the equivalence of the $K$ and $K_\infty$-functionals, it suffices to show that
    \begin{align}
    \label{eq: E ineq}
         E(Ct,f,\T^{p_0,q,r}_{\beta} , \T^{\infty,q,r}_{\beta}) \leq C E(t,m^{q,r}_{\beta,\frac{1}{4}}(f),\L^{p_0} , \L^{\infty})
    \end{align}
    for all $t>0$ and some $C>0$.
    We start by defining for every $Q\in\square$ and $t>0$ the subsets
    \begin{align}
    \label{def: Q_t}
        Q_t^+ \coloneqq Q\cap \{x\in\IR^d: m^{q,r}_{\beta,\frac{1}{4}}(f)(x)>t\} \quad \text{and} \quad Q_t^-\coloneqq Q\setminus Q_t^+.
    \end{align}
    Furthermore, we define
    \begin{align*}
        A^+_t\coloneqq \Big\{Q\in\square: |Q_t^+|>\frac{|Q|}{2}\Big\}\quad \text{and} \quad A^-_t\coloneqq \Big\{Q\in\square: |Q_t^-|\geq\frac{|Q|}{2}\Big\}.
    \end{align*}
    Observe that $\square = A^+_t \cup A^-_t$. Thus,
    \begin{align*}
        f_0\coloneqq \sum\limits_{Q\in  A^+_t} \mathbf{1}_{\bar Q}f \quad \text{and} \quad f_1\coloneqq \sum\limits_{Q\in  A^-_t} \mathbf{1}_{\bar Q}f
    \end{align*}
    satisfy $f=f_0 + f_1$. Next, we introduce the subsets
    \begin{align*}
        E^+_Q \coloneq \big\{x\in Q: G^{q,r}_{\beta,Q}(f)(x)\leq m^{q,r}_{\beta,\frac{1}{4}}(f)(x) \big\}\cap Q^+_t
    \end{align*}
    if $Q\in A_t^+$, and 
    \begin{align*}
        E^-_Q \coloneq \big\{x\in Q: G^{q,r}_{\beta,Q}(f)(x)\leq m^{q,r}_{\beta,\frac{1}{4}}(f)(x) \big\}\cap Q^-_t
    \end{align*}
    if $Q\in A_t^-$. By definition of $m^{q,r}_{\beta,\frac{1}{4}}(f)$ in \eqref{eq: def of m}, we observe that
    \begin{align*}
        \big| \big\{x\in Q: G^{q,r}_{\beta,Q}(f)(x)\leq m^{q,r}_{\beta,\frac{1}{4}}(f)(x) \big\} \big| &= |Q|- \big| \big\{x\in Q: G^{q,r}_{\beta,Q}(f)(x)> m^{q,r}_{\beta,\frac{1}{4}}(f)(x) \big\} \big|\\
        &\geq |Q|-\frac{1}{4}|Q|\\
        &\geq \frac{3}{4}|Q|.
    \end{align*}
    Since $|Q^\pm_t|\geq \frac{|Q|}{2}$, we conclude by a set-theoretical argument that
    \begin{align*}
        |E_Q^{\pm}| = \big|\big\{x\in Q: G^{q,r}_{\beta,Q}(f)(x)\leq m^{q,r}_{\beta,\frac{1}{4}}(f)(x) \big\}\cap Q^\pm_t \big|\geq \frac{|Q|}{4}.
    \end{align*}
    Thus, we can apply Lemma~\ref{lem: dyadic char with subset} to $E^+_Q$ in place of $E_Q$ to obtain
    \begin{align*}
        \|f_0\|_{\T^{p_0,q,r}_{\beta}}^{p_0} &\simeq  \int\limits_{\IR^d} \Big(\sum\limits_{Q\in\square} \mathbf{1}_{E^+_Q}(x)|Q|^{-\frac{\beta q}{d}} \|f_0\|_{\L^r\big(\bar Q , \frac{\d y \d s}{s^{d+1}}\big)}^q \Big)^\frac{p_0}{q}\,\d x.
    \end{align*}
    Using the definitions of $f_0$ and $Q_t^+$, with the support properties of Whitney boxes $\bar Q$, we get
    \begin{align}
    \label{eq: f0 bound}
        \|f_0\|_{\T^{p_0,q,r}_{\beta}}^{p_0}&\simeq \int\limits_{\IR^d} \Big(\sum\limits_{Q\in A^+_t} \mathbf{1}_{E^+_Q}(x)|Q|^{-\frac{\beta q}{d}} \|f\|_{\L^r\big(\bar Q , \frac{\d y \d s}{s^{d+1}}\big)}^q \Big)^\frac{p_0}{q}\,\d x \nonumber\\
        &= \int\limits_{\big\{x\in\IR^d: m^{q,r}_{\beta,\frac{1}{4}}(f)(x)>t \big\}} \Big(\sum\limits_{Q\in A^+_t} \mathbf{1}_{E^+_Q}(x)|Q|^{-\frac{\beta q}{d}} \|f\|_{\L^r\big(\bar Q , \frac{\d y \d s}{s^{d+1}}\big)}^q \Big)^\frac{p_0}{q}\,\d x.
    \end{align}
    Similarly, using Lemma~\ref{lem: dyadic char with subset} to $E^-_Q$ in place of $E_Q$, we derive
    \begin{align*}
        \|f_1\|_{\T^{\infty,q,r}_{\beta}} &\simeq \sup\limits_{P\in\square} \bigg(\fint\limits_{P}\sum\limits_{Q\subset P}  \mathbf{1}_{E^-_Q}(x)|Q|^{-\frac{\beta q}{d}} \|f_1\|_{\L^r\big(\bar Q , \frac{\d y \d s}{s^{d+1}}\big)}^q \,\d x \bigg)^\frac{1}{q}.
    \end{align*}Using the definitions of $f_1$ and $Q_t^-$, and the support properties of Whitney boxes $\bar Q$, we obtain
    \begin{align}
    \label{eq: f1 bound}
        \|f_1\|_{\T^{\infty,q,r}_{\beta}} &\simeq\sup\limits_{P\in\square} \bigg(\fint\limits_{P} \sum\limits_{\substack{Q\subset P^-_t \\ Q\in A_t^-}} \mathbf{1}_{E^{-}_Q}(x)|Q|^{-\frac{\beta q}{d}} \|f\|_{\L^r\big(\bar Q , \frac{\d y \d s}{s^{d+1}}\big)}^q \,\d x \bigg)^\frac{1}{q} \nonumber\nonumber \\
        &\leq\sup\limits_{P\in\square} \bigg(\frac{1}{|P|}\int\limits_{P_t^-} \sum\limits_{Q\in A_t^-} \mathbf{1}_{E^{-}_Q}(x)|Q|^{-\frac{\beta q}{d}} \|f\|_{\L^r\big(\bar Q , \frac{\d y \d s}{s^{d+1}}\big)}^q \,\d x \bigg)^\frac{1}{q},
    \end{align}
    where we used $E_Q^-\subset Q_t^-\subset P_t^-$ in the last step.
    Due to the observations \eqref{eq: equiv Gf = Gs} and \eqref{eq: equiv mf = ms} and the definitions of $E_Q^\pm$, we can utilize the inequality (5.10) of \cite{Frazier_Jawerth}, which reads in our situation as follows:
    \begin{align*}
        \Big(\sum\limits_{Q\in\square} \mathbf{1}_{E^\pm_Q}(x)|Q|^{-\frac{\beta q}{d}} \|f\|_{\L^r\big(\bar Q , \frac{\d y \d s}{s^{d+1}}\big)}^q \Big)^\frac{1}{q} \lesssim  m^{q,r}_{\beta,\frac{1}{4}}(f)(x)
    \end{align*}
    for every $x\in\IR^d$. Applying this estimate in \eqref{eq: f0 bound} and \eqref{eq: f1 bound} yields
    \begin{align*}
          \|f_0\|_{\T^{p_0,q,r}_{\beta}}^{p_0}\leq C_0\int\limits_{\big\{x\in\IR^d: m^{q,r}_{\beta,\frac{1}{4}}(f)(x)>t \big\}} |(m^{q,r}_{\beta}(f)(x)|^{p_0}\,\d x
    \end{align*}
    for some generic constant $C_0>0$, and
    \begin{align*}
        \|f_1\|_{\T^{\infty,q,r}_{\beta}} \leq C_1\sup\limits_{P\in\square} \bigg(\frac{1}{|P|}\int\limits_{P_t^-} | m^{q,r}_{\beta,\frac{1}{4}}(f)(x)|^q \,\d x \bigg)^\frac{1}{q}
        \leq C_1\sup\limits_{P\in\square} \bigg(\frac{1}{|P|}\int\limits_{P_t^-} t^q \,\d x\bigg)^\frac{1}{q}
        \leq C_1t,
    \end{align*}
    for some generic constant $C_1>0$, where we used the definition of $P_t^-$ from \eqref{def: Q_t} in the second step. Both estimates of $f_0$ and $f_1$ imply the following bound on the $E$-functional
    \begin{align}
    \label{eq: E int}
        E(C_1t, f, \T^{p_0,q,r}_{\beta} , \T^{\infty,q,r}_{\beta})\leq C_0 \bigg(\int\limits_{\{x\in\IR^d:  m^{q,r}_{\beta,\frac{1}{4}}(f)(x)>t\}} \big|m^{q,r}_{\beta,\frac{1}{4}}(f)(x) \big|^{p_0}\,\d x\bigg)^\frac{1}{p_0}.
    \end{align}
    Finally, we claim that the right-hand side of this inequality can be estimated by\\
    $ E(t,m^{q,r}_{\beta,\frac{1}{4}}(f),\L^{p_0} , \L^{\infty})$. Indeed, if $g\in \L^\infty(\IR^d)$ with $\|g\|_{\L^\infty}\leq  t/2$, then
    \begin{align*}
        \int\limits_{\{x\in\IR^d:  m^{q,r}_{\beta,\frac{1}{4}}(f)(x)>t\}} \big|m^{q,r}_{\beta,\frac{1}{4}}(f)(x) \big|^{p_0}\,\d x&=\int\limits_{\{x\in\IR^d:  m^{q,r}_{\beta,\frac{1}{4}}(f)(x)>t\}} \big|2m^{q,r}_{\beta,\frac{1}{4}}(f)(x) -m^{q,r}_{\beta,\frac{1}{4}}(f)(x) \big|^{p_0}\,\d x\\
        &\leq\int\limits_{\{x\in\IR^d:  m^{q,r}_{\beta,\frac{1}{4}}(f)(x)>t\}} \big|2m^{q,r}_{\beta,\frac{1}{4}}(f)(x)- 2g(x) \big|^{p_0}\,\d x\\
        &\leq 2^{p_0}\|m^{q,r}_{\beta,\frac{1}{4}}(f)- g\|^{p_0}_{\L^{p_0}}
    \end{align*}
    Since this holds for all $g\in \L^\infty(\IR^d)$ with $\|g\|_{\L^\infty}\leq  t/2$, we conclude
    \begin{align*}
         \bigg(\int\limits_{\{x\in\IR^d:  m^{q,r}_{\beta,\frac{1}{4}}(f)(x)>t\}} \big|m^{q,r}_{\beta,\frac{1}{4}}(f)(x) \big|^{p_0}\,\d x\bigg)^\frac{1}{p_0} &\leq 2 E\Big(\frac{t}{2},m^{q,r}_{\beta,\frac{1}{4}}(f),\L^{p_0} , \L^{\infty}\Big).
    \end{align*}
    Plugging this estimate into \eqref{eq: E int} yields 
    \begin{align*}
        E(C_1t, f, \T^{p_0,q,r}_{\beta} , \T^{\infty,q,r}_{\beta})\leq 2C_0 E\Big(\frac{t}{2},m^{q,r}_{\beta,\frac{1}{4}}(f),\L^{p_0} , \L^{\infty}\Big)
    \end{align*}
    for  all $t>0$. Finally, a change of variables $t=2t'$ and a case-by-case analysis of the resulting constants leads to the estimate
    \begin{align*}
        E(Ct', f, \T^{p_0,q,r}_{\beta} , \T^{\infty,q,r}_{\beta})\leq C E(t',m^{q,r}_{\beta,\frac{1}{4}}(f),\L^{p_0}\L^{\infty})
    \end{align*}
    for some generic constant $C>0$ and all $t'>0$. This proves the \eqref{eq: E ineq}, and thus, completes the proof.
\end{proof}

With Theorem~\ref{thm: intermediate result} at hand, we can reduce Theorem~\ref{thm: real interpolation in p} to the well-known real interpolation of $\L^p$-spaces combined with the reiteration theorem for real interpolation spaces.

\begin{proof}[Proof of Theorem~\ref{thm: real interpolation in p}]
    We first prove the theorem in the endpoint case $0<p_0<p_1=\infty$. Afterwards, we conclude by a reiteration argument.\\
    
    \noindent\textbf{Case 1: $0<p_0<p_1=\infty$}. Let $f\in \T^{p_0,q,r}_{\beta} + \T^{\infty,q,r}_{\beta} $. Then, by definition of real interpolation spaces \eqref{eq: def real int space} and Theorem~\ref{thm: intermediate result}, we have
    \begin{align*}
        \|f\|_{(\T^{p_0,q,r}_{\beta} , \T^{\infty,q,r}_{\beta} )_{\theta, p}} &= \bigg(\int\limits_0^\infty t^{-\theta q}  \big[K(t,f,\T^{p_0,q,r}_{\beta} , \T^{\infty,q,r}_{\beta})\big]^q \, \frac{\d t}{t}\bigg)^\frac{1}{q}\\
        &\simeq\bigg(\int\limits_0^\infty t^{-\theta q}  \big[K(t,m^{q,r}_{\beta,\frac{1}{4}}(f),\L^{p_0} , \L^{\infty})\big]^q \, \frac{\d t}{t}\bigg)^\frac{1}{q}\\
        &=\|m^{q,r}_{\beta,\frac{1}{4}}(f)\|_{(\L^{p_0} , \L^{\infty} )_{\theta, p}}.
    \end{align*}
    It is well-known that $(\L^{p_0} , \L^{\infty} )_{\theta, p}= \L^p$ if $\frac{1}{p} = \frac{1-\theta}{p_0}$. Hence,
    \begin{align*}
        \|f\|_{(\T^{p_0,q,r}_{\beta} , \T^{\infty,q,r}_{\beta} )_{\theta, p}} \simeq \|m^{q,r}_{\beta,\frac{1}{4}}(f)\|_{\L^p} \simeq \|f\|_{\T^{p,q,r}_\beta},
    \end{align*}
    where we used Proposition~\ref{prop: char with m} in the last step. \\

    \noindent \textbf{Case 2: $0<p_0<p_1<\infty$}. Let $0<\tilde p<p_0$ and $\theta_0,\theta_1\in(0,1)$ such that
    \begin{align*}
         \frac{1}{p_0} = \frac{1-\theta_0}{\tilde p} \quad \text{and} \quad \frac{1}{p_1} = \frac{1-\theta_1}{ \tilde p}.
    \end{align*}
    Then, Case~1, together with the reiteration theorem for quasi-Banach spaces (\cite[Thm.\@ 3.11.5]{Berg_Löfström}), imply
    \begin{align*}
        (\T^{p_0,q,r}_{\beta} , \T^{p_1,q,r}_{\beta} )_{\theta, p} &= \Big((\T^{\tilde p,q,r}_{\beta} , \T^{\infty,q,r}_{\beta} )_{\theta_0, p_0} , (\T^{\tilde p,q,r}_{\beta} , \T^{\infty,q,r}_{\beta} )_{\theta_1, p_1}\Big)_{\theta, p}\\
        &= (\T^{\tilde p,q,r}_{\beta} , \T^{\infty,q,r}_{\beta} )_{(1-\theta)\theta_0+ \theta \theta_1, p}.
    \end{align*}
    Observe that
    \begin{align*}
        \frac{1}{p} = \frac{1-\theta}{p_0} + \frac{\theta}{p_1} = \frac{(1-\theta)(1-\theta_0)}{\tilde p} + \frac{\theta(1-\theta_1)}{\tilde p} = \frac{1-[(1-\theta)\theta_0+ \theta \theta_1]}{\tilde p}.
    \end{align*}
    Hence, we conclude 
    \begin{align*}
         (\T^{p_0,q,r}_{\beta} , \T^{p_1,q,r}_{\beta} )_{\theta, p} = (\T^{\tilde p,q,r}_{\beta} , \T^{\infty,q,r}_{\beta} )_{(1-\theta)\theta_0+ \theta \theta_1, p} = \T^{p,q,r}_\beta,
    \end{align*}
    where we used Case~1 in the last step again.
\end{proof}

We finish this section by presenting a real interpolation result for $\T^{p,q,r}_\beta$-spaces in the parameter $q$. It is known from \cite[Prop.\@ 4.4]{Auscher_Bechtel_Haardt} that the $\Z^{p,q,r}_\beta$-spaces can be derived from real interpolation of $\T^{p,q,r}_\beta$-spaces. However, a corresponding characterization of the endpoint space $\Z^{\infty,q,r}_\beta$ was still missing. The subsequent statement provides this missing endpoint.

\begin{proposition}    
\label{prop: real interpolation tent spaces full range}
    Let $0<p\leq \infty$ and $0< q, q_0,q_1, r\leq \infty$. Furthermore, let $\beta_0,\beta_1\in\IR$ with $\beta_0\neq \beta_1$ and $\theta\in (0,1)$. Then, we have
    \begin{align*}
        (\T^{p,q_0,r}_{\beta_0} , \T^{p,q_1,r}_{\beta_1} )_{\theta, q} = \Z^{p,q,r}_\beta,
    \end{align*}
    where $\beta = (1-\theta)\beta_0 + \theta \beta_1$.
\end{proposition}

\begin{proof}
    By the real interpolation result \cite[Thm.\@ 4.1]{Auscher_Bechtel_Haardt} for $\Z$-spaces and the nesting property of Lemma~\ref{lem: embedding tent into Z}, we obtain
    \begin{align*}
         \Z^{p,q,r}_\beta &= (\Z^{p,\min\{p,q_0\},r}_{\beta_0} , \Z^{p,\min\{p,q_1\},r}_{\beta_1} )_{\theta, q}\\
         &\subset (\T^{p,q_0,r}_{\beta_0} , \T^{p,q_1,r}_{\beta_1} )_{\theta, q}\\
         &\subset  (\Z^{p,\max\{p,q_0\},r}_{\beta_0} , \Z^{p,\max\{p,q_1\},r}_{\beta_1} )_{\theta, q} =  \Z^{p,q,r}_\beta,
    \end{align*}
    which proves the claim.
\end{proof}

\end{document}